\documentclass[a4paper, 12pt]{scrartcl}

\usepackage[english]{babel}
\usepackage[T1]{fontenc}
\usepackage{amsmath}
\usepackage{amsthm}
\usepackage{amssymb}
\usepackage{chngcntr}
\usepackage{hyperref}
\usepackage{cleveref}
\usepackage{dsfont}
\usepackage{graphicx}
\usepackage{mathtools}
\usepackage{todonotes}

\counterwithin{equation}{section}
\counterwithin{figure}{section}

\newtheorem{theorem}{Theorem}
\newtheorem{proposition}[theorem]{Proposition}
\newtheorem{lemma}[theorem]{Lemma}
\newtheorem{remark}[theorem]{Remark}

\allowdisplaybreaks

\newcommand{\R}{\mathbb{R}}
\newcommand{\N}{\mathbb{N}}
\renewcommand{\d}{\mathrm{d}}
\renewcommand{\O}{{\mathcal{O}}}
\newcommand{\comp}{\mathrm{comp}}
\newcommand{\vol}{\mathrm{vol}}
\newcommand{\per}{\mathrm{per}}
\newcommand{\J}{\mathrm{J}}
\newcommand{\E}{\mathrm{E}}
\newcommand{\G}{\mathrm{G}}
\newcommand{\hd}{\mathcal{H}}
\newcommand{\TV}{\mathrm{TV}}
\newcommand{\sym}{{\mathrm{sym}}}
\newcommand{\ad}{{\mathrm{ad}}}
\renewcommand{\div}{\mathrm{div}}
\newcommand{\tr}{\mathrm{tr}}
\newcommand{\co}{\mathrm{co}}
\newcommand{\dist}{\mathrm{dist}}
\newcommand{\Id}{\mathds{1}}
\newcommand{\diag}{\mathrm{diag}}
\newcommand{\cell}{\mathrm{cell}}

\newcommand{\sigmain}{\sigma^\mathrm{cyl}}
\newcommand{\sigmarad}{\sigma^\mathrm{rad}}
\newcommand{\sigmashear}{\sigma^\mathrm{shear}}
\newcommand{\sigmasheartilde}{\tilde\sigma^\mathrm{shear}}
\newcommand{\opt}{\mathrm{opt}}

\newcommand{\notinclude}[1]{}

\begin{document}

\title{Optimal fine-scale structures in compliance minimization for a uniaxial load in three space dimensions}
\author{Jonas Potthoff \and Benedikt Wirth}
\date{}
\maketitle
\pagenumbering{arabic}

\begin{abstract}
We consider the shape and topology optimization problem to design a structure
that minimizes a weighted sum of material consumption and (linearly) elastic compliance under a fixed given boundary load.
As is well-known, this problem is in general not well-posed since its solution typically requires the use of infinitesimally fine microstructure.
Therefore we examine the effect of singularly perturbing the problem by adding the structure perimeter to the cost.
For a uniaxial and a shear load in two space dimensions, corresponding energy scaling laws were already derived in the literature.
This work now derives the scaling law for the case of a uniaxial load in three space dimensions,
which can be considered the simplest three-dimensional setting.
In essence, it is expected (and confirmed in this article) that for a uniaxial load the compliance behaves almost like the dissipation in a scalar flux problem
so that lower bounds from pattern analysis in superconductors can directly be applied.
The upper bounds though require nontrivial modifications of the constructions known from superconductors.
Those become necessary since in elasticity one has the additional constraint of torque balance.
\end{abstract}

\section{Introduction}
Using the concept of energy scaling laws, this article predicts the optimal (or rather an almost optimal) shape
of a three-dimensional structure under a uniaxial tension or compression load.
Optimality here is with respect to a cost that consists
of the compliance (a measure of structural weakness), the material consumption, and the perimeter or surface area (a measure of complexity, for instance during production of the structure).
Our work is essentially a continuation and extension of \cite{KoWi14} as well as \cite{ChKoOt04,ChCoKo08}:
The former solves the same problem in two space dimensions,
and the latter solves a closely related pattern analysis problem for intermediate states in type-I superconductors,
whose lower bounds and basic patterns can (up to minor modifications) directly be applied here.
The major difficulty in and contribution of the current work is to turn those basic patterns into feasible constructions for the compliance minimization setting.

\subsection{Motivation}

Elastic shape optimization or compliance minimization is a well-studied field
with lots of results since the 1960s on the minimum possible compliance for a given material volume and on the corresponding optimal microstructures,
see \cite{HaSh63,KoSt86,Al02,Mi02} and the references therein.
It is of interest not only for material design,
but also for getting some understanding of structures appearing in nature,
adopting the hypothesis that biological evolution actually solves an optimization problem.
A common example is the microstructure of bones (`spongiosa').
Since compliance minimization in general leads to non-natural, infinitely fine microstructure,
there must be some additional complexity-limiting mechanism involved.
Various aspects might potentially contribute to this complexity limitation
such as growth processes, energy consumption for remodelling and maintenance, etc.
In absence of any corresponding biological model we will consider surface area or perimeter as a measure of structural complexity,
since remodelling of biological structures typically happens at their surface.
Given the fine length scales of many biological structures (pore sizes in human trabecular bone range from tens to hundreds of micrometres)
the complexity limitation cannot be very strong.
Therefore it is natural to examine the regime of small perimeter penalization.
Unfortunately, the optimal structures in this regime are far too complex for a numerical resolution,
but they are amenable to asymptotic analysis, for instance in the form of scaling laws as will be proved here.

From a more mathematical viewpoint, elastic shape optimization or compliance minimization serves as a formidable model for the study of energy-driven pattern formation. Similarly to many other systems studied in the literature (such as thin sheets, micromagnetics, martensite models, or nonconvex versions of optimal transport)
it seems to exhibit several regimes of very distinct behaviour and a rich class of possible patterns. The task then is to quantify all parameter regimes and study the full phase diagrams of these systems,
and our work is in line with this general objective.
In the specific setting of compliance minimization
the diverse pattern types and their levels of complexity can be controlled by the applied boundary load.
This diversity is inherited from the problem without perimeter regularization, as becomes apparent already in two space dimensions:
If the macroscopic principal stresses have opposite sign, then the situation is known to be very rigid in that any optimal microgeometry must resemble a rank-2 laminate \cite{AlAu99}.
In contrast, if the macroscopic principal stresses have equal sign, then there is a high degree of freedom with a multitude of optimal microgeometries besides sequential laminates
(such as the ``confocal ellipse'' construction based on elliptic inclusions at all length scales \cite{GrKo95}
or the ``Vigdergauz construction'' based on more complicated inclusions at a single length scale \cite{GrKo95b}).
Sometimes, microstructure might not even be necessary: The optimal geometry for a unit disc domain under hydrostatic pressure is an annulus
(which is a special case of the confocal ellipse construction).
In three space dimensions it is expected that combinations of these properties may occur.
The setting of compliance minimization shares some features with patterns formed in thin sheets or martensites.
In the former, there exist regimes of simple hanging drapes with dyadic coarsening structures along essentially just one dimension \cite{VaPiBr+11},
but also regimes with complicated multidirectional patterns as in crumpling paper \cite{CoMa08}.
Likewise, in martensites plain dyadic branching patterns occur in the simplest situation \cite{KoMu94},
while in some shape memory alloys there seems to be enormous flexibility in the possible patterns \cite{RuTr22}.

In this article we study the borderline case between principal stresses of equal and different sign, the situation of just a single nonzero principal stress.
This represents the simplest possible situation and therefore should be studied first.
The reason is its resemblance to pattern formation problems of scalar conservation laws or equivalently of vector fields:
While in elasticity one in principal has conservation of the vector-valued momentum and thus has to work with tensor fields (the stress),
many other pattern formation problems just deal with vector fields.
Examples include micromagnetic structures in ferromagnets \cite{ChKo98,ChKoOt99} or the intermediate state of type-I superconductors \cite{ChKoOt04,ChCoKo08} (in both cases the field is the magnetic field).
Applying a uniaxial load now implies that mainly the momentum along only one direction (say the vertical, $z$-direction) is transmitted
so that the corresponding row of the stress tensor behaves just like the magnetic field in the above problems.

We will analyse the pattern formation problem by an energy scaling law (a common tool in the theory of pattern formation),
that is, we will quantify how the total cost scales in both the strength of the load and of the perimeter regularization.
To this end we have to prove upper bounds (via explicit geometric construction) and matching lower bounds.
For the latter it turns out we can almost completely adapt the lower bounds from \cite{ChKoOt99,ChKoOt04} for patterns in type-I superconductors.
The upper bounds are more challenging:
The basic geometric shapes manifesting in type-I superconductors can indeed be adapted as well,
but modifications have to be applied that account for the balance of torque in addition to linear momenta.
While this is relatively straightforward in two space dimensions \cite{KoWi14} it becomes considerably more difficult in three space dimensions.

\subsection{Problem formulation and main result}\label{sec:formulation}

We fix a design domain $\Omega=(0,\ell)^2\times(0,L)\subset\R^3$ with square base
and consider the optimization of the geometry and topology of a structure $\O\subset\Omega$.
The set $\O$ here represents the region occupied by an elastic material, while $\Omega\setminus\O$ is void.
On the boundary $\partial\Omega$ of the design domain we apply a uniaxial, vertical boundary load of strength $F\in\R$ to be supported by $\O$, see \cref{fig:setting} left.
In more detail, abbreviating
\begin{equation*}
\hat\sigma=Fe_3\otimes e_3
\quad\text{for the Euclidean basis vector }
e_3=(0\ 0\ 1)^T,
\end{equation*}
we apply the stress (force per area) $\hat\sigma n$ on all of $\partial\Omega$, where $n$ denotes the unit outward normal to $\partial\Omega$
(note that this stress is nonzero only at the top and bottom face $\Gamma_{\mathrm{t}}=(0,\ell)^2\times\{L\}$ and $\Gamma_{\mathrm{b}}=(0,\ell)^2\times\{0\}$ of $\Omega$).
Throughout we will employ linearized elasticity to describe equilibrium displacements and stresses,
which is appropriate for small displacements and strains as they occur for instance in trabecular bone.
Furthermore, for simplicity we will assume $\O$ to consist of a homogeneous, isotropic material with zero Poisson ratio
(a nonzero Poisson ratio is not expected to qualitatively change the result, see \cref{rem:Poisson}, but would introduce substantial additional notation).
This implies that the material response to stress is governed by a single material parameter, the shear modulus (or second Lam\'e parameter) $\mu>0$.

The cost function with respect to which we optimize the structure $\O$ is a weighted sum
of its so-called compliance $\comp^{\mu,F,\ell,L}(\O)$ under the applied load, its volume $\vol(\O)$, and its surface area or perimeter $\per_\Omega(\O)$,
\begin{equation}\label{eqn:cost}
\J^{\alpha,\beta,\varepsilon,\mu,F,\ell,L}(\O)=\alpha \, \comp^{\mu,F,\ell,L}(\O)+\beta \, \vol(\O)+\varepsilon \, \per_\Omega(\O),
\end{equation}
where $\alpha,\beta,\varepsilon>0$ are positive weights and the compliance will depend on the material and configuration parameters $\mu,F,\ell,L$.
The corresponding optimization problem is commonly known as compliance minimization with perimeter regularization
(in the literature the volume is sometimes constrained instead of being part of the cost function).
The compliance is the work performed by the applied load and can in linearized elasticity also be interpreted as the stored elastic energy.
If $\partial\O$ does not contain the top and bottom face of $\Omega$ and thus cannot support the load, one sets $\comp^{\mu,F,\ell,L}(\O)=\infty$.
Since in linearized elasticity the compliance is known to be invariant under a sign change of the load
we will without loss of generality assume a tensile load $F>0$ throughout.
The volume and perimeter serve as measures for material consumption or weight and structural complexity, respectively.
While for vanishing perimeter regularization $\varepsilon=0$ the cost $\J^{\alpha,\beta,\varepsilon,\mu,F,\ell,L}$ admits no minimizer
(instead, an infinitely fine microstructure, a so-called rank-1-laminate is optimal \cite{Al02})
optimal structures $\O$ are known to exist for $\varepsilon>0$ \cite{AmBu93}.

Note that by $\per_\Omega(\O)$ we mean the perimeter relative to the open domain $\Omega$ so that components of $\partial\O$ inside $\partial\Omega$ do not contribute to the cost.
The reason is that unlike the surface area of $\partial\O\cap\Omega$ in the domain interior,
the surface area of $\partial\O\cap\partial\Omega$ within the domain boundary is not really connected to structural complexity.
Nevertheless one can also consider the alternative cost
\begin{equation}\label{eqn:costWithBoundary}
\tilde\J^{\alpha,\beta,\varepsilon,\mu,F,\ell,L}(\O)=\alpha \, \comp^{\mu,F,\ell,L}(\O)+\beta \, \vol(\O)+\varepsilon \, \per_{\R^3}(\O)
\end{equation}
with $\per_{\R^3}(\O)$ the full perimeter of $\O$, viewed as a subset of $\R^3$.
In that case the perimeter contribution from $\partial\O\cap\partial\Omega$ will sometimes dominate the cost.
For the sake of completeness we will treat this setting alongside the one of actual interest.

\begin{figure}
	\includegraphics[scale=1]{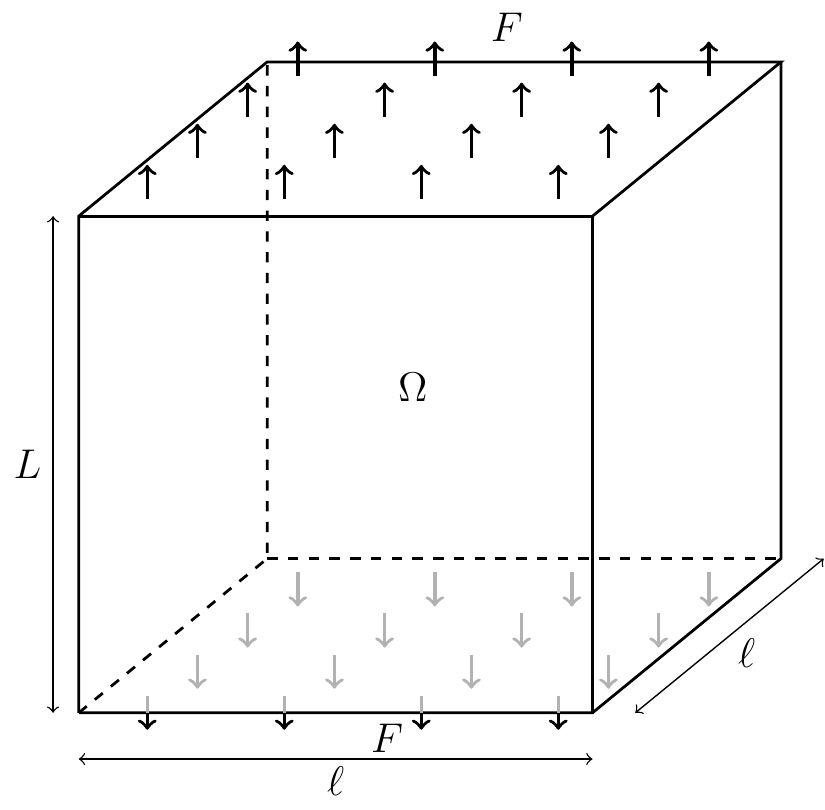} 
	\includegraphics[scale=0.29]{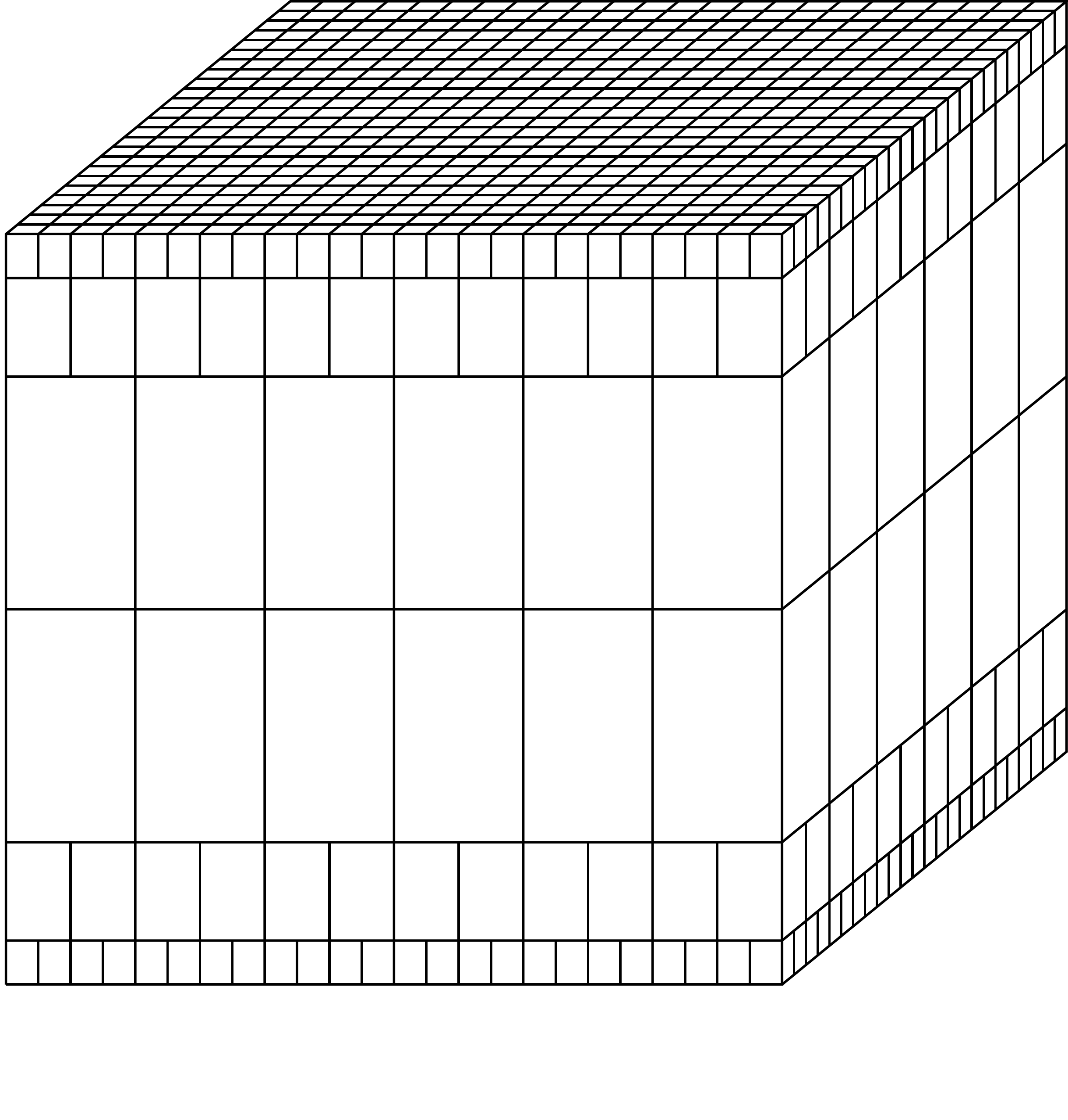}
\caption{Left: Sketch of the design domain $\Omega$ and the applied load $F$ to be supported by the structure $\O$.
Right: Sketch of basic construction scheme consisting of layers (here three layers above and below the midplane) of elementary cells, whose width halves from layer to layer.
}
\label{fig:setting}
\end{figure}

As mentioned previously, we are interested in the case of small perimeter regularization $\varepsilon\ll1$.
Furthermore we will assume $F\leq\sqrt{{4\mu\beta}/{\alpha}}$ since for larger $F$ the optimal structure is known to be $\O=\Omega$
(with or without perimeter regularization), see \cref{rem:largeForce} later.
We additionally require the domain to be sufficiently wide.
In that regime we prove the following energy scaling law.

\begin{theorem}[Energy scaling law for compliance minimization under a uniaxial load]\label{thm:scalingLawDimensional}
Assume $\bar F\vcentcolon=F\sqrt{\frac{\alpha}{4\mu\beta}}\leq1$ and $\varepsilon<\frac{\beta L}4$ as well as $\ell^3\geq\min\{L^3,\varepsilon L^2/\min\{\sqrt F,(1-F)^{3/2}|\log(1-F)|\}\}$.
There exist constants $c,C>0$ independent of $\alpha,\beta,\varepsilon,\mu,F,\ell,L$ such that
  \begin{equation*}
    c \ell^2f(\beta,\varepsilon,\bar F,L)\leq\min_{\O \, \subset \, \Omega} \, \J^{\alpha,\beta,\varepsilon,\mu,F,\ell,L}(\O)-\J^{\alpha,\beta,*,\mu,F,\ell,L}_0\leq C\ell^2  f(\beta,\varepsilon,\bar F,L)
  \end{equation*}
holds for
  \begin{equation*}
     f(\beta,\varepsilon,\bar F,L)
     =\begin{cases}
     \hfil\varepsilon & \text{if } \bar F\leq(\frac{\varepsilon}{\beta L})^{\frac{1}{2}} \\
     \hfil\beta^{\frac{1}{3}}\bar F^{\frac{2}{3}}\varepsilon^{\frac{2}{3}}L^{\frac{1}{3}} & \text{if } (\frac{\varepsilon}{\beta L})^{\frac{1}{2}}<\bar F\leq\frac{1}{2} \\
     \beta^{\frac{1}{3}}(1\!-\!\bar F)|\log(1\!-\!\bar F)|^{\frac{1}{3}}\varepsilon^{\frac{2}{3}}L^{\frac{1}{3}} & \text{if } \frac12<\bar F\text{ and } (\frac{\varepsilon}{\beta L})^{\frac{2}{3}}\leq \, (1\!-\!\bar F) \, |\log(1\!-\!\bar F)|^{-\frac{1}{3}} \\
     \hfil\beta(1\!-\!\bar F)^2L & \text{if } (1\!-\!\bar F) \, |\log(1\!-\!\bar F)|^{-\frac{1}{3}}<(\frac{\varepsilon}{\beta L})^{\frac{2}{3}}
     \end{cases}
  \end{equation*}
with $\J^{\alpha,\beta,*,\mu,F,\ell,L}_0=\inf_{\O\subset\Omega}\J^{\alpha,\beta,0,\mu,F,\ell,L}=L\ell^2F\sqrt{{\alpha\beta}/\mu}$.
If $\J^{\alpha,\beta,\varepsilon,\mu,F,\ell,L}$ is replaced with $\tilde\J^{\alpha,\beta,\varepsilon,\mu,F,\ell,L}$, then $f$ is replaced with
\begin{equation*}
\tilde f(\beta,\varepsilon,\bar F,L)
=\max\{\varepsilon,\sqrt{\bar F}\varepsilon L/\ell,f(\beta,\varepsilon,\bar F,L)\}.
%
\end{equation*}
\end{theorem}

Above, $\J^{\alpha,\beta,*,\mu,F,\ell,L}_0$ represents the infimum cost that has to be paid even without perimeter regularization.
A positive $\varepsilon$ produces an additional excess cost beyond $\J^{\alpha,\beta,*,\mu,F,\ell,L}_0$, whose optimal scaling we characterize in \cref{thm:scalingLawDimensional}.
This scaling depends on the relation between $\varepsilon$ and $F$ or equivalently $\bar F$ -- the \namecref{thm:scalingLawDimensional} shows four different regimes.
Almost the same regimes occur in the pattern analysis for type-I superconductors \cite{ChKoOt04,ChCoKo08}.
In each regime a different structure $\O$ leads to the optimal scaling,
and the main contribution of this work is the construction of these structures, thereby proving the upper bound in \cref{thm:scalingLawDimensional}
(the lower bound will essentially follow from \cite{ChCoKo08} except for the first regime and for energy $\tilde\J^{\alpha,\beta,\varepsilon,\mu,F,\ell,L}$).

Below we briefly summarize these regimes in order of increasing $\bar F$.
As explained before, their optimal patterns will resemble the constructions from \cite{ChKoOt04}.
The main part of these patterns will always be composed of elementary cells that all look the same and only differ in size and height-width ratio.
These elementary cells are organized side by side in single horizontal layers of identical cells,
where the layers at the centre contain the largest elementary cells and where from each layer to the next the cell width is halved
so that towards the top and bottom boundary $\Gamma_{\mathrm{t}}$ and $\Gamma_{\mathrm{b}}$ the cells become finer and finer (see \cref{fig:setting} right).
The reason for this ansatz is that near the boundaries the structures have to be very evenly distributed and thus very fine in order to withstand the boundary load,
while away from the load it pays off to coarsen since this way one can reduce surface area.
The elementary cells will be designed in such a way that each cell connects seamlessly to four cells of half the width
in order to be able to stack the different cell layers on top of each other.

Throughout we will use the notation $f\lesssim g$ for two expressions $f$ and $g$ to indicate that there is a universal constant $C>0$ such that $f\leq Cg$.
Likewise, $f\gtrsim g$ means $g\lesssim f$, and $f\sim g$ is short for $f\gtrsim g$ and $f\lesssim g$.

\paragraph{Extremely small force.}
In this regime the nondimensional force $\bar F$ is small compared to $\sqrt\varepsilon$,
\begin{equation*}
\bar F\lesssim(\tfrac{\varepsilon}{\beta L})^{\frac{1}{2}}.
\end{equation*}
(Note that in the definition of $f$ in \cref{thm:scalingLawDimensional} the regime of low force was actually defined with $\leq$ in place of $\lesssim$,
but of course there is a transition region between any two neighbouring regimes in which they have the same scaling and thus the constructions of both are valid.
Therefore the regime boundaries are in fact only specified up to a constant factor.)
In this regime the dominant energy contribution comes from the perimeter regularization near the top and bottom boundary $\Gamma_{\mathrm{t}}$ and $\Gamma_{\mathrm{b}}$:
Since $\Gamma_{\mathrm{t}}\cup\Gamma_{\mathrm{b}}$ must be contained in $\partial\O$,
but $\O$ may only have a tiny material consumption on any cross-section between the top and bottom boundary,
the shape $\O$ must exhibit a surface area of roughly $\hd^2(\Gamma_{\mathrm{t}}\cup\Gamma_{\mathrm{b}})$ (where $\hd^2$ denotes the two-dimensional Hausdorff measure)
for the transition from $\Gamma_{\mathrm{t}}$ and $\Gamma_{\mathrm{b}}$ to an almost empty cross-section nearby.
This causes the excess cost of order $\ell^2\varepsilon$.
The actual geometry of the force-transmitting structure in between $\Gamma_{\mathrm{t}}$ and $\Gamma_{\mathrm{b}}$ has substantially smaller cost and thus is not so important;
one can for instance construct it similarly to the regime of small forces, shown in \cref{fig:constructionOverview} left.
Furthermore, in this regime it does not matter whether $\per_\Omega$ or $\per_{\R^3}$ is used as perimeter regularization as the scaling is the same.

\paragraph{Small force.}
The regime of small forces can be identified with the range
\begin{equation*}
(\tfrac{\varepsilon}{\beta L})^{\frac{1}{2}}\lesssim\bar F\leq C
\end{equation*}
for an (arbitrary) positive constant $C<1$.
The corresponding construction with optimal scaling consists of a framework of struts as shown in \cref{fig:constructionOverview} left.
It is composed of elementary cells of different sizes, where each elementary cell contains eight struts, arranged along the edges of an imaginary pyramid.
Here the major cost contribution stems from the coarsest elementary cells near the midplane, for which it is important to find the right balance between compliance and perimeter.
The structure differs from the optimal pattern in type-I superconductors by additional cross trusses (corresponding to the base edges of the pyramid).
Without these all other struts would be bent towards each other, resulting in a huge compliance.
Such cross trusses were already required in the two-dimensional setting in \cite{KoWi14}.
The structure may be seen as approximating a laminate; just like for laminates, all struts have unit stress inside.
Here, too, there is no difference in the energy scaling whether $\per_\Omega$ or $\per_{\R^3}$ is used as perimeter regularization,
since the boundary contribution to the perimeter is negligible compared to the total excess cost.

  \begin{figure}
   	\includegraphics[scale=0.29]{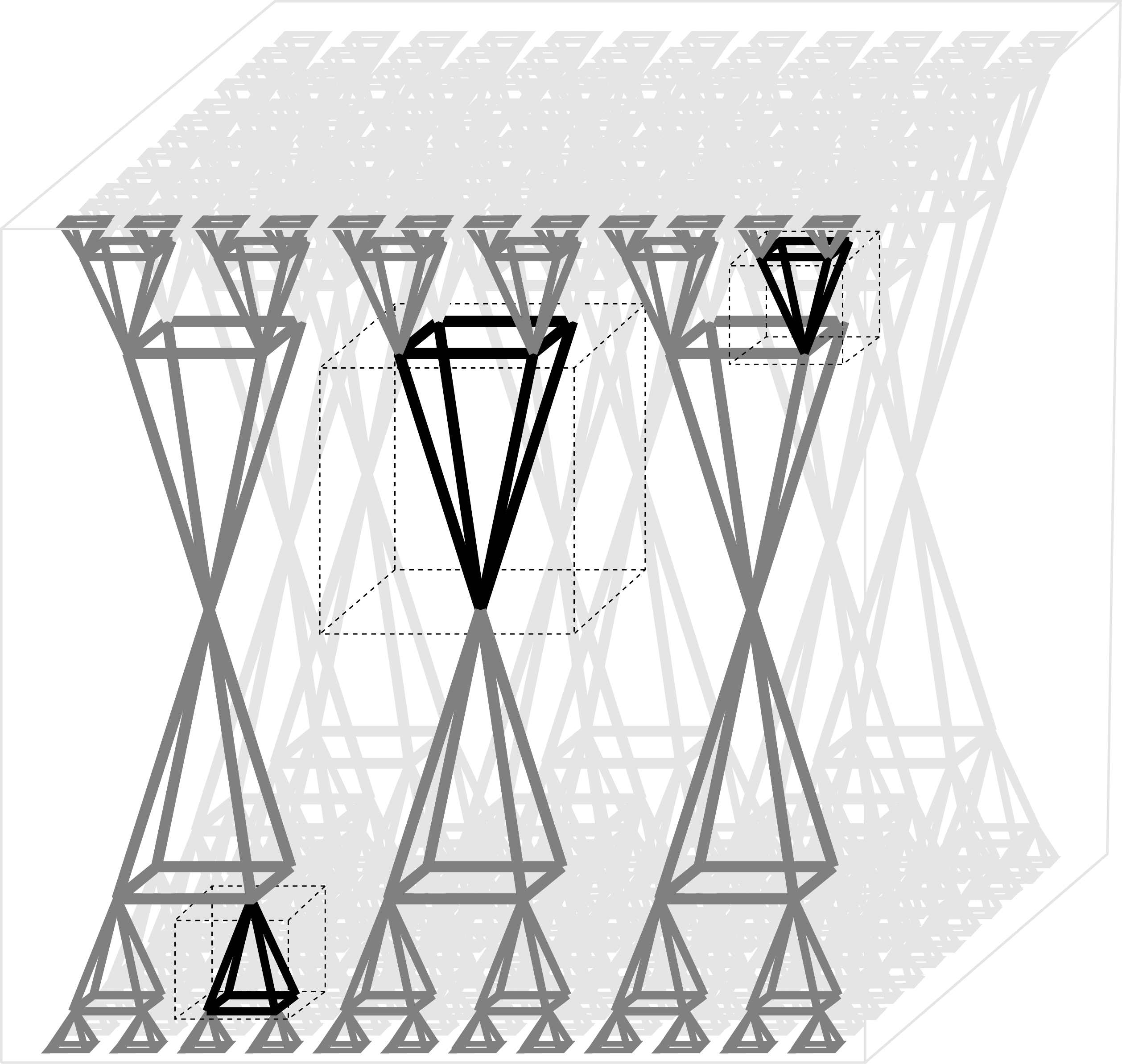} \hfill
   	\includegraphics[scale=0.305, angle=90]{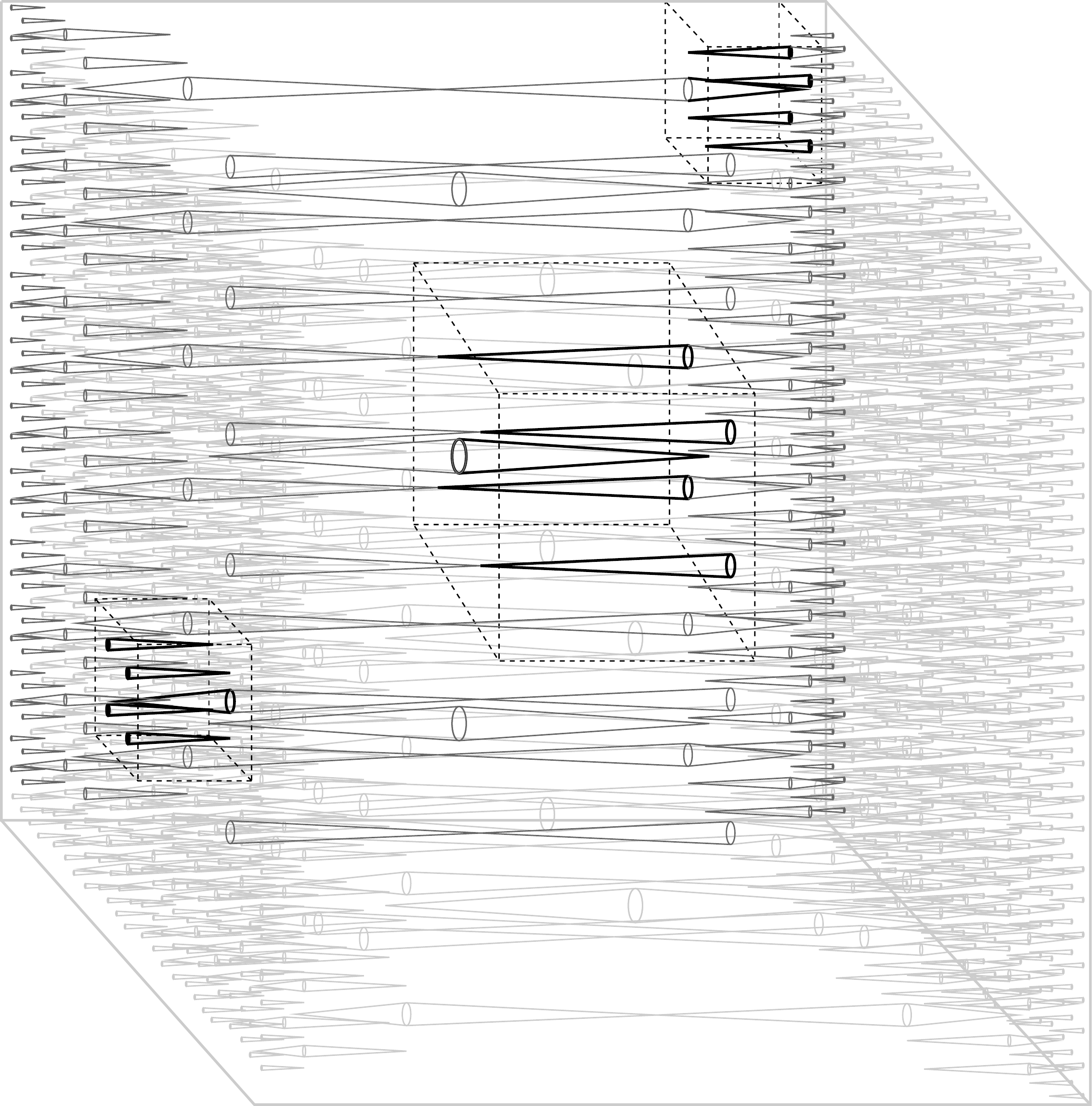}
	\caption[]{Illustration of constructions (with three layers above and below the midplane and three by three cells in the coarsest layer) for small and extremely small force (left) and large force (right). The latter one is a material block with many disjoint interior, cone-like voids. Three elementary cells are indicated with black colour.}
	\label{fig:constructionOverview}
  \end{figure}

\paragraph{Intermediate force.}
By intermediate forces we want to refer to the regime
\begin{equation*}
c\leq\bar F\leq C
\end{equation*}
for two (arbitrary) positive constants $0<c<C<1$.
This regime did not occur explicitly in \cref{thm:scalingLawDimensional}
but instead forms part of the regimes of small and of large forces, respectively.
Indeed, if $\bar F$ is bounded away from $0$ and from $1$, the scaling in both these regimes coincides.
Nevertheless we here mention the range of intermediate forces separately
since in this range the optimal scaling is in fact also obtained by a third construction different from the one for small or large forces:
One can consider the optimal pattern for compliance minimization in two space dimensions and constantly extend it along the third dimension, see \cref{fig:intermediateForces}.
This construction exhibits the same energy scaling independent of whether $\per_\Omega$ or $\per_{\R^3}$ is used.

  \begin{figure}
 	\includegraphics[scale=0.31]{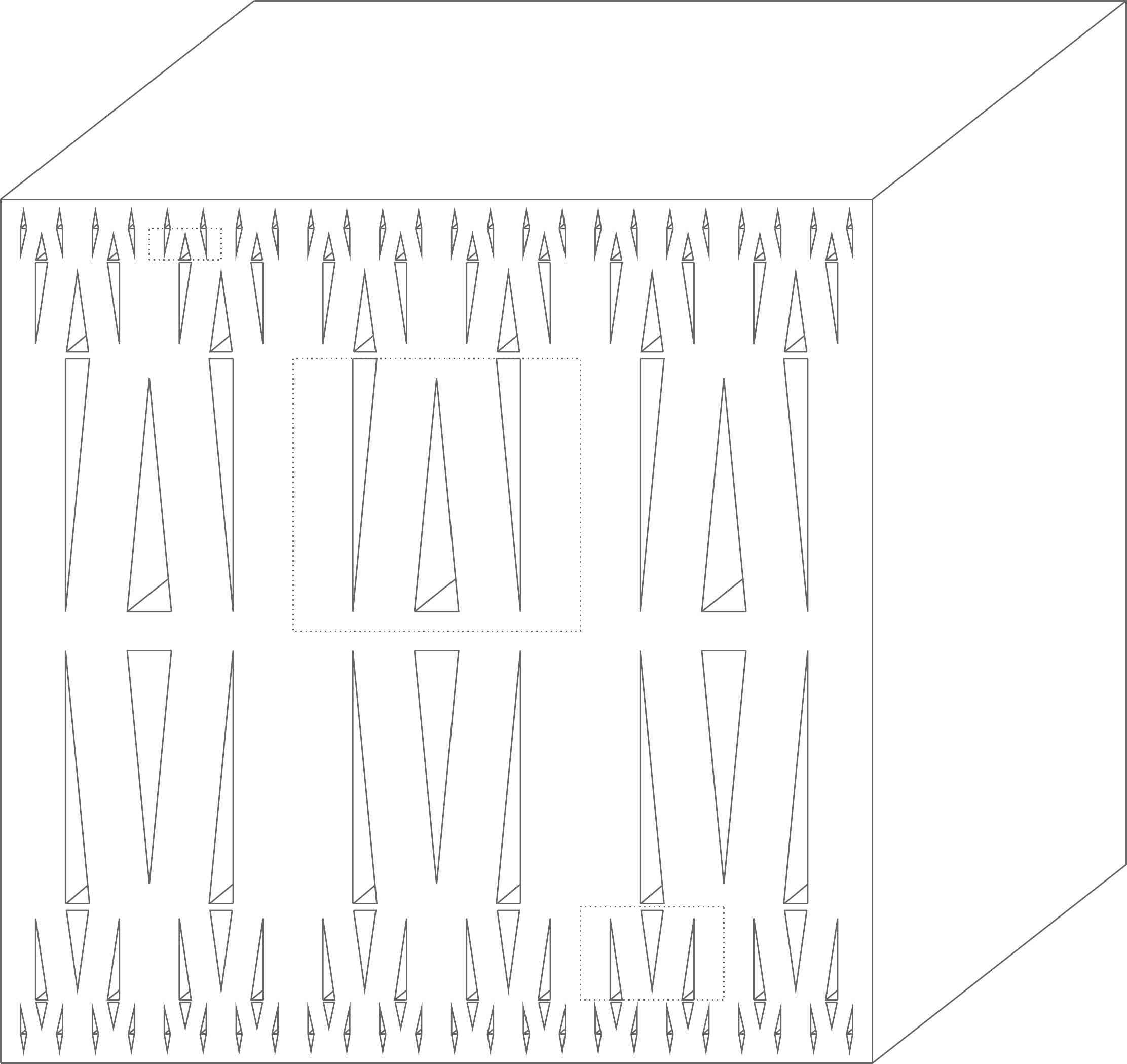}  \hfill
 	\includegraphics[scale=0.31]{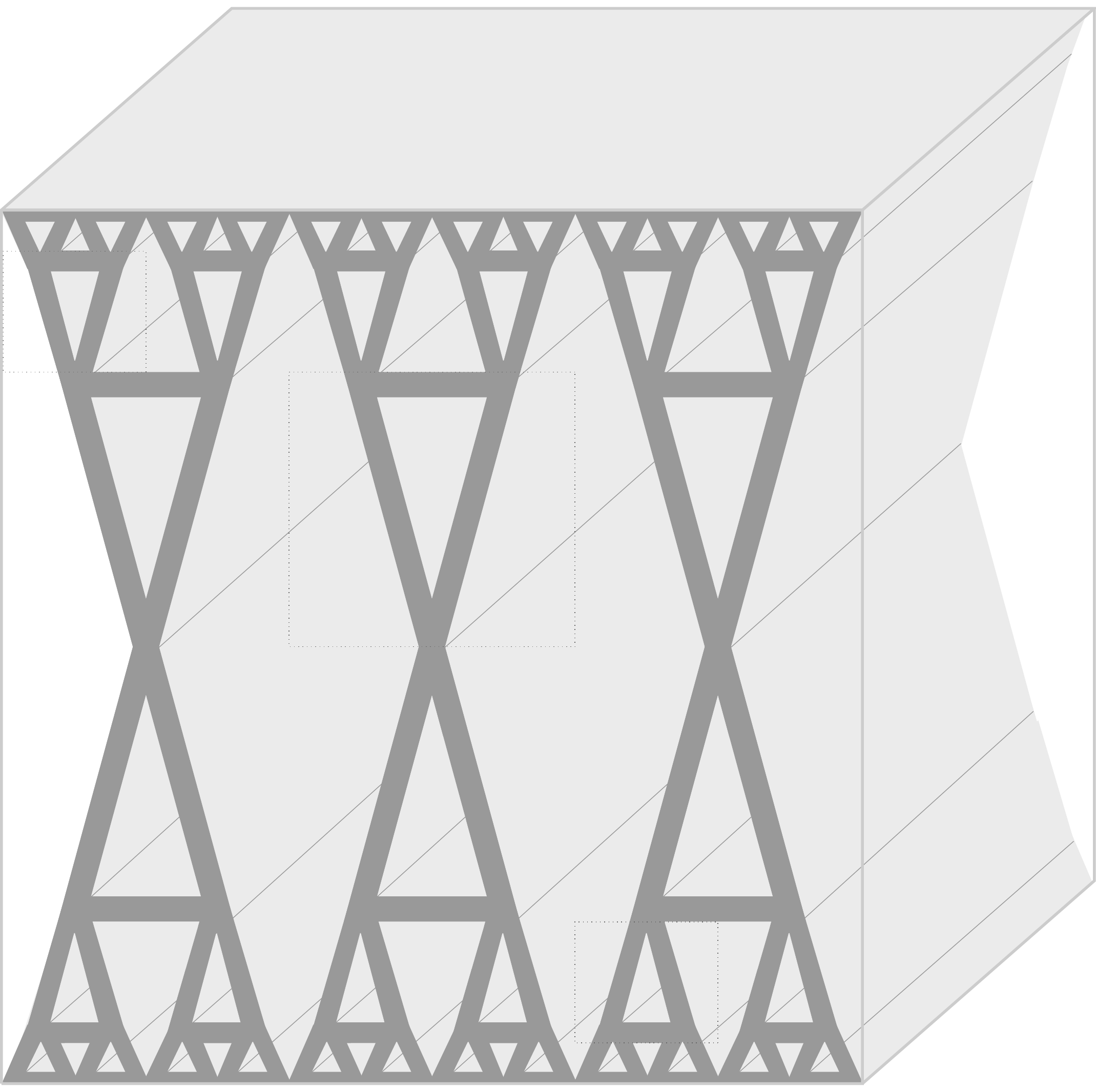} 
 	\caption{Two constructions with optimal scaling for intermediate forces, both an extension of optimal two-dimensional patterns (and thus composed of two-dimensional elementary cells). In each picture, three of these cells are indicated by dotted lines.
 	}
 	\label{fig:intermediateForces}
  \end{figure} 

\paragraph{Large force.}
This regime contains all forces with
\begin{equation*}
\bar F\geq C
\quad\text{and}\quad
(\tfrac{\varepsilon}{\beta L})^{\frac{2}{3}}\leq \, (1\!-\!\bar F) \, |\log(1\!-\!\bar F)|^{-\frac{1}{3}}
\end{equation*}
for some (arbitrary) positive constant $C<1$.
Those forces are relatively close to the maximum $\bar F=1$ from which on $\O=\Omega$ is known to be optimal,
but they still stay some distance (expressed in terms of $\varepsilon$) away from it.
The corresponding construction is again composed of elementary cells,
where each cell is a solid cuboid perforated by five thin cone-like voids, see \cref{fig:constructionOverview} right.
In this regime, for the first time, it can make a difference whether the perimeter is regularized with $\per_\Omega$ or $\per_{\R^3}$
since the maximum in the definition of $\tilde f$ in \cref{thm:scalingLawDimensional} may arise from any of the terms, depending on the size of $\bar F$ and $\ell$.
Indeed, if $\bar F$ gets too close to $1$, the perimeter contribution from $\partial\Omega$ dominates and masks the cost associated with the fine structure in the domain interior.
In the domain interior, however, again the coarsest elementary cells near the midplane contribute the major part of the cost.

\paragraph{Extremely large force.}
The regime of extremely large forces is characterized by
\begin{equation*}
\bar F\geq1
\qquad\text{or}\qquad
\bar F\leq1
\quad\text{and}\quad
(1\!-\!\bar F) \, |\log(1\!-\!\bar F)|^{-\frac{1}{3}}\lesssim(\tfrac{\varepsilon}{\beta L})^{\frac{2}{3}}.
\end{equation*}
In this regime the force is large enough such that the optimal geometry $\O$ is just a solid block of material $\O=\Omega$.
Any reduction in volume via material removal would be outweighed by the associated increase in compliance and perimeter.
If $\per_{\R^3}$ is used instead of $\per_\Omega$, the boundary cost dominates.

\begin{remark}[Two- and three-dimensional constructions]
Note that in the regime of intermediate forces both two-dimensional constructions (constantly extended along the third dimension) and three-dimensional constructions (the constructions from the regimes of small and of large forces) achieve the optimal energy scaling.
The same holds true in the regime of extremely small forces:
The constructions from \cref{fig:constructionOverview} left and \cref{fig:intermediateForces} right both achieve the scaling $\varepsilon$
(though the three-dimensional construction is expected to have a better constant):
The perimeter contribution from the top and bottom boundary simply dominates the excess cost so much that the construction in between becomes unimportant.
In the regime of extremely large forces even a zero-dimensional construction, constantly extended in all three dimensions, achieves the optimal energy scaling: the full material block.
However, in the regimes of small and large forces, respectively, only three-dimensional constructions can achieve the optimal energy scaling:
As shown in \cite{KoWi14} for small forces, the excess cost of any two-dimensional construction can at most scale like the third root in $\bar F$,
which is worse than the power $\frac23$ obtained here via a three-dimensional construction.
Likewise, via the same proof technique as in \cite{KoWi14} or via an optimal transport-based argument as in \cref{sec:lowerBoundSmallForce}
one can show that the excess cost of any two-dimensional construction can at most scale like the power $\frac23$ in $1-\bar F$
(we do not know whether this lower bound is sharp, though, cf.\ \cref{sec:intermediateForce}),
which is worse than the scaling $(1-\bar F)|\log(1-\bar F)|^{1/3}$ obtained by the three-dimensional construction.
\end{remark}

\begin{remark}[Nonzero Poisson ratio]\label{rem:Poisson}
Above we assumed a material with zero Poisson ratio.
While switching from zero to nonzero Poisson ratio changes the compliance and thus $\J^{\alpha,\beta,\varepsilon,\mu,F,\ell,L}(\O)$ by at most a constant factor,
this need not be true for the excess cost $\J^{\alpha,\beta,\varepsilon,\mu,F,\ell,L}(\O)-\J^{\alpha,\beta,*,\mu,F,\ell,L}_0$, which is the quantity we estimate.
However, all constructions are based on almost vertical strut-like structures bearing just uniaxial loads
(except maybe the construction from the large force regime),
whose excess cost is (in addition to the perimeter cost) mainly caused by their slight deviation from vertical alignment.
This contribution is actually independent of the Poisson ratio so that a nonzero Poisson ratio is not expected to change the energy scaling.
\end{remark}


\section{Model description and relation to superconductivity}

In this section we provide a detailed definition of the considered cost functional,
introduce related costs in models of type-I superconductors
and show that the latter provide lower bounds for our problem.

\subsection{Cost components}

To define compliance we first briefly recapitulate the basic framework of linearized elasticity.
Assume an elastic body $\O\subset\R^3$ with Lipschitz boundary is subjected to a boundary load $f:\partial\O\to\R^3$ (say in $L^2(\partial\O;\R^3)$).
The load represents a surface stress acting on the boundary.
As a result the body $\O$ deforms, and the equilibrium displacement $u:\O\to\R^3$ will be a minimizer of free energy
\begin{equation*}
\E(\tilde u)=\frac12\int_{\O}\mathbb{C}\epsilon(\tilde u):\epsilon(\tilde u)\,\d x-\int_{\partial\O}f\cdot\tilde u\,\d\hd^2,
\end{equation*}
where $\mathbb{C}$ is the fourth order elasticity tensor of the elastic material, $A:B=\tr(A^TB)$ denotes the Frobenius inner product between two matrices $A$ and $B$, and
\begin{equation*}
\epsilon(\tilde u)=\tfrac{1}{2}\left(\nabla\tilde u+\nabla\tilde u^T\right)
\end{equation*}
is the symmetrized gradient of the displacement (the skew-symmetric part corresponds to linearized rotations and thus does not contribute to the energy).
The first term of the free energy is the internally stored elastic energy, the second is the negative work performed by the load during the deformation. For the minimizer $u$ of the free energy, this work is also called compliance of $\O$ under the load $f$,
\begin{equation*}
\comp(\O)=\frac12\int_{\partial\O}f\cdot u\,\d\hd^2,
\end{equation*}
and it represents a quantitative measure of structural weakness.
In linearized elasticity there are multiple equivalent formulations of the compliance.
In particular, it is well-known (and essentially follows from convex duality, see for instance \cite{TeMi05}) that the compliance can also be expressed as
\begin{equation*}
\comp(\O)=\inf\left\{\int_\O\frac12\mathbb{C}^{-1}\sigma:\sigma\,\d x\,\middle|\,\sigma:\O\to\R^{3\times 3}_{\sym},\,\div\sigma=0\text{ in }\O,\,\sigma n=f\text{ on }\partial\O\right\}
\end{equation*}
for $\R^{3\times3}_{\sym}$ the real symmetric $3\times3$ matrices and $n$ the unit outward normal to $\partial\O$.
The minimizing $\sigma$ is the so-called stress tensor, describing the internal forces in equlibrium, and it is related to the equilibrium displacement $u$ via $\sigma=\mathbb{C}\epsilon(u)$.
As explained in \cref{sec:formulation}, we consider the simple elastic constitutive law of a homogeneous isotropic material with vanishing Poisson ratio and shear modulus $\mu$,
which means that the elasticity tensor is given by $\mathbb{C}\epsilon=2\mu\epsilon$.
Furthermore, in our case $\O\subset\Omega$ as well as $f=\hat\sigma n=(0\ 0\ F)^T$ on $\partial\Omega\supset\partial\O$ and $f=0$ else
so that we can extend any stress field $\sigma$ by zero to $\Omega\setminus\O$ and thus obtain a slightly simpler formulation of the compliance via
\begin{equation*}
\comp^{\mu,F,\ell,L}(\O)=\inf_{\sigma\in\Sigma_\ad^{F,\ell,L}(\O)}\int_\Omega\frac1{4\mu}|\sigma|^2\,\d x
\end{equation*}
for the admissible stress fields
\begin{equation*}
\Sigma_\ad^{F,\ell,L}(\O)=\{\sigma:\Omega\to\R^{3\times 3}_{\sym}\,|\,\div\sigma=0\text{ in }\Omega,\,\sigma n=\hat\sigma n\text{ on }\partial\Omega,\,\sigma=0\text{ on }\Omega\setminus\O\}.
\end{equation*}
This definition of the compliance can even be extended to arbitrary Borel sets $\O\subset\Omega$, so the condition of Lipschitz boundary can be dropped.

The second cost contribution, the material consumption or volume $\vol(\O)$ is nothing else than the Lebesgue measure of $\O$.

The third cost contribution, the perimeter, is defined as the total variation of the characteristic function $\chi_\O$ of $\O$,
\begin{multline*}
\per_A(\O)
=|\chi_\O|_{\TV(A)}
\vcentcolon=\sup\left\{\int_A\chi_\O\,\div\phi\,\d x\,\middle|\,\phi\in C_0^\infty(A;\R^3),\,|\phi|\leq1\text{ on }A\right\}\\
\text{for }A=\Omega\text{ or }A=\R^3.
\end{multline*}
Sets with $\per_{\R^3}(\O)<\infty$ are commonly called sets of finite perimeter,
and there is a measure-theoretic notion of boundary for them, the essential boundary $\partial^*\O$ (which coincides $\hd^2$-almost everywhere with $\partial\O$ for Lipschitz sets $\O$),
such that $\per_{\R^3}(\O)=\hd^2(\partial^*\O)$.


\subsection{Nondimensionalization}

In this paragraph we briefly nondimensionalize our cost $\J^{\alpha,\beta,\varepsilon,\mu,F,\ell,L}$ in order to reduce the number of parameters.
\begin{lemma}[Cost nondimensionalization]\label{thm:nondimensionalization}
The cost $\J^{\alpha,\beta,\varepsilon,\mu,F,\ell,L}$ from \eqref{eqn:cost} satisfies
  \begin{equation*}
    \J^{\alpha,\beta,\varepsilon,\mu,F,\ell,L}(L\O)=\beta L^3 \, \J^{1,1,\frac{\varepsilon}{\beta L},\frac{1}{4},\sqrt{\frac{\alpha}{4\mu\beta}}F,\frac{\ell}{L},1}(\O)
  \end{equation*}
for any $\O\subset(0,\frac{\ell}L)^2\times(0,1)$.
\end{lemma}
\begin{proof}
By the transformation rule for volume and surface integrals we have
  \begin{equation*}
     \vol(L\O)=L^3 \, \vol(\O) \quad \text{and} \quad \per_\Omega(L\O)=L^2 \, \per_\Omega(\O).
  \end{equation*}
As for the compliance, it is straightforward to check that $\sigma\in\Sigma_\ad^{F,\ell,L}(L\O)$ is equivalent to
\begin{equation*}
\bar\sigma
\in\Sigma_\ad^{\sqrt{\frac{\alpha}{4\mu\beta}}F,\frac{\ell}{L},1}(\O)
\qquad\text{for }
\bar\sigma(Lx)=\sqrt{\frac{\alpha}{4\mu\beta}}\sigma(x).
\end{equation*}
Thus we have
\begin{multline*}
\comp^{\mu,F,\ell,L}(L\O)
=\inf_{\sigma\in\Sigma_\ad^{F,\ell,L}(L\O)}\frac1{4\mu}\int_{\Omega^{\ell,L}}|\sigma|^2\,\d x
=\inf_{\bar\sigma\in\Sigma_\ad^{\sqrt{\frac{\alpha}{4\mu\beta}}F,\frac{\ell}{L},1}(\O)}\frac{\beta}{\alpha}\int_{\Omega^{\ell,L}}|\bar\sigma(Lx)|^2\,\d x\\
=\inf_{\bar\sigma\in\Sigma_\ad^{\sqrt{\frac{\alpha}{4\mu\beta}}F,\frac{\ell}{L},1}(\O)}\frac{\beta L^3}{\alpha}\int_{\Omega^{\frac\ell L,1}}|\bar\sigma(x)|^2\,\d x
=\frac{\beta L^3}{\alpha}\comp^{\frac14,\sqrt{\frac{\alpha}{4\mu\beta}}F,\frac{\ell}{L},1}(\O).
\end{multline*}
where we abbreviated $\Omega^{\ell,L}=(0,\ell)^2\times(0,L)$.
The combination of the above formulas yields the desired result.
\end{proof}

The analogous identity holds for $\tilde\J^{\alpha,\beta,\varepsilon,\mu,F,\ell,L}$ from \eqref{eqn:costWithBoundary}.
Consequently, without loss of generality we may set $\alpha=\beta=L=1$ and $\mu=\frac{1}{4}$ from now on;
the energy scaling for other parameter choices then follows from \cref{thm:nondimensionalization}.
We will therefore abbreviate $\J^{\varepsilon,F,\ell}=\J^{1,1,\varepsilon,\frac{1}{4},F,\ell,1}$ and $\tilde\J^{\varepsilon,F,\ell}=\tilde\J^{1,1,\varepsilon,\frac{1}{4},F,\ell,1}$ as well as $\comp^{F,\ell}=\comp^{\frac14,F,\ell,1}$.
Furthermore, instead of \cref{thm:scalingLawDimensional} we prove the following, from which \cref{thm:scalingLawDimensional} is a direct consequence.

\begin{theorem}[Nondimensional energy scaling law for compliance minimization under a uniaxial load]\label{thm:scaling}
Let $\Omega=(0,\ell)^2\times(0,1)$ and assume $F\leq1$ and $\varepsilon<\frac14$ as well as $\ell^3\geq\min\{1,\varepsilon/\min\{\sqrt F,(1-F)^{3/2}|\log(1-F)|\}\}$.
There exist constants $c,C>0$ such that
  \begin{equation*}
    c \ell^2f(\varepsilon,F)\leq\min_{\O \, \subset \, \Omega} \, \J^{\varepsilon,F,\ell}(\O)-\J^{*,F,\ell}_0\leq C\ell^2  f(\varepsilon,F)
  \end{equation*}
holds for
  \begin{equation*}
     f(\varepsilon,F)
     =\begin{cases}
     \hfil\varepsilon & \text{if } F\leq\varepsilon^{\frac{1}{2}} \\
     \hfil F^{\frac{2}{3}}\varepsilon^{\frac{2}{3}} & \text{if } \varepsilon^{\frac{1}{2}}<F\leq\frac{1}{2} \\
     (1\!-\!F)|\log(1\!-\!F)|^{\frac{1}{3}}\varepsilon^{\frac{2}{3}} & \text{if } \frac12<F\text{ and } \varepsilon^{\frac{2}{3}}\leq \, (1\!-\!F) \, |\log(1\!-\!F)|^{-\frac{1}{3}} \\
     \hfil(1\!-\!F)^2 & \text{if } (1\!-\!F) \, |\log(1\!-\!F)|^{-\frac{1}{3}}<\varepsilon^{\frac{2}{3}}
     \end{cases}
  \end{equation*}
with $\J^{*,F,\ell}_0=\inf_{\O\subset\Omega}\J^{0,F,\ell}=2\ell^2F$.
If $\J^{\varepsilon,F,\ell}$ is replaced with $\tilde\J^{\varepsilon,F,\ell}$, then $f$ is replaced with
\begin{equation*}
\tilde f(\varepsilon,F)
=\max\{\varepsilon,\sqrt{F}\varepsilon/\ell,f(\varepsilon,F)\}.
\end{equation*}
\end{theorem}

\begin{remark}[Optimal configuration for large forces]\label{rem:largeForce}
Note that for $F>1$ one obtains $\min_{\O \subset \Omega} \J^{\varepsilon,F,\ell}(\O)=\ell^2(1+|F|^2)$, which is readily checked to be achieved by the choice $\O=\Omega$.
Indeed, define $g(0)=0$ and $g(\sigma)=1+|\sigma|^2$ for $\sigma\neq0$, and denote its convexification by $g^{**}(\sigma)=1+|\sigma|^2$ if $|\sigma|>1$ and $g^{**}(\sigma)=2|\sigma|$ else.
Then using Jensen's inequality one has
\begin{multline*}
\J^{\varepsilon,F,\ell}(\O)
\geq\inf_{\sigma\in\Sigma_\ad^{F,\ell,1}(\O)}\int_\Omega|\sigma|^2\,\d x+\vol(\O)
\geq\inf_{\sigma\in\Sigma_\ad^{F,\ell,1}(\O)}\int_\Omega g(\sigma)\,\d x
\geq\inf_{\sigma\in\Sigma_\ad^{F,\ell,1}(\O)}\int_\Omega g^{**}(\sigma)\,\d x\\
\geq\inf_{\sigma\in\Sigma_\ad^{F,\ell,1}(\O)}\vol(\Omega)\,g^{**}\left(\frac1{\vol(\Omega)}\int_\Omega \sigma\,\d x\right)
=\vol(\Omega)\,g^{**}\left(\hat\sigma\right)
=\ell^2(1+|F|^2),
\end{multline*}
exploiting that the average value of any admissible stress field $\sigma$ must equal $\hat\sigma$ due to the divergence constraint.
\end{remark}

\subsection{A model for the intermediate state in type-I superconductors}

Here we briefly recapitulate the model from \cite{ChKoOt04,ChCoKo08} whose lower bounds will be applicable to our setting.
So-called type-I superconductors reveal an interesting intermediate state if exposed to a magnetic field.
They partition into two regions, the superconducting one, in which the magnetic field is suppressed,
and the normal one, where the magnetic field passes through the material without loss.
The resulting patterns were analysed via energy scaling laws in \cite{ChKoOt04} (upper bounds) and \cite{ChCoKo08} (matching lower bounds).
The authors consider an infinite material plate of thickness $L$ (we adapt the notation slightly to ours), and assuming periodicity they set
  \begin{equation*}
     \Omega=(0,\ell)^2\times(0,L) \quad \text{and} \quad \Omega^c=(0,\ell)^2\times[(-\infty,0]\cup[L,\infty)]
  \end{equation*}
to represent the material and its complement with periodic boundary conditions along the first two, inplane dimensions
(in fact they rescale lengths such that $\ell=1$).
The plate is exposed to a transverse magnetic field $b_ae_3$ (the `applied field') of strength $b_a>0$.
The magnetic field has been rescaled such that the critical field strength above which the material immediately loses its superconducting properties is $1$.
Consequently, the range of interest for the applied field strength is $0<b_a<1$.

Denoting the induced magnetic field by $B:\Omega\cup\Omega^c\to\R^3$ and the characteristic function of the superconducting region by $\chi:\Omega\to\{0,1\}$, both periodically extended along the first two dimensions, the free energy is given by
  \begin{equation*}
    \G^{\varepsilon,b_a,\ell,L}(B,\chi)=\int_{\Omega}\big(| B-b_ae_3|^2-\chi\big) \, \d x+\varepsilon|\chi|_{\TV(\Omega)}+\int_{\Omega^c}| B-b_ae_3|^2 \, \d x
  \end{equation*}
with the additional constraints
  \begin{equation*}
    \div B=0\text{ in }\R^3\quad\text{and}\quad B\chi=0\text{ in }\Omega
  \end{equation*}
due to Maxwell's equations and the magnetic field being suppressed in the superconducting region, respectively.
The first and last term express the magnetic energy due to the discrepancy between induced and applied magnetic field within and outside the material, respectively
(the occurrence of $\chi$ in the first term is motivated via a relaxation argument in \cite{ChKoOt04}).
The middle term measures the interfacial area between superconducting and normal regions via the total variation of $\chi$, weighted by $\varepsilon$,
representing a (small) magnetic energy loss from the transition between normal and superconducting regions.
Actually, in \cite{ChKoOt04,ChCoKo08} the total variation is taken with respect to periodic boundary conditions along the first two dimensions,
while we defined $|\chi|_{\TV(\Omega)}$ as the total variation relative to the open set $\Omega$ and thus ignore contributions from the periodic boundary.
Note that this slight change does not affect the energy scaling law proved in \cite{ChKoOt04,ChCoKo08} and cited below,
since none of their constructions shows any contact area between normal and superconducting regions within the periodic boundary and since their lower bound proofs also apply in this setting.

\begin{theorem}[Energy scaling in intermediate state of type-I superconductors, \cite{ChKoOt04,ChCoKo08}]\label{thm:superconductor}
Let $\Omega=(0,\ell)^2\times(0,L)$ and assume $b_a\leq1$, $\varepsilon<L/4$, and $\ell^3\geq\min\{L^3,\varepsilon L^2/\min\{\sqrt{b_a},(1-b_a)^{3/2}|\log(1-b_a)|\}\}$.
There exist constants $c,C>0$ such that
\begin{equation*}
c \, \ell^2f(\varepsilon,b_a,L) \, \leq\min_{\substack{\div B=0, \\ B\chi=0\text{ in }\Omega}}\G^{\varepsilon,b_a,\ell,L}(B,\chi)-\G^{*,b_a,\ell,L}_{0} \, \leq C \, \ell^2f(\varepsilon,b_a,L)
\end{equation*}
holds for
\begin{equation*}
f(\varepsilon,b_a,L)
=\begin{cases}
\hfil b_a \, \varepsilon^{\frac{4}{7}}L^{\frac{3}{7}} & \text{if } b_a\leq\left(\frac{\varepsilon}{L}\right)^{\frac{2}{7}} \\
\hfil b_a^{\frac{2}{3}}\varepsilon^{\frac{2}{3}}L^{\frac{1}{3}} & \text{if } \left(\frac{\varepsilon}{L}\right)^{\frac{2}{7}}<b_a\leq\frac{1}{2} \\ 
\left(1-b_a\right)|\log(1-b_a)|^{\frac{1}{3}}\varepsilon^{\frac{2}{3}}L^{\frac{1}{3}} & \text{if } \frac{1}{2}<b_a\text{ and }\left(\frac{\varepsilon}{L}\right)^{\frac{2}{3}}\leq \, (1-b_a) \, |\log(1-b_a)|^{-\frac{1}{3}} \\
\hfil\left(1-b_a\right)^2L & \text{if } (1-b_a) \, |\log(1-b_a)|^{-\frac{1}{3}}<\left(\frac{\varepsilon}{L}\right)^{\frac{2}{3}}
\end{cases}
\end{equation*}
where $\G^{*,b_a,\ell,L}_{0}=\inf_{\div B=0,B\chi=0\text{ in }\Omega}\G^{0,b_a,\ell,L}(B,\chi)=-\ell^2L(b_a-1)^2$.
\end{theorem}

In fact, the theorem is proved for $\ell=1$ in \cite{ChKoOt04,ChCoKo08}, but the above version immediately follows from the straightforward relation
\begin{equation*}
\G^{\varepsilon,b_a,\ell,L}(B,\chi)
=\G^{r\varepsilon,b_a,r\ell,rL}(\bar B,\bar\chi)/r^3
\qquad\text{with }
\bar B(x)=B(x/r)
\text{ and }
\bar\chi(x)=\chi(x/r)
\end{equation*}
for any $r>0$.
We will actually employ the result for $L=1$ since in our setting we think it more natural to nondimensionalize the domain height rather than its base area to $1$.
The base area then simply enters the energy scaling linearly, as one would expect.

\subsection{The relation between both models}
There are essentially three differences between our compliance minimization problem and the superconductor setting,
\begin{enumerate}
\item
the former has a matrix-valued state variable $\sigma$, the latter only a vector-valued variable $B$,
\item
$\sigma\cdot n$ is prescribed on the top and bottom domain boundary, while for $B\cdot n$ only the deviation from a preferred value is penalized via $\int_{\Omega^c}|B-b_ae_3|^2\,\d x$,
\item
$\sigma\cdot n$ is prescribed as zero on the sides of $\Omega$, while $B\cdot n$ is only required to be compatible with periodic boundary conditions.
\end{enumerate}
As a consequence, as already noted in \cite{KoWi14}, the superconductor energy minorizes the compliance minimization cost.

\begin{lemma}[Superconductor energy minorizes compliance minimization cost]\label{thm:minorization}
For $F=b_a$ we have
\begin{equation*}
\min_{\substack{\div B=0, \\ B\chi=0\text{ in }\Omega}}\G^{\varepsilon,b_a,\ell,1}(B,\chi)-\G^{*,b_a,\ell,1}_{0}
\leq\min_{\O \, \subset \, \Omega} \, \J^{\varepsilon,F,\ell}(\O)-\J^{*,F,\ell}_0.
\end{equation*}
\end{lemma}
\begin{proof}
Let $\O\subset\Omega$.
Pick an arbitrary $\delta>0$ and let $\sigma\in\Sigma_\ad^{F,\ell,1}(\O)$ such that $\comp^{F,\ell}(\O)\geq\int_\Omega|\sigma|^2\,\d x-\delta$.
Now set $\chi$ to be the characteristic function of $\Omega\setminus\O$ and $B$ to be the last row of $\sigma$, extended to $\Omega^c$ by $b_ae_3$.
Then by construction, $B$ is divergence-free with $\chi B=0$ in $\Omega$ and periodic boundary conditions along the first two dimensions.
Furthermore,
\begin{align*}
\J^{\varepsilon,F,\ell}(\O)-\J^{*,F,\ell}_0
&=\comp^{F,\ell}(\O)+\vol(\O)+\varepsilon\per_\Omega(\O)-2F\ell^2\\
&\geq\int_\Omega|\sigma|^2\,\d x-\delta+\int_\Omega1-\chi\,\d x+\varepsilon|\chi|_{\TV(\Omega)}-2F\ell^2\\
&\geq\int_\Omega|B|^2\,\d x-\delta+\int_\Omega1-\chi\,\d x+\varepsilon|\chi|_{\TV(\Omega)}-2F\ell^2\\
&=\int_\Omega|B|^2\,\d x-\delta+\int_\Omega1-\chi\,\d x+\varepsilon|\chi|_{\TV(\Omega)}+\int_{\Omega^c}|B-b_ae_3|^2\,\d x-2b_a\ell^2\\
\end{align*}
We now use that the total magnetic flux through any cross-section is conserved since $B$ is divergence-free.
Indeed, let $\Omega_{(0,t)}=(0,\ell)^2\times(0,t)$, then
\begin{multline*}
0=\int_{\Omega_{(0,t)}}\div B\,\d x
=\int_{\partial\Omega_{(0,t)}}B\cdot n\,\d\hd^2\\
=\int_{(0,\ell)^2\times\{t\}}B\cdot e_3\,\d\hd^2-\int_{(0,\ell)^2\times\{0\}}B\cdot e_3\,\d\hd^2
=\int_{(0,\ell)^2\times\{t\}}B\cdot e_3\,\d\hd^2-b_a\ell^2.
\end{multline*}
This implies $\int_\Omega B\cdot e_3\,\d x=\int_0^1\int_{(0,\ell)^2\times\{t\}}B\cdot e_3\,\d\hd^2\,\d t=b_a\ell^2$
and thus $\int_\Omega|B-b_ae_3|^2\,\d x=\int_\Omega|B|^2\,\d x-b_a^2\ell^2$ so that the above estimate can be continued with
\begin{align*}
\J^{\varepsilon,F,\ell}(\O)-\J^{*,F,\ell}_0
&\geq\int_\Omega(|B-b_ae_3|^2-\chi)\,\d x-\delta+\varepsilon|\chi|_{\TV(\Omega)}+\int_{\Omega^c}|B-b_ae_3|^2\,\d x+(b_a-1)^2\ell^2\\
&=\G^{\varepsilon,b_a,\ell,1}(B,\chi)-\G^{*,b_a,\ell,1}_{0}-\delta.
\end{align*}
The result now follows from the arbitrariness of $\delta>0$.
\end{proof}

Since $\tilde\J^{\varepsilon,F,\ell}\geq\J^{\varepsilon,F,\ell}$, the result of course also holds for $\tilde\J^{\varepsilon,F,\ell}$.


\section{Lower bounds}

In this section we prove the lower bounds of \cref{thm:scaling}.
Due to \cref{thm:minorization,thm:superconductor} only three lower bounds remain to be shown,
\begin{enumerate}
\item
the bound $\tilde\J^{\varepsilon,F,\ell}(\O)-\J^{*,F,\ell}_0\gtrsim\max\{\varepsilon\ell^2,\sqrt F\varepsilon\ell\}$
for the case when the perimeter regularization is performed with $\per_{\R^3}$ rather than $\per_\Omega$,
\item
the bound $\J^{\varepsilon,F,\ell}(\O)-\J^{*,F,\ell}_0\gtrsim\varepsilon\ell^2$ in the regime $F\leq\sqrt\varepsilon$, and
\item
the bound $\J^{\varepsilon,F,\ell}(\O)-\J^{*,F,\ell}_0\gtrsim F^{\frac23}\varepsilon^{\frac23}\ell^2$ in the regime $\varepsilon^{\frac12}\leq F<\varepsilon^{\frac27}$
(while for larger $F$ this bound is already implied by \cref{thm:minorization,thm:superconductor}).
\end{enumerate}
The following sections provide the corresponding estimates.
Before, let us briefly introduce some notation.
For $t\in[0,1]$ we will abbreviate
\begin{equation*}
\Omega_{(0,t)}=(0,\ell)^2\times(0,t),
\qquad
\Omega_t=(0,\ell)^2\times\{t\},
\qquad\text{and}\qquad
\O_t=\O\cap\Omega_t.
\end{equation*}
We will use that any stress $\sigma\in\Sigma_\ad^{F,\ell,1}(\O)$ has the same average vertical tension in all cross-sections.
Indeed, for almost all $t\in(0,1)$ we have
\begin{multline*}
0=\int_{\Omega_{(0,t)}}\div\sigma\,\d x
=\int_{\partial\Omega_{(0,t)}}\sigma n\,\d\hd^2\\
=\int_{\Omega_t}\sigma e_3\,\d\hd^2-\int_{\Omega_0}\sigma e_3\,\d\hd^2
=\int_{\Omega_t}\sigma e_3\,\d\hd^2-F\ell^2e_3,
\end{multline*}
which implies
\begin{equation*}
\int_{\O_t}\sigma e_3\,\d\hd^2=F\ell^2e_3.
\end{equation*}
Using Jensen's inequality this implies a bound on the compliance in almost all cross-sections,
\begin{equation}\label{eqn:minCompliance}
\int_{\Omega_t}|\sigma|^2\,\d\hd^2
\geq\int_{\O_t}|e_3^T\sigma e_3|^2\,\d\hd^2
\geq\frac1{\hd^2(\O_t)}\left|\int_{\O_t}e_3^T\sigma e_3\,\d x\right|^2
=\frac{F^2\ell^4}{\hd^2(\O_t)}.
\end{equation}

\subsection{Lower bound on exterior perimeter contribution}
Here we estimate the cost contribution from the surface area of the optimal geometry $\O$ within $\partial\Omega$.
Since necessarily $\partial\O$ is a subset of the top and bottom face $\Gamma_{\mathrm{t}}$ and $\Gamma_{\mathrm{b}}$ of $\Omega$ (as otherwise the compliance is infinite),
we automatically have
\begin{equation*}
\tilde\J^{\varepsilon,F,\ell}(\O)-\J^{*,F,\ell}_0
\geq\varepsilon\per_{\R^3}(\O)
\geq2\varepsilon\ell^2.
\end{equation*}
It thus remains to estimate the cost contribution related to the perimeter at the sides of the domain.
To this end we use the following elementary result.

\begin{lemma}[Polynomial estimate]\label{thm:polynomialEstimate}
For any $y,a>0$ we have
\begin{equation*}
\left(\frac1y-y\right)^2+ay\geq \frac12\min\{a,a^{2/3}\}.
\end{equation*}
\end{lemma}
\begin{proof}
Abbreviate $f(y)=(1/y-y)^2+ay$, then for any $y\in(0,\frac12)$ we have
\begin{equation*}
f(y)
\geq\frac12\left(\frac1y\right)^2+ay
\geq\min_{z>0}\frac1{2z^2}+az
=\frac32a^{2/3},
\end{equation*}
where the minimizer is given by $z=a^{-1/3}$.
Furthermore, for any $y\geq1$ we have
\begin{equation*}
f(y)
\geq f(1)
=a.
\end{equation*}
Finally, for $y\in[\frac12,1)$ we have
\begin{equation*}
f(y)
\geq\frac12(y-1)^2+ay.
\end{equation*}
This expression is minimized by $y=1-a$ so that
\begin{equation*}
f(y)
\geq\left.\begin{cases}
a-\frac{a^2}2&\text{if }a<\frac12,\\
\frac18+\frac a2&\text{else}
\end{cases}
\right\}\geq\frac a2.
\qedhere
\end{equation*}
\end{proof}

The desired estimate now follows essentially from the isoperimetric inequality.

\begin{proposition}[Exterior perimeter estimate]
For any $\O\subset\Omega=(0,\ell)^2\times(0,1)$ we have
\begin{equation*}
\tilde\J^{\varepsilon,F,\ell}(\O)-\J^{*,F,\ell}_0\gtrsim\min\left\{\sqrt F\varepsilon\ell,(F\varepsilon\ell^2)^{2/3}\right\}.
\end{equation*}
\end{proposition}
\begin{proof}
Let $A_t=\hd^2(\O_t)$ denote the area of $\O_t$.
The isoperimetric inequality implies
\begin{equation*}
\per_{\R^2}(\O_t)
\geq2\sqrt{\pi A_t},
\end{equation*}
where $\per_{\R^2}$ indicates the perimeter of the two-dimensional set $\O_t$ (projected into $\R^2$).
Now by Fubini's theorem, \eqref{eqn:minCompliance} and \cref{thm:polynomialEstimate} we have
\begin{align*}
\tilde\J^{\varepsilon,F,\ell}(\O)-\J^{*,F,\ell}_0
&\geq\inf_{\sigma\in\Sigma_\ad^{F,\ell,1}(\O)}\int_0^1\int_{\Omega_t}|\sigma|^2\,\d x+A_t+\varepsilon\per_{\R^2}(\O_t)-2F\ell^2\,\d t\\
&\geq\int_0^1\frac{F^2\ell^4}{A_t}+A_t+2\varepsilon\sqrt{\pi A_t}-2F\ell^2\,\d t\\
&=F\ell^2\int_0^1\left(\frac1{y(t)}-y(t)\right)^2+\frac{2\sqrt\pi\varepsilon}{\sqrt{F\ell^2}}y(t)\,\d t\\
&\geq F\ell^2\int_0^1\frac12\min\left\{\frac{\varepsilon}{\sqrt{F\ell^2}},\left(\frac{\varepsilon}{\sqrt{F\ell^2}}\right)^{2/3}\right\}\,\d t\\
&\gtrsim\min\left\{\sqrt F\varepsilon\ell,(F\varepsilon\ell^2)^{2/3}\right\}
\end{align*}
where we abbreviated $y(t)=\sqrt{A_t/(F\ell^2)}$.
\end{proof}

Summarizing, in this section we have shown
\begin{equation*}
\tilde\J^{\varepsilon,F,\ell}(\O)-\J^{*,F,\ell}_0
\gtrsim\max\left\{\varepsilon\ell^2,\min\left\{\sqrt F\varepsilon\ell,(F\varepsilon\ell^2)^{2/3}\right\}\right\}.
\end{equation*}
Now assume the right-hand side equals $(F\varepsilon\ell^2)^{2/3}$ and not $\sqrt F\varepsilon\ell$,
then consequently $\sqrt F\varepsilon\ell>(F\varepsilon\ell^2)^{2/3}\geq\varepsilon\ell^2$ and thus $\ell<\sqrt F$.
For $F\leq1$ and $\ell^3\geq\min\{1,\varepsilon/\sqrt F\}$ as in \cref{thm:scaling} this implies $\ell<1$ and thus $\ell>\ell^3\geq\varepsilon/\sqrt F$.
However, this implies $\sqrt F\varepsilon\ell<(F\varepsilon\ell^2)^{2/3}$, a contradiction.
Therefore we have actually shown
\begin{equation*}
\tilde\J^{\varepsilon,F,\ell}(\O)-\J^{*,F,\ell}_0
\gtrsim\max\left\{\varepsilon\ell^2,\sqrt F\varepsilon\ell\right\},
\end{equation*}
as desired.

\subsection{Lower bound for extremely small force}

Here we aim to exploit that for small forces almost no material may be used so that most cross-sections are almost empty.
The transition to the boundary then produces substantial perimeter cost.

\begin{proposition}[Interior perimeter estimate]
Let $\varepsilon<\frac1{8}$ and $F<\frac14$. For any $\O\subset\Omega=(0,\ell)^2\times(0,1)$ we have
\begin{equation*}
\J^{\varepsilon,F,\ell}(\O)-\J^{*,F,\ell}_0\geq\varepsilon\ell^2.
\end{equation*}
\end{proposition}
\begin{proof}
There are two cases, depending on whether $A_t \geq l^2/2$ for almost every cross-section $t\in(0,1)$ or not.
In the former case, by Fubini's theorem and \eqref{eqn:minCompliance} we have
\begin{multline*}
\J^{\varepsilon,F,\ell}(\O)-\J^{*,F,\ell}_0
\geq\inf_{\sigma\in\Sigma_\ad^{F,\ell,1}(\O)}\int_0^1\int_{\Omega_t}|\sigma|^2\,\d x+A_t-2F\ell^2\,\d t\\
\geq\int_0^1\frac{F^2\ell^4}{A_t}+A_t-2F\ell^2\,\d t
\geq\int_0^1\frac{2F^2\ell^4}{\ell^2}+\frac{\ell^2}2-2F\ell^2\,\d t
=2\ell^2(F-\tfrac12)^2
>\varepsilon\ell^2,
\end{multline*}
as desired. In the latter case let $t\in(0,1)$ be \notinclude{such }a cross-section with $A_t\leq\ell^2/2$.
Then
\begin{equation*}
\J^{\varepsilon,F,\ell}(\O)-\J^{*,F,\ell}_0
\geq\varepsilon\per_{\Omega}(\O)
\geq\varepsilon\int_{(0,\ell)^2}|\chi_\O(x_1,x_2,\cdot)|_{\TV((0,1))}\,\d(x_1,x_2)
\end{equation*}
for $\chi_\O$ the characteristic function of $\O$ and $|\cdot|_{\TV((0,1))}$ the total variation seminorm of a function on $(0,1)$.
Due to
\begin{align*}
|\chi_\O(x_1,x_2,\cdot)|_{\TV((0,1))}
&\geq|\chi_\O(x_1,x_2,0)-\chi_\O(x_1,x_2,t)|+|\chi_\O(x_1,x_2,t)-\chi_\O(x_1,x_2,1)|\\
&\geq\chi_\O(x_1,x_2,1)+\chi_\O(x_1,x_2,0)-2\chi_\O(x_1,x_2,t)
\end{align*}
we obtain
\begin{multline*}
\J^{\varepsilon,F,\ell}(\O)-\J^{*,F,\ell}_0
\geq\varepsilon\int_{(0,\ell)^2}\chi_\O(x_1,x_2,1)+\chi_\O(x_1,x_2,0)-2\chi_\O(x_1,x_2,t)\,\d(x_1,x_2)\\
=2\varepsilon(\ell^2-A_t)
\geq\varepsilon\ell^2
\end{multline*}
as desired, where $\chi_\O(\cdot,\cdot,s)$ with $s=0,1$ is the trace of the function of bounded variation $\chi_\O$
and equals $1$ whenever $\comp^{F,\ell}(\O)<\infty$.
\end{proof}


\subsection{Lower bound for small force}\label{sec:lowerBoundSmallForce}
Here we show that the lower bound in the regime of small forces actually extends to $\varepsilon^{1/2}\lesssim F$
(\cref{thm:minorization,thm:superconductor} only imply that it holds for $\varepsilon^{2/7}\lesssim F$).
In fact, we can simply repeat the proof of \cite[Thm.\,4.3]{ChCoKo08} (the lower bound for the superconductor energy under a small applied field). We only have to modify the last inequalities of that proof, which actually simplify in our setting.
For the sake of completeness we recapitulate the proof below.
We slightly changed its structure to provide the reader with a better intuition of why the proof is actually quite direct.

\begin{proposition}[Lower bound for small force, \protect{\cite[Thm.\,4.3]{ChCoKo08}}]
Let $\varepsilon^{1/2}<F\leq\frac12$. For any $\O\subset\Omega=(0,\ell)^2\times(0,1)$ we have
\begin{equation*}
\J^{\varepsilon,F,\ell}(\O)-\J^{*,F,\ell}_0\gtrsim F^{2/3}\varepsilon^{2/3}\ell^2.
\end{equation*}
\end{proposition}
\begin{proof}
\notinclude{
First note that it suffices to prove the statement for $F\leq\bar c$ with $\bar c>0$ an arbitrary but fixed constant.
The result for $\bar c<F\leq\frac12$ then follows from \cref{thm:minorization,thm:superconductor}
(in the case $\varepsilon^{2/7}<F$ there is nothing left to show,
while in the case $\varepsilon^{2/7}\geq F$ both $\varepsilon$ and $F$ are bounded away from $0$ so that all scalings coincide up to constant factors).
}
One can argue by contradiction, so assume
\begin{equation*}
\J^{\varepsilon,F,\ell}(\O)-\J^{*,F,\ell}_0
<c^*F^{2/3}\varepsilon^{2/3}\ell^2
\end{equation*}
for some $\O$ and some constant $c^*\in(0,\frac12)$ to be specified later.
Furthermore, let $\sigma\in\Sigma_\ad^{F,\ell,L}(\O)$ be such that $\comp^{F,\ell}(\O)=\int_\Omega|\sigma|^2\,\d x-\delta$ with $\delta$ small enough
so that also
\begin{equation*}
\Delta\J\vcentcolon=
\int_\Omega|\sigma|^2\,\d x+\vol(\O)+\varepsilon\per_\Omega(\O)-2 F\ell^2
<c^*F^{2/3}\varepsilon^{2/3}\ell^2.
\end{equation*}
For $t\in(0,1)$ we will abbreviate $A_t\vcentcolon=\hd^2(\O_t)$ and $P_t\vcentcolon=\per_{(0,\ell)^2}(\O_t)$.
The argument is performed in three steps.

\emph{Step\,1.}
For a generic cross-section we bound the compliance, volume and perimeter from above
in order to obtain estimates for the material volume and the value of the stress in that cross-section.
For a contradiction assume that there exists $I\subset(0,1)$ with Lebesgue measure at least $\frac12$ such that almost all cross-sections $t\in I$ satisfy
$\int_{\O_t}|\sigma|^2\,\d\hd^2+A_t-2F\ell^2>2c^*F^{2/3}\varepsilon^{2/3}\ell^2$ or $P_t>F^{2/3}\varepsilon^{-1/3}\ell^2$.
By Fubini's theorem we would then have
\begin{multline*}
\int_\Omega|\sigma|^2\,\d x+\vol(\O)+\varepsilon\per_\Omega(\O)-2 F\ell^2\\
\geq\int_I\int_{\O_t}|\sigma|^2\,\d\hd^2+A_t-2F\ell^2+\varepsilon P_t\,\d t
>\int_I2c^*F^{2/3}\varepsilon^{2/3}\ell^2\,\d t
\geq c^*F^{2/3}\varepsilon^{2/3}\ell^2,
\end{multline*}
a contradiction.
Therefore we may assume that for at least half the cross-sections $t\in(0,1)$ we have
\begin{equation*}
\int_{\O_t}|\sigma|^2\,\d\hd^2+A_t-2F\ell^2\leq2c^*F^{2/3}\varepsilon^{2/3}\ell^2
\qquad\text{and}\qquad
P_t\leq F^{2/3}\varepsilon^{-1/3}\ell^2.
\end{equation*}
From now on let $t$ denote such a cross-section.
The above bound directly implies an estimate on the cross-sectional volume.
Indeed, if $A_t>2F\ell^2$ or $A_t<\frac{F\ell^2}2$, then
\begin{equation*}
\int_{\O_t}|\sigma|^2\,\d\hd^2+A_t-2F\ell^2
\geq\frac{F^2\ell^4}{A_t}+A_t-2F\ell^2
\geq\frac{F\ell^2}2
>F^{2/3}\varepsilon^{2/3}\ell^2
>2c^*F^{2/3}\varepsilon^{2/3}\ell^2,
\end{equation*}
another contradiction.
Therefore we have in addition
\begin{equation*}
A_t\geq\frac{F\ell^2}2
\qquad\text{and}\qquad
A_t\leq2F\ell^2.
\end{equation*}
Finally, the above bound also implies that on the cross-section the material stress deviates only little from a vertical tensile unit stress,
but we will postpone the quantification of this fact to step\,3.

\emph{Step\,2.}
Using the previous perimeter and volume bounds we now characterize $\O_t$ as being (mainly) composed of connected components
with typical area $(\frac{A_t}{P_t})^2\sim(F\varepsilon)^{2/3}$ and perimeter (or equivalently diameter) $\frac{A_t}{P_t}\sim(F\varepsilon)^{1/3}$.
This could be done directly by elementary methods, but for later purposes it is a little more convenient to instead first replace $\O_t$ by a slightly nicer set $\tilde\O_t$
and then to express the above characterization in terms of a scaling behaviour for the volume of dilations of $\tilde\O_t$.
To this end abbreviate
\begin{equation*}
l=\frac18(F\varepsilon)^{1/3}
\end{equation*}
to be the typical diameter of the connected components of $\O_t$.
By definition of $l$ and the bounds from step\,1 we have $l\per_{\Omega_t}(\O_t)\leq\frac14\hd^2(\O_t)$.
Now \cite[Lem.\,3.1]{ChCoKo08} says that this condition implies the existence of a set $\tilde\O_t\subset\Omega_t$ with
\begin{equation*}
\hd^2(\O_t\cap\tilde\O_t)\geq\frac{\hd^2(\O_t)}2=\frac{A_t}2\geq\frac{F\ell^2}4
\end{equation*}
and, denoting the dilation of $\tilde\O_t$ by $r>0$ as $\tilde\O_t^r=\{x\in\Omega_t\,|\,\dist(x,\tilde\O_t)<r\}$,
\begin{equation*}
\hd^2(\tilde\O_t^r)\lesssim \hd^2(\O_t)\left(\frac rl\right)^2=A_t\left(\frac rl\right)^2\leq128\frac{F^{1/3}\ell^2r^2}{\varepsilon^{2/3}}
\qquad\text{for all }r>l.
\end{equation*}
This latter condition encodes that the connected components of $\tilde\O_t$ exhibit the behaviour mentioned above.

\emph{Step\,3.}
Finally, one estimates the excess compliance
necessary to distribute the approximate vertical unit stress in each of the connected components of $\O_t$ evenly on the upper and lower boundary of $\Omega$.
This redistribution of the stress can actually be viewed as a transport problem of the vertical momentum:
We interpret the vertical momentum $\rho_s(x_1,x_2)=e_3^T\sigma(x_1,x_2,s)e_3$ as a temporally changing material distribution, where time runs from $s=0$ to $s=t$.
The corresponding temporally changing material flux is then given by the time-dependent vector field $\omega_s(x_1,x_2)=(\sigma_{13}(x_1,x_2,s)\ \sigma_{23}(x_1,x_2,s))^T$ on $(0,\ell)^2$.
Indeed, $\rho$ and $\omega$ satisfy the transport equation
\begin{equation*}
\tfrac\partial{\partial s}\rho_s+\div\omega_s=0
\end{equation*}
in the distributional sense
since for any smooth $\phi:\Omega_t\to\R$ we have
\begin{multline*}
\int_0^t\int_{(0,\ell)^2}\rho_s\,\tfrac\partial{\partial s}\phi+\omega_s\cdot\nabla_{(x_1,x_2)}\phi\,\d(x_1,x_2)\,\d s
=\int_0^t\int_{(0,\ell)^2}e_3^T\sigma\,\nabla_{(x_1,x_2,s)}\phi\,\d(x_1,x_2)\,\d s\\
=e_3^T\int_{\Omega_{(0,t)}}\sigma\,\nabla\phi\,\d x
=e_3^T\int_{\partial\Omega_{(0,t)}}\phi\,\sigma\,n\,\d\hd^2
=e_3^T\left(\int_{\Omega_t}\phi\,\sigma\,e_3\,\d\hd^2-\int_{\Omega_0}\phi\,\sigma\,e_3\,\d\hd^2\right)\\
=\int_{(0,\ell)^2}\phi(\cdot,\cdot,t)\,\rho_t\,\d\hd^2-\int_{(0,\ell)^2}\phi(\cdot,\cdot,0)\,\rho_0\,\d\hd^2,
\end{multline*}
where we used Stokes' theorem and that $\sigma$ is divergence-free.
Now part of the compliance can be viewed as a cost associated with this transport:
Using \eqref{eqn:minCompliance} and Jensen's inequality we have
\begin{multline*}
\Delta\J
\geq\int_0^t\int_{\O_s}\sigma_{13}^2+\sigma_{23}^2\,\d\hd^2
+\int_{\O_s}\sigma_{33}^2\,\d\hd^2+A_t-2F\ell^2\,\d s
\geq\int_0^t\int_{\O_s}\sigma_{13}^2+\sigma_{23}^2\,\d\hd^2\,\d s\\
=\int_0^t\!\int_{\pi(\O_s)}|\omega_s|^2\,\d x\,\d s
\geq\frac1{\int_0^t\!A_s\,\d s}\left(\int_0^t\!\int_{\pi(\O_s)}|\omega_s|\,\d x\,\d s\right)^{\!2}
\geq\frac1{2F\ell^2}\left(\int_0^t\!\int_{(0,\ell)^2}|\omega_s|\,\d x\,\d s\right)^{\!2},
\end{multline*}
where $\pi(\O_s)$ stands for the projection of $\O_s$ into the $x_1$-$x_2$-plane.
Now the Wasserstein-1 optimal transport distance between $\rho_0$ and $\rho_t$ in the Benamou--Brenier formulation is \cite[\S\,6.1]{Sa15}
\begin{align*}
W_1(\rho_0,\rho_t)
&=\inf\left\{\int_0^t\int_{(0,\ell)^2}|\omega_s|\,\d x\,\d s\,\middle|\,\tfrac\partial{\partial s}\rho_s+\div\omega_s=0\text{ in the distributional sense}\right\}\\
&=\sup\left\{\int_{(0,\ell)^2}\psi(\rho_t-\rho_0)\,\d x\,\middle|\,\psi:(0,\ell)^2\to\R\text{ is Lipschitz with constant }1\right\},
\end{align*}
where the last equality is known as Kantorovich--Rubinstein duality (see for instance \cite[\S\,4.2.1]{Sa15} for the last formula).
Thus we have $\Delta\J\geq\frac1{2F\ell^2}(W_1(\rho_0,\rho_t))^2$,
and to bound the Wasserstein-1 distance from below we can simply construct a dual variable $\psi$ for the Kantorovich--Rubinstein formula,
which we do in the following.
Since $\rho_t$ is approximately $1$ in the connected components of $\O_t$ (which we still have to show) and those roughly have diameter $l$ and area $l^2$,
the mass of $\rho_t$ in each of these components needs to be spread out to an area $l^2/F$ in order to reach the density $\rho_0=F$.
Thus we expect the typical distance that a particle is transported to be $\sqrt{l^2/F}$.
We therefore abbreviate
\begin{equation*}
r=\frac l{\sqrt{F}}=\frac1{8}F^{-1/6}\varepsilon^{1/3}
\end{equation*}
to be the typical transport distance.
A good $\psi$ then is given by
\begin{equation*}
\psi(x)=\max\left\{r-\dist(x,\pi(\tilde\O_t)),0\right\},
\end{equation*}
which is $r$ on $\tilde\O_t$ and decreases to $0$ linearly with the distance to $\tilde\O_t$.
We thus obtain
\begin{equation*}
\Delta\J
\geq\frac1{2F\ell^2}\left(W_1(\rho_0,\rho_t)\right)^2
\geq\frac1{2F\ell^2}\left(\int_{(0,\ell)^2}\psi(\rho_t-\rho_0)\,\d x\right)^2,
\end{equation*}
and it remains to estimate $\int_{(0,\ell)^2}\psi(\rho_t-\rho_0)\,\d x=\int_{\Omega_t}\sigma_{33}\psi\,\d\hd^2-F\int_{(0,\ell)^2}\psi\,\d x$.
It is the estimate of the first integral that shows that $\sigma$ approximately has a unit vertical component
or at least that $\sigma_{33}$ is more or less greater than or equal to $1$.
Indeed, we calculate
\begin{equation*}
\int_{\Omega_t}\sigma_{33}\psi\,\d\hd^2
=\int_{\Omega_t}(\sigma_{33}-1)\psi\,\d\hd^2+\int_{\Omega_t}\psi\,\d\hd^2,
\end{equation*}
where the summands can be estimated as
\begin{gather*}
\int_{\Omega_t}\psi\,\d\hd^2
\geq r\hd^2(\tilde\O_t)
\geq r\hd^2(\tilde\O_t\cap\O_t)
\geq\frac{rF\ell^2}4,\\
\left\vert\int_{\Omega_t}(\sigma_{33}-1)\psi\,\d\hd^2\right\vert
\leq\left(\int_{\O_t}(\sigma_{33}-1)^2\,\d\hd^2\right)^{\frac12}\left(\int_{\Omega_t}\psi^2\,\d\hd^2\right)^{\frac12}.
\end{gather*}
We bound the two factors via
\begin{gather*}
\int_{\O_t}(\sigma_{33}-1)^2\,\d\hd^2
=\int_{\O_t}\sigma_{33}^2-2\sigma_{33}+1\,\d\hd^2
=\int_{\O_t}\sigma_{33}^2\,\d\hd^2-2F\ell^2+A_t
\leq2c^*F^{2/3}\varepsilon^{2/3}\ell^2,\\
\int_{\Omega_t}\psi^2\,\d\hd^2
\leq r^2\hd^2(\tilde\O_t^r)
\leq C\frac{F^{1/3}\ell^2r^4}{\varepsilon^{2/3}}
\end{gather*}
for some fixed constant $C$ so that in summary we obtain
\begin{gather*}
\int_{(0,\ell)^2}\psi(\rho_1-\rho_0)\,\d x
=\int_{\Omega_t}(\sigma_{33}-1)\psi\,\d\hd^2+(1\!-\!F)\!\int_{(0,\ell)^2}\!\psi\,\d x
\geq\tfrac{r(1-F)F\ell^2}4-r^2\sqrt{2c^*CF}\ell^2.
\end{gather*}
By choosing $c^*$ small enough (depending on $C$) and using $\sqrt\varepsilon<F\leq\frac12$ we get $\int_{(0,\ell)^2}\psi(\rho_1-\rho_0)\,\d x\geq\frac{rF\ell^2}{16}$ and thus
\begin{equation*}
\Delta\J
\geq\frac1{2F\ell^2}\left(\frac{rF\ell^2}{16}\right)^2
\geq\frac1{2^{15}}F^{2/3}\varepsilon^{2/3}\ell^2,
\end{equation*}
which is the desired contradiction to our assumption if we choose $c^*<2^{-15}$.
\end{proof}


\section{Upper bounds}

As already explained in the introduction, the upper bounds are obtained by constructions that are composed of layers of elementary cells.
We briefly introduce our corresponding notation.

We will only specify the constructions and compute their energy for the upper half $(0,\ell)^2\times(\frac12,1)$ of $\Omega$
since the construction for the lower half is always mirror-symmetric.
The different layers of elementary cells in that upper half are numbered in ascending order,
beginning with $1$ for the layer sitting on the midplane $\Omega_{\frac12}$.
The index of the last layer will be denoted $n\in\N\cup\{\infty\}$.
Sometimes this last layer will not be composed of elementary cells but will be constructed as a special boundary layer.
The height of the $i$th layer (or equivalently of the elementary cells in that layer) is denoted $l_i$,
and $L_i=l_1+...+l_{i}$ denotes the accumulated height of the first $i$ layers so that the bottom of layer $i+1$ is at height $\frac12+L_i$ and we must have $\frac12+L_n=1$.
The width of the elementary cells in the $i$th layer is denoted $w_i$ and halves from layer to layer,
\begin{equation*}
w_i=2^{1-i}w_1.
\end{equation*}
The characteristic function of the material distribution in the elementary cell of layer $i$ is denoted $\chi_i:(0,w_i)^2\times(0,l_i)\to\{0,1\}$,
the stress field is $\sigma_i:(0,w_i)^2\times(0,l_i)\to\R_{\sym}^{3\times 3}$.
The characteristic function of our construction $\O$ then is given by
  \begin{equation*}
     \chi_\O(x,y,z)=\begin{cases} \chi_i\left(x \bmod w_i,y \bmod w_i,z-\frac{1}{2}-L_{i-1}\right) & \text{if } z-\frac{1}{2}\in(L_{i-1},L_i), \\ \chi_i\left(x \bmod w_i,y \bmod w_i,\frac{1}{2}-z-L_{i-1}\right) & \text{if } \frac{1}{2}-z\in(L_{i-1},L_i), \end{cases}
  \end{equation*}
while the constructed global stress field is given by
  \begin{equation*}
     \sigma(x,y,z)=\begin{cases} \sigma_i\left(x \bmod w_i,y \bmod w_i,z-\frac{1}{2}-L_{i-1}\right) & \text{if } z-\frac{1}{2}\in(L_{i-1},L_i), \\ g\left[\sigma_i\left(x \bmod w_i,y \bmod w_i,\frac{1}{2}-z-L_{i-1}\right)\right] & \text{if } \frac{1}{2}-z\in(L_{i-1},L_i), \end{cases}
  \end{equation*}
where the operator $g$ inverts the sign of the vertical shear stress components,
\begin{equation*}
g:\left(\begin{smallmatrix}
\sigma_{11}&\sigma_{12}&\sigma_{13}\\
\sigma_{12}&\sigma_{22}&\sigma_{23}\\
\sigma_{13}&\sigma_{23}&\sigma_{33}
\end{smallmatrix}\right)
\mapsto\left(\begin{smallmatrix}
\sigma_{11}&\sigma_{12}&-\sigma_{13}\\
\sigma_{12}&\sigma_{22}&-\sigma_{23}\\
-\sigma_{13}&-\sigma_{23}&\sigma_{33}
\end{smallmatrix}\right).
\end{equation*}
All our constructions will ensure $\sigma\in\Sigma_\ad^{F,\ell,L}(\O)$.
Finally, the excess cost of our construction will be abbreviated as
  \begin{equation*}
    \Delta\J
    =\int_\Omega|\sigma|^2\,\d x+\vol(\O)+\varepsilon\per_\Omega(\O)-2F\ell^2
    \geq\J^{\varepsilon,F,\ell}(\O)-\J^{*,F,\ell}_0.
  \end{equation*}
We will only explicitly calculate the perimeter relative to $\Omega$;
changing from $\per_\Omega$ to $\per_{\R^3}$ is trivial for all our constructions.

%

\subsection{Small and extremely small force}

Here we detail the construction from \cref{fig:constructionOverview} left, whose elementary cells consist of struts along the edges of a pyramid.
We employ the same construction for the regimes of small and of extremely small forces.
The different energy scaling for extremely small forces then only comes from the boundary layer dominating the total cost.
We first describe the construction within a single elementary cell and estimate its excess cost contribution.
We then describe the construction of a boundary cell in the boundary layer and estimate its excess cost contribution.
Finally, we describe the assembly of all cells into the full construction and estimate its excess cost.

\paragraph{Elementary cell construction.}
The material distribution within an elementary cell is illustrated in \cref{fig:elementaryCellSmallForce}.
We will just detail the front right quarter of the construction (shown in black), the other quarters being mirror-symmetric.

\begin{figure}
	\centering
	\includegraphics[scale=1]{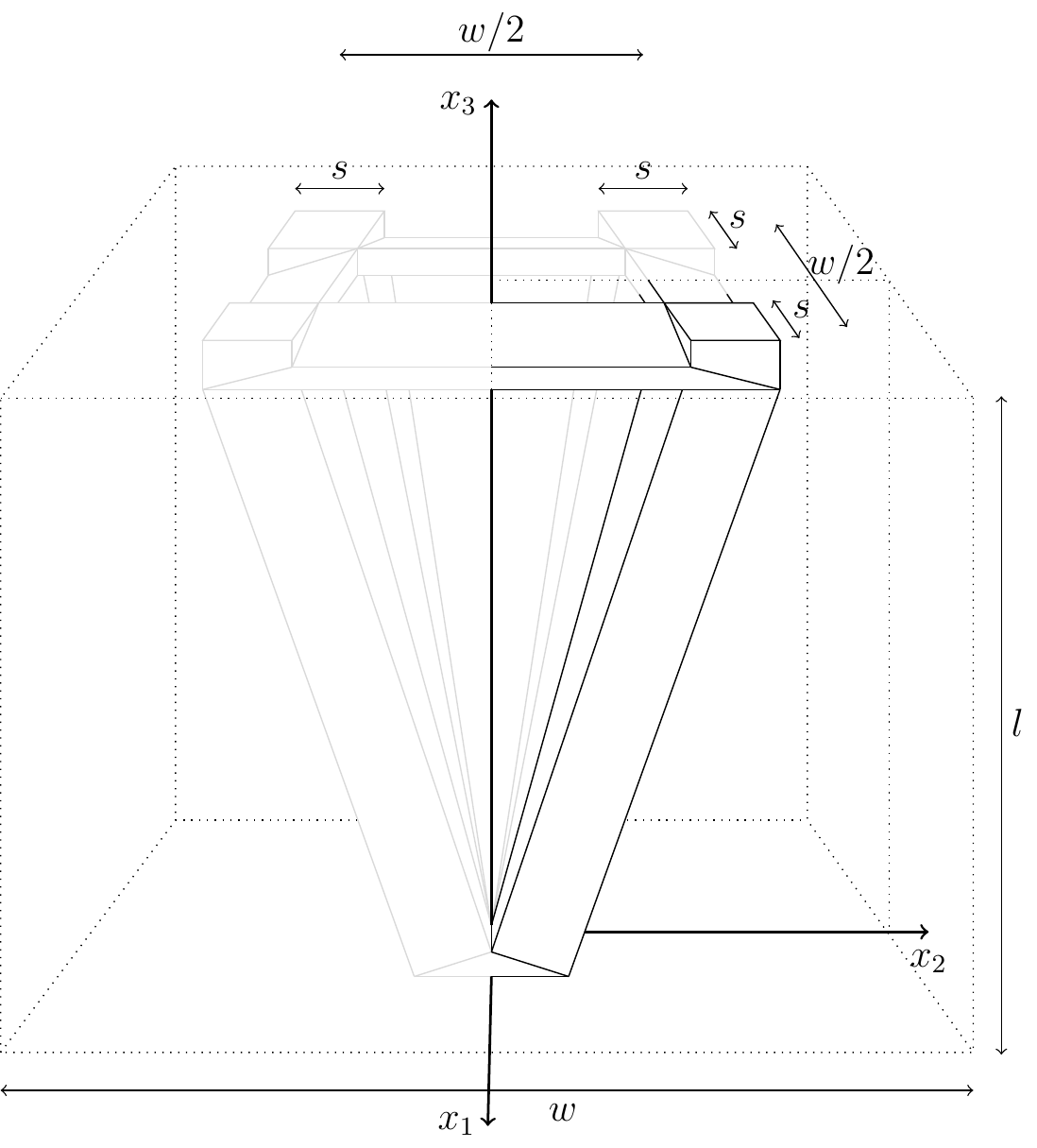}
	\caption{Detailed view of an elementary cell of width $w$ and height $l$.}
	\label{fig:elementaryCellSmallForce}
\end{figure}

Let $\co A$ denote the convex hull of a set $A\subset\R^3$.
We specify the material distribution $\O$ within the front right quarter of the unit cell as the union of (not necessarily disjoint) simple geometric shapes
\begin{align*}
\O^1&=\co\{P_1,\ldots,P_8\},
&\O^7&=\co\{A,\ldots,G\},\\
\O^2&=\co\{P_6,\ldots,P_9\},
&\O^8&=\co\{A,\ldots,D,E,H\},\\
\O^3&=\co\{P_1,P_5,P_6,P_8\},
&\O^9&=\co\{D,F,G,H,I,\ldots,L\},\\
\O^4&=\co\{P_1,P_2,P_5,\ldots,P_8\},
&\O^{10}&=\co\{A,\ldots,D,G,H\},\\
\O^5&=\co\{P_1,P_4,P_5,\ldots,P_8\},
&\O^{11}&=\co\{M,\ldots,P,B,E,F,H\}.\\
\O^6&=\co\{P_1,P_6,P_9,P_8,A,E,F,G\},
\end{align*}
The points $P_1,\ldots,P_9,A,\ldots,P$ are illustrated in \cref{fig:trussBase,fig:trussTop}; we next specify their coordinates.
To this end we place a coordinate system at the bottom centre of the elementary cell as indicated in \cref{fig:elementaryCellSmallForce}
so that the elementary cell of width $w$ and height $l$ (for simplicity we drop the index $i$ indicating the layer) occupies the volume $(-\frac w2,\frac w2)^2\times(0,l)$.
We will assume $w\leq l$ which will be ensured throughout the construction.
We further fix the lengths $a$ and $s$, using their relation shown in \cref{fig:trussBase} left, via
\begin{equation*}
s^2=\frac{Fw^2}{4},
\qquad
a=\sqrt2s\tan\alpha
\end{equation*}
(they are chosen such that all struts will have unit stress),
where $\alpha$ denotes the angle of the $z$-axis with the upwards pointing trusses.
We do not provide an explicit formula for $\alpha$ and just note that it is implicitly and uniquely determined by the construction as a function of the lengths $w$, $l$, and $s$.
However, using $F\leq\frac12$ it is straightforward to see that
\begin{equation*}
\frac{\sqrt2-1}4\frac{w}{l}
\leq\frac{\sqrt2(w/4-s/2)}l
\leq\tan\alpha
\leq\frac1{\sqrt2}\frac w{l}
\end{equation*}
as long as $w\leq l$.
With this the point coordinates are given as
\begin{equation*}
\arraycolsep=1.4pt
\begin{array}{rl}
P_1&=(s,s,0),\\
P_2&=(0,s,0),\\
P_3&=(0,0,0),\\
P_4&=(s,0,0),\\
P_5&=(s,s,\frac a2),\\
P_6&=(0,s,\frac a2),\\
P_7&=(0,0,\frac a2),\\
P_8&=(s,0,\frac a2),\\
P_9&=(0,0,a),
\end{array}
\qquad
\begin{array}{rl}
A&=(\frac w4-\frac s2,\frac w4-\frac s2,l),\\
B&=(\frac w4+\frac s2,\frac w4-\frac s2,l),\\
C&=(\frac w4+\frac s2,\frac w4+\frac s2,l),\\
D&=(\frac w4-\frac s2,\frac w4+\frac s2,l),\\
E&=(\frac w4+\frac s2,\frac w4-\frac s2,l-\frac a2),\\
F&=(\frac w4+\frac s2,\frac w4+\frac s2,l-a),\\
G&=(\frac w4-\frac s2,\frac w4+\frac s2,l-\frac a2),\\
H&=(\frac w4+\frac s2,\frac w4+\frac s2,l-\frac a2),\\
\ &\
\end{array}
\qquad
\begin{array}{rl}
I&=(\frac w4-\frac s2,0,l),\\
J&=(\frac w4+\frac s2,0,l-\frac a2),\\
K&=(\frac w4+\frac s2,0,l-a),\\
L&=(\frac w4-\frac s2,0,l-\frac a2),\\
M&=(0,\frac w4-\frac s2,l),\\
N&=(0,\frac w4+\frac s2,l-\frac a2),\\
O&=(0,\frac w4+\frac s2,l-a),\\
P&=(0,\frac w4-\frac s2,l-\frac a2).\\
\ &\
\end{array}
\end{equation*}

\begin{figure}
	\centering
    \includegraphics[scale=0.7]{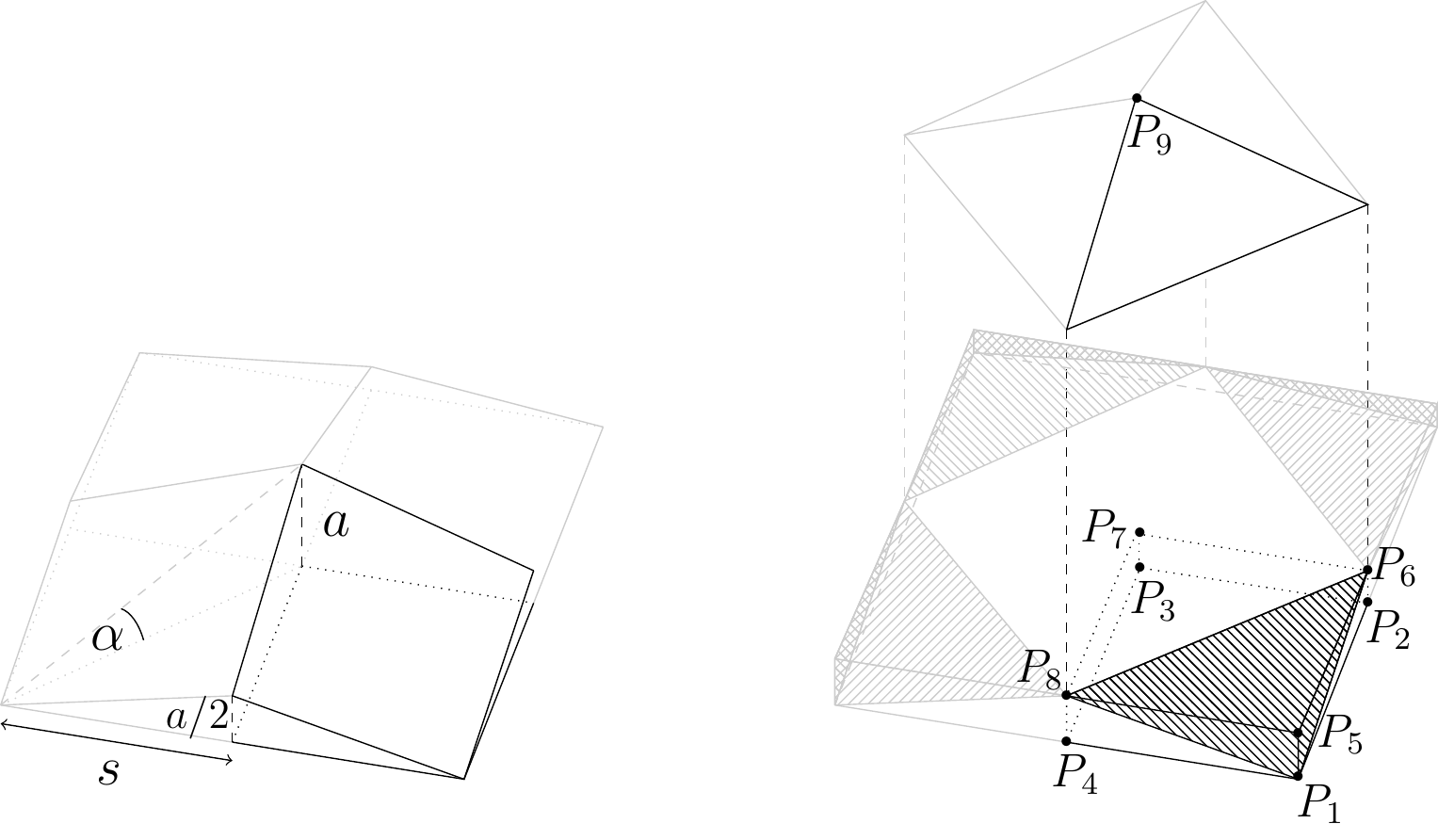} \\
    \includegraphics[scale=0.7]{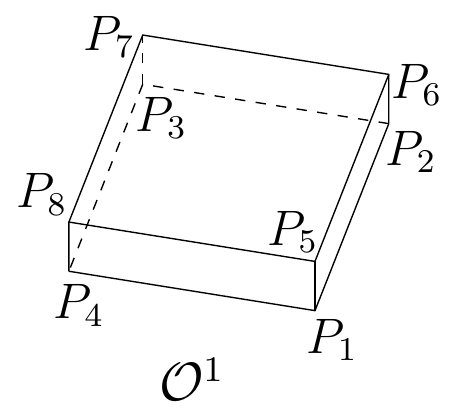}\,%
    \includegraphics[scale=0.7]{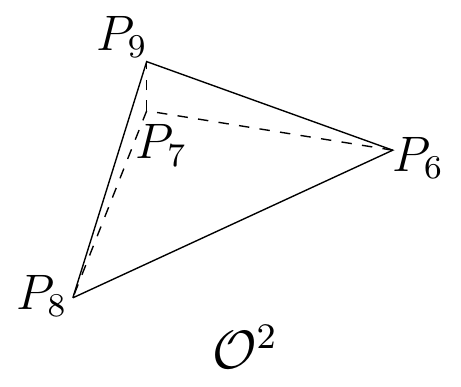}\,%
    \includegraphics[scale=0.7]{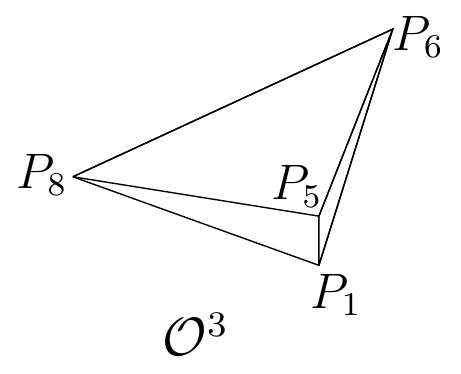}\,%
    \includegraphics[scale=0.7]{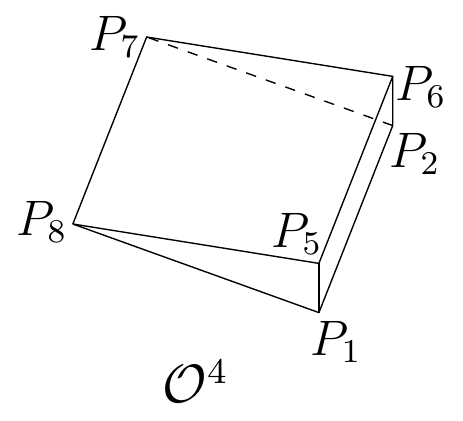}\,%
    \includegraphics[scale=0.7]{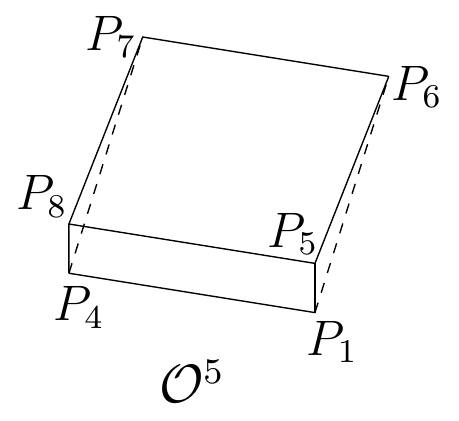}
	\caption{Detailed view of the bottom part of the elementary cell from which the upwards pointing trusses emanate.}
	\label{fig:trussBase}
\end{figure}

\begin{figure}
	\centering
	\includegraphics[scale=0.7]{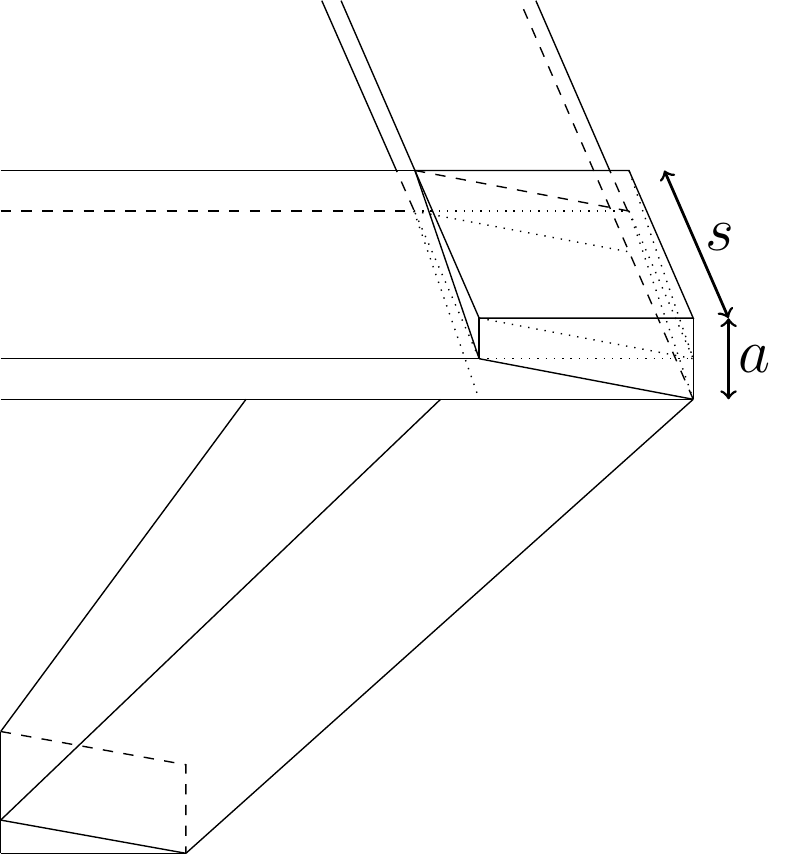} 
	\includegraphics[scale=0.7]{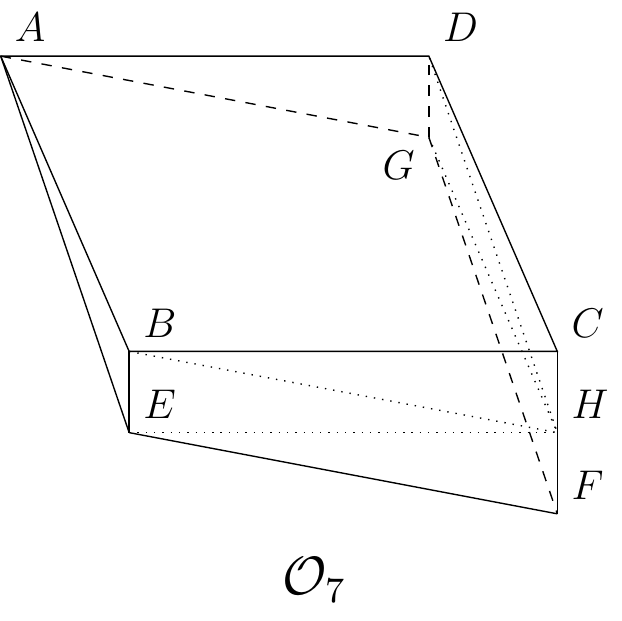} \\
	\includegraphics[scale=0.7]{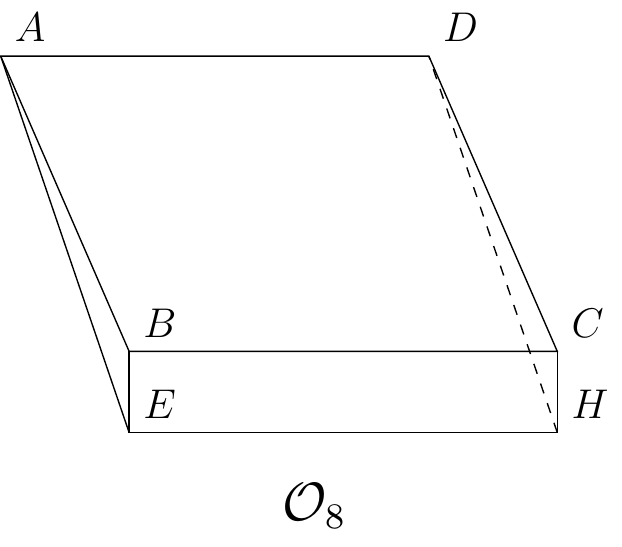}
	\includegraphics[scale=0.7]{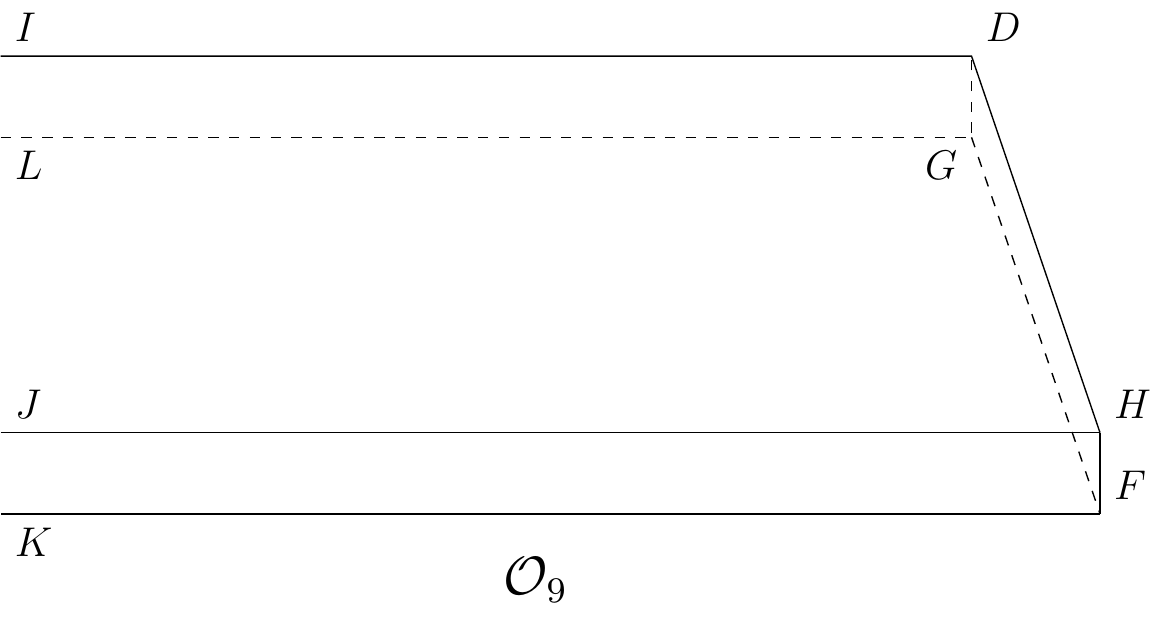} \\
	\includegraphics[scale=0.7]{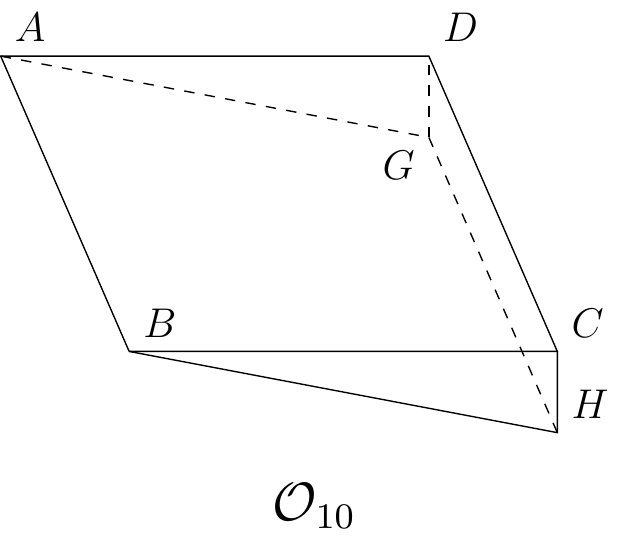}
	\includegraphics[scale=0.7]{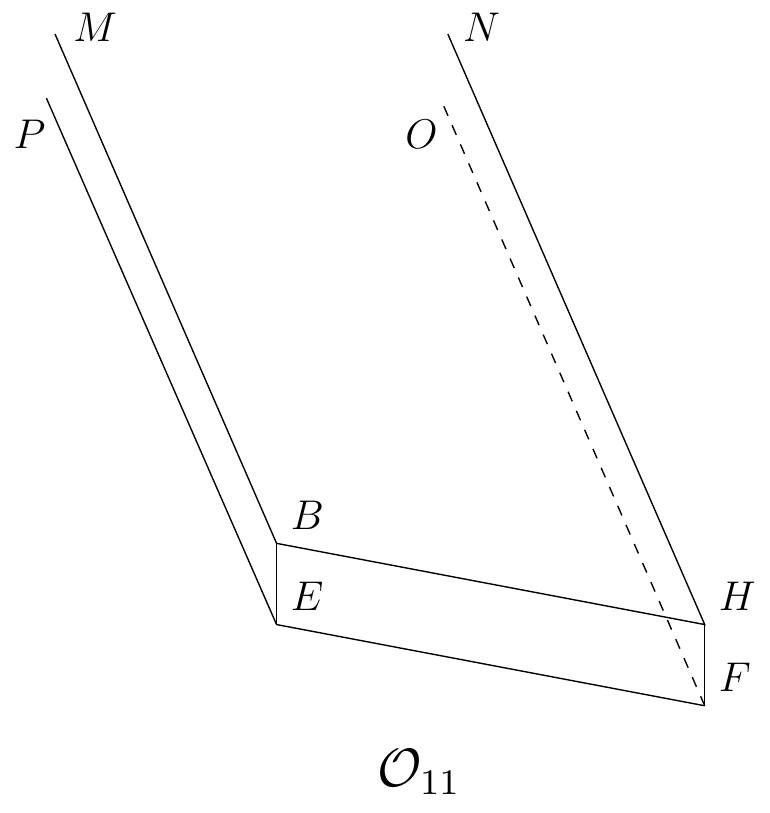}
	\caption{Detailed view of the top part of the elementary cell at which the front right upwards pointing truss ends.
	Occluded edges are dashed, auxiliary lines are dotted.
	}
	\label{fig:trussTop}
\end{figure}

In addition to the material distribution we have to specify an admissible stress field $\sigma$ in the front right quarter of the elementary cell
(the stress field $\sigma'$ in the front left quarter is then obtained as $\sigma'(x_1,x_2,x_3)=\diag(1,-1,1)\sigma(x_1,-x_2,x_3)\diag(1,-1,1)$
for $\diag(r,s,t)$ denoting a diagonal matrix with entries $r,s,t$,
and the stress field $\sigma''$ in the back half is obtained as $\sigma''(x_1,x_2,x_3)=\diag(-1,1,1)[\sigma+\sigma'](-x_1,x_2,x_3)\diag(-1,1,1)$).
Denoting the characteristic function of shape $\O^j$ by $\chi_{\O^j}$, we set
\begin{equation*}
\sigma=\sum_{j=1}^{11}\chi_{\O^j}\sigma^j
\end{equation*}
for constant stresses $\sigma^j$ specified below.
We will abbreviate the identity matrix by $\Id$, the Euclidean standard basis by $\{e_1,e_2,e_3\}$ and the tangent vector to the upwards pointing truss by
\begin{equation*}
\tau=(\tfrac1{\sqrt2}\sin\alpha,\tfrac1{\sqrt2}\sin\alpha,\cos\alpha)^T.
\end{equation*}
We now fix
\begin{equation*}
\arraycolsep=1.4pt
\begin{array}{rl}
\sigma^1&=e_3\otimes e_3,\\
\sigma^2&=\Id,\\
\sigma^3&=-\Id,
\end{array}
\qquad
\begin{array}{rl}
\sigma^4&=e_1\otimes e_1,\\
\sigma^5&=e_2\otimes e_2,\\
\sigma^6&=\tau\otimes\tau,
\end{array}
\qquad
\begin{array}{rl}
\sigma^7&=\Id,\\
\sigma^8&=-e_2\otimes e_2,\\
\sigma^9&=-e_2\otimes e_2,
\end{array}
\qquad
\begin{array}{rl}
\sigma^{10}&=-e_1\otimes e_1,\\
\sigma^{11}&=-e_1\otimes e_1.\\
\ &\
\end{array}
\end{equation*}
Note that the outer product $v\otimes v$ of a vector $v\in\R^3$ with itself represents a uniaxial tensile stress of magnitude $|v|^2$ along the direction $v/|v|$.
By construction, $\sigma=0$ outside the material.
It is furthermore straightforward to check that $\sigma$ is divergence-free throughout the elementary cell.
To this end it suffices to note that on each face of any simple geometry $\O^j$ at most two more simple geometries $\O^k$ are adjacent
and their stresses normal to the face add up to zero.
Finally, on the boundary of the elementary cell there is a unit normal stress on $(-s,s)^2\times\{0\}$
as well as on $(\pm\frac w4-\frac s2,\pm\frac w4+\frac s2)\times(\pm\frac w4-\frac s2,\pm\frac w4+\frac s2)\times\{l\}$.
Since four elementary cells of half the width are stacked on top of each elementary cell, this implies that the stresses between subsequent layers are compatible.

\paragraph{Elementary cell excess cost.}
The volumes of the simple geometries can readily be calculated as
\begin{equation*}
\setcounter{MaxMatrixCols}{11}
\begin{matrix}
\O^1&\O^2&\O^3&\O^4&\O^5&\O^6&\O^7&\O^8&\O^9&\O^{10}&\O^{11}\\
\hline
s^2\tfrac a2&
s^2\tfrac a{12}&
s^2\tfrac a{12}&
s^2\tfrac a{4}&
s^2\tfrac a{4}&
s^2\frac{|F-P_1|}{\cos\alpha}&
s^2\tfrac a2&
s^2\tfrac a{4}&
s\tfrac a2(\frac w4+\frac s2)&
s^2\tfrac a{4}&
s\tfrac a2(\frac w4+\frac s2)
\end{matrix}
\end{equation*}
so that, using $|F-P_1|\leq l/\cos\alpha$, the material volume of the elementary cell can be estimated from above as
\begin{equation*}
\vol_\cell
\leq4\sum_{j=1}^{11}\vol(\O^j)
=4\left[s^2\tfrac l{\cos^2\alpha}+s^2\tfrac{8a}{3}+sw\tfrac a4\right].
\end{equation*}
In the following we will frequently use the estimates
\begin{equation*}
a\leq \frac{sw}l,
\qquad
\frac1{\cos^2\alpha}=1+\tan^2\alpha\leq1+\frac{w^2}{2l^2},
\qquad\text{and}\qquad
\frac1{\cos\alpha}\leq\sqrt{1+\frac{w^2}{2l^2}}\leq1+\frac{w^2}{4l^2}.
\end{equation*}
With those, the volume becomes
\begin{equation*}
\vol_\cell
\leq4s^2l+3s^2\tfrac{w^2}l+\tfrac{32}3s^3\tfrac wl
\leq4s^2l+\tfrac{41}3s^2\tfrac{w^2}l.
\end{equation*}
Likewise, the surface area shared between void and each simple geometry can readily be calculated as
\begin{equation*}
\setcounter{MaxMatrixCols}{11}
\begin{matrix}
\O^1&\O^2,\O^3,\O^4,\O^5&\O^6&\O^7&\O^9,\O^{11}\\
\hline
s\frac a2&
0&
<4|F-P_1|\sqrt{s^2+(\frac a2)^2}&
2sa&
<2\frac{w+s}2(\frac a2+\sqrt{s^2+(\frac a2)^2})
\end{matrix}
\end{equation*}
so that the perimeter contribution from the elementary cell can be estimated from above as
\begin{align*}\textstyle
\per_\cell
&\leq2sa+16|F-P_1|\sqrt{s^2+(\tfrac a2)^2}+4s^2+8sa+8(w+s)(\tfrac a2+\sqrt{s^2+(\tfrac a2)^2})\\
&\leq2sa+16l(1+\tfrac{w^2}{4l^2})s(1+\tfrac{a^2}{8s^2})+4s^2+8sa+8(w+s)(\tfrac a2+s(1+\tfrac{a^2}{8s^2}))\\
&\leq16ls+14s^2\tfrac wl+2s\tfrac{w^2}l+4s\tfrac{w^2}l+\tfrac{sw^4}{2l^3}+12s^2+4s\tfrac{w^2}l+8ws+s\tfrac{w^3}{l^2}+\tfrac{s^2w^2}{l^2}\\
&\leq\tfrac{125}2ls
\end{align*}
(where in the last step we exploited $s\leq w\leq l$).
Finally, noting that the stress nowhere exceeds Frobenius norm $\sqrt3$, the squared $L^2$-norm of the stress on the elementary cell (which we abbreviate as $\comp_\cell$) can be estimated via
\begin{multline*}
\comp_\cell
\leq4\left[|\sigma^6|^2\vol(\O^6)+3\vol(\O^1\cup\O^7\cup\O^9\cup\O^{11})\right]\\
\leq4s^2\tfrac l{\cos^2\alpha}+12(\tfrac32s^2a+\tfrac14saw)
\leq4s^2l+23s^2\tfrac{w^2}l
\end{multline*}
Summarizing, the excess cost contribution of an elementary cell can be estimated via
\begin{equation*}
\Delta\J_\cell
=\vol_\cell+\comp_\cell+\varepsilon\per_\cell-2Fw^2l
\leq(\tfrac{41}3+23)s^2\tfrac{w^2}{l}+\varepsilon\tfrac{125}2sl
\leq32(F\tfrac{w^4}{l}+\varepsilon\sqrt Fwl).
\end{equation*}
We here evenly distributed $\J^{*,F,\ell}_0=2F\ell^2$ over the total volume $\Omega$ so that the amount corresponding to the elementary cell is $2Fw^2l$.
We now pick the minimizing elementary cell height which is still no smaller than $w$,
\begin{equation*}
l=\max\{w,F^{\frac14}w^{\frac32}\varepsilon^{-\frac12}\}.
\end{equation*}
(In fact, had we not simplified the excess cost using the assumption $w\leq l$,
we would have arrived at the same choice of $l$ as the minimizer of the non-simplified elementary cell excess cost.)
With this choice we obtain
\begin{equation*}
\Delta\J_\cell
\leq64\max\{F^{\frac34}w^{\frac52}\varepsilon^{\frac12},\sqrt Fw^2\varepsilon\}.
\end{equation*}

\paragraph{Boundary cell construction.}
The last layer has to evenly distribute the stress from the previous layer of elementary cells over the top boundary of $\Omega$.
It will again be partitioned into identical cells of width $w$ and height $l$ which we term `boundary cells' (again we drop the index $n$ indicating the layer).
We again place a coordinate system at the bottom centre of the boundary cell so that it occupies the volume
\begin{equation*}
\omega\vcentcolon=(-\tfrac w2,\tfrac w2)^2\times(0,l).
\end{equation*}
We choose
\begin{equation*}
l=\sqrt3w/2
\end{equation*}
and fill the boundary cell completely with material.
The stress field in the boundary cell is again specified as
\begin{equation*}
\sigma=\sum_{j=12}^{18}\chi_{\O^j}\sigma^j
\end{equation*}
for particular domains $\O^{12},\ldots,\O^{18}\subset \omega$ and stresses $\sigma^{12},\ldots,\sigma^{18}$ (see \cref{fig:boundaryLayer}).
To define these, let us introduce the point $Z=(0,0,-s)$ with $s$ defined via
\begin{equation*}
s^2=Fw^2/4.
\end{equation*}
We abbreviate by $B_Z(r)$ the ball with centre $Z$ and radius $r$
and by $C_{Z,v}(r)$ the infinite cylinder with centre $Z$, axis $v\in\R^3$, and radius $r$.
With this preparation the domains are given as
\begin{equation*}
\arraycolsep=1.4pt
\begin{array}{rl}
\O^{12}&=B_Z(\sqrt3s)\cap[(-s,s)^2\times(0,l)],\\
\O^{13}&=C_{Z,e_1}(\sqrt2s)\cap[(-s,s)^2\times(0,l)],\\
\O^{14}&=C_{Z,e_2}(\sqrt2s)\cap[(-s,s)^2\times(0,l)],\\
\O^{15}&=\{x=Z+tv\in \omega\,|\,|v|=1,\,t\in(\sqrt3s,\frac{\sqrt3}2w),\,Z+\frac{\sqrt3}2wv\in \omega\},\\
\end{array}
\begin{array}{rl}
\O^{16}&=\omega\setminus B_Z(\frac{\sqrt3}2w),\\
\O^{17}&=\omega\setminus C_{Z,e_1}(\frac{\sqrt2}2w),\\
\O^{18}&=\omega\setminus C_{Z,e_2}(\frac{\sqrt2}2w)\\
\ &\
\end{array}
\end{equation*}
and the stresses as
\begin{equation*}
\arraycolsep=1.4pt
\begin{array}{rl}
\sigma^{12}&=\Id,\\
\sigma^{13}&=-e_1\otimes e_1,\\
\sigma^{14}&=-e_2\otimes e_2,\\
\sigma^{15}(x)&=\frac{3s^2}{|x-Z|^2}\frac{(x-Z)\otimes(x-Z)}{|x-Z|^2},\\
\end{array}
\qquad
\begin{array}{rl}
\sigma^{16}&=F\Id,\\
\sigma^{17}&=-Fe_1\otimes e_1,\\
\sigma^{18}&=-Fe_2\otimes e_2.\\
\ &\
\end{array}
\end{equation*}
Again by checking that the normal stresses add up to zero at all domain interfaces (except the bottom face of $\O^{12}$ and the top face of $\O^{16}$)
we obtain that $\sigma$ is divergence-free with normal tensile stress of magnitude $F$ at the top boundary of the boundary cell
and normal tensile stress of magnitude $1$ on $(-s,s)^2\times\{0\}$, where it is attached to the elementary cell underneath, whose stress it exactly balances.

\begin{figure}
	\centering
	\includegraphics[scale=0.6]{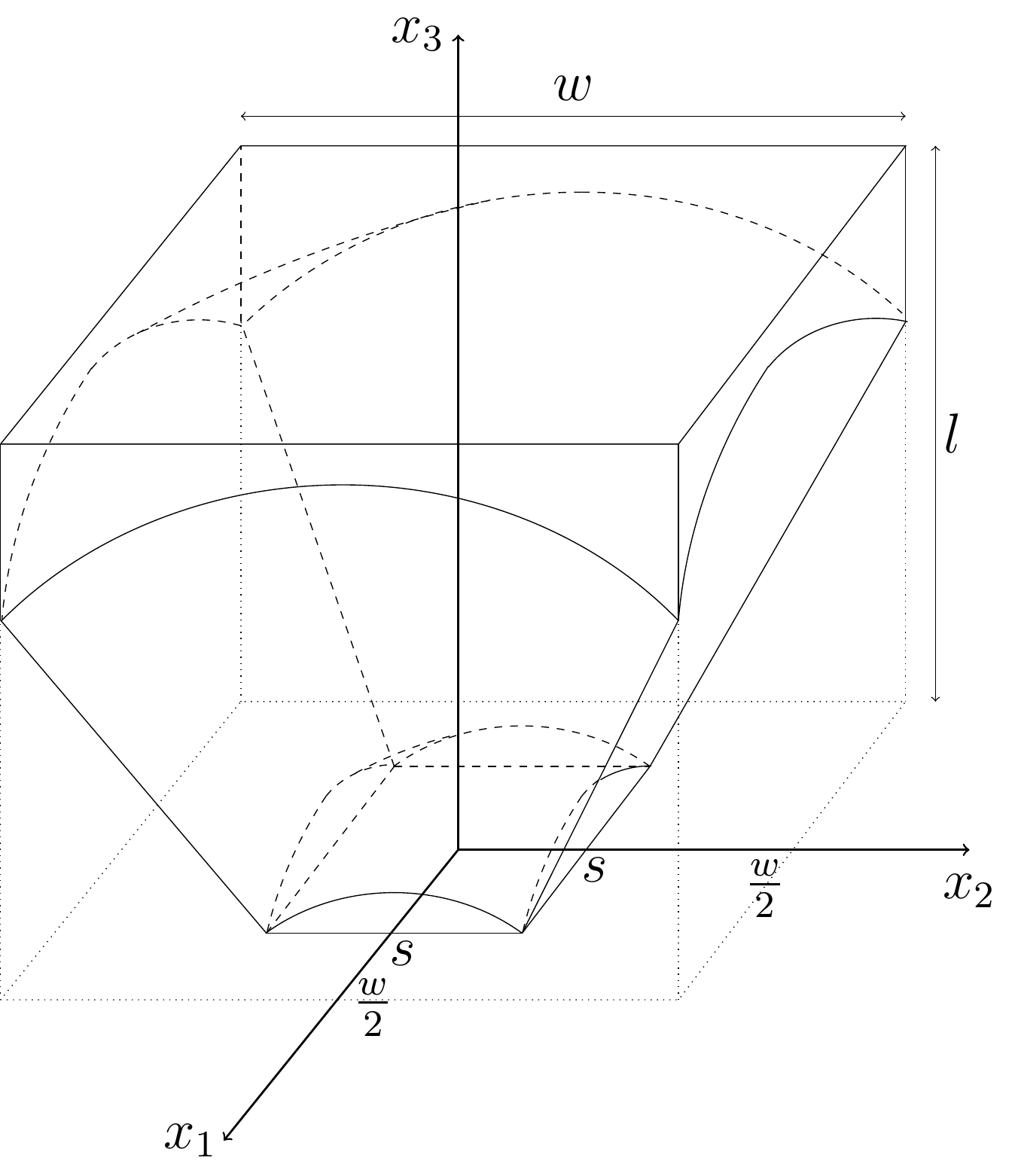} 
    \includegraphics[scale=0.6]{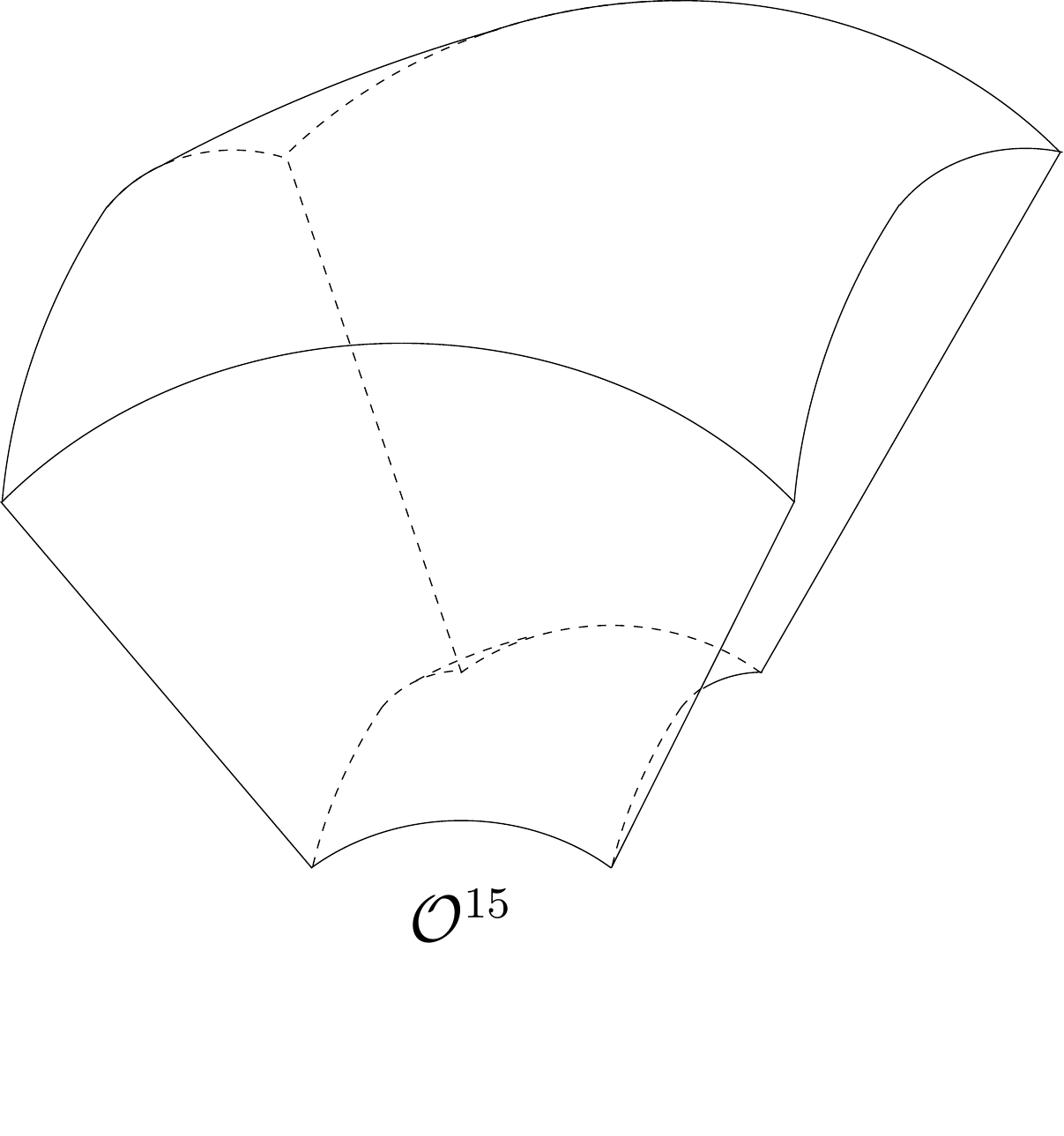} \\ 
	\includegraphics[scale=0.6]{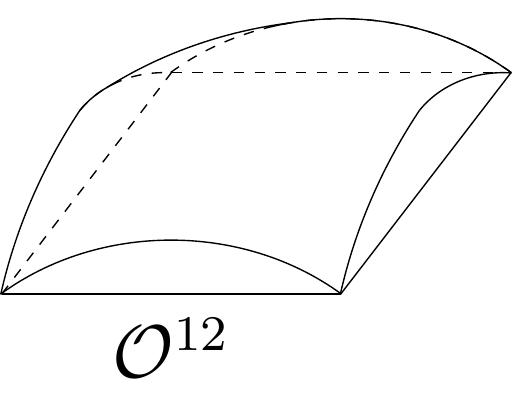}
	\includegraphics[scale=0.6]{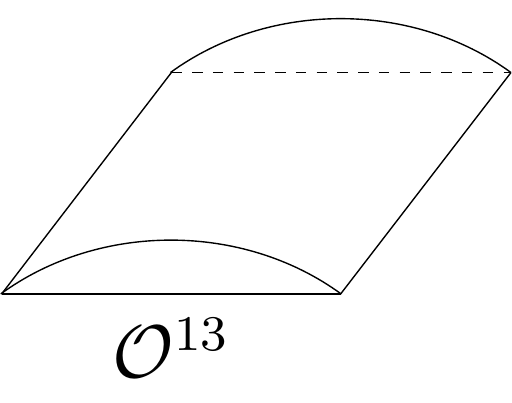}
	\includegraphics[scale=0.6]{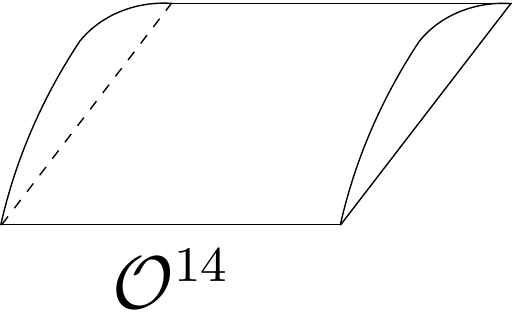}  
	\includegraphics[scale=0.6]{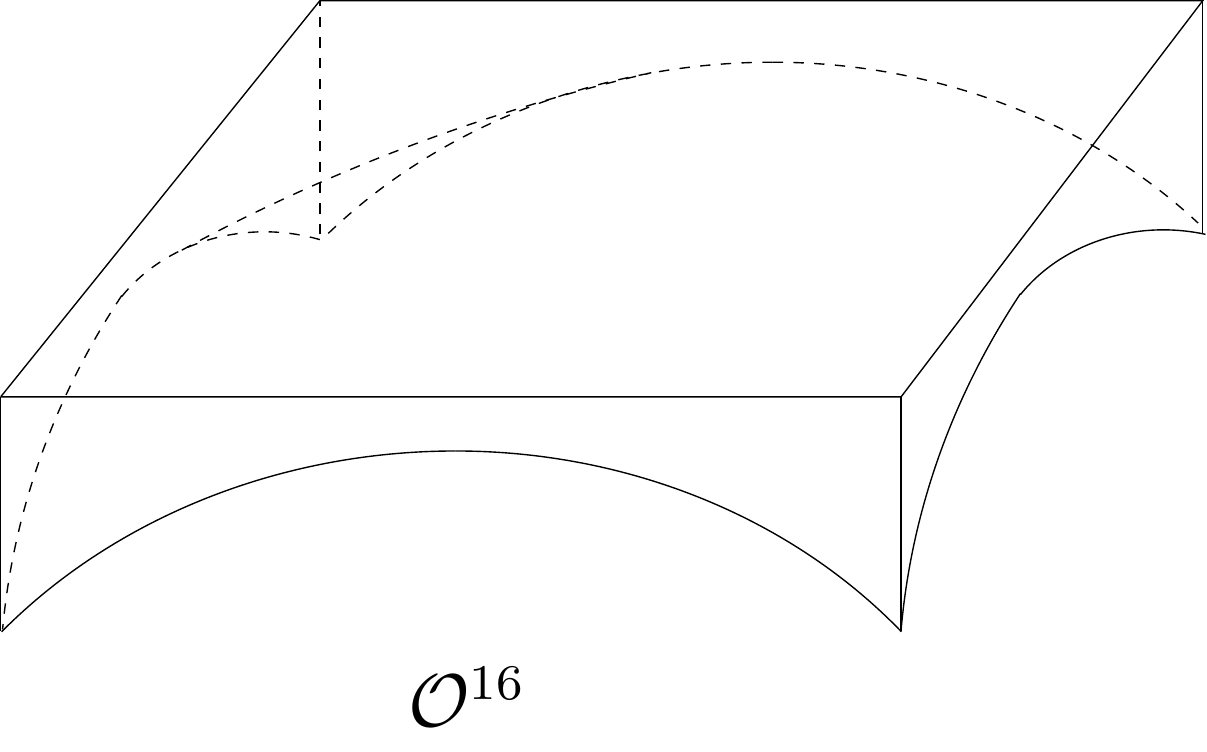}
	\includegraphics[scale=0.6]{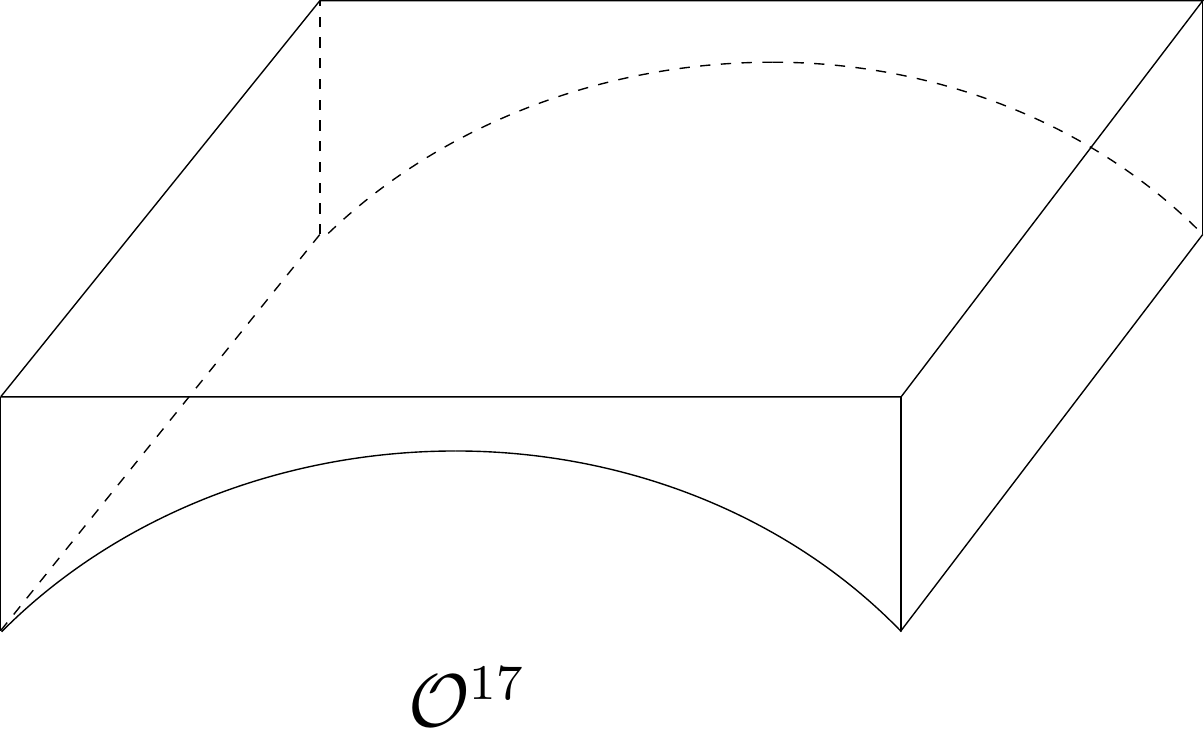}	\includegraphics[scale=0.6]{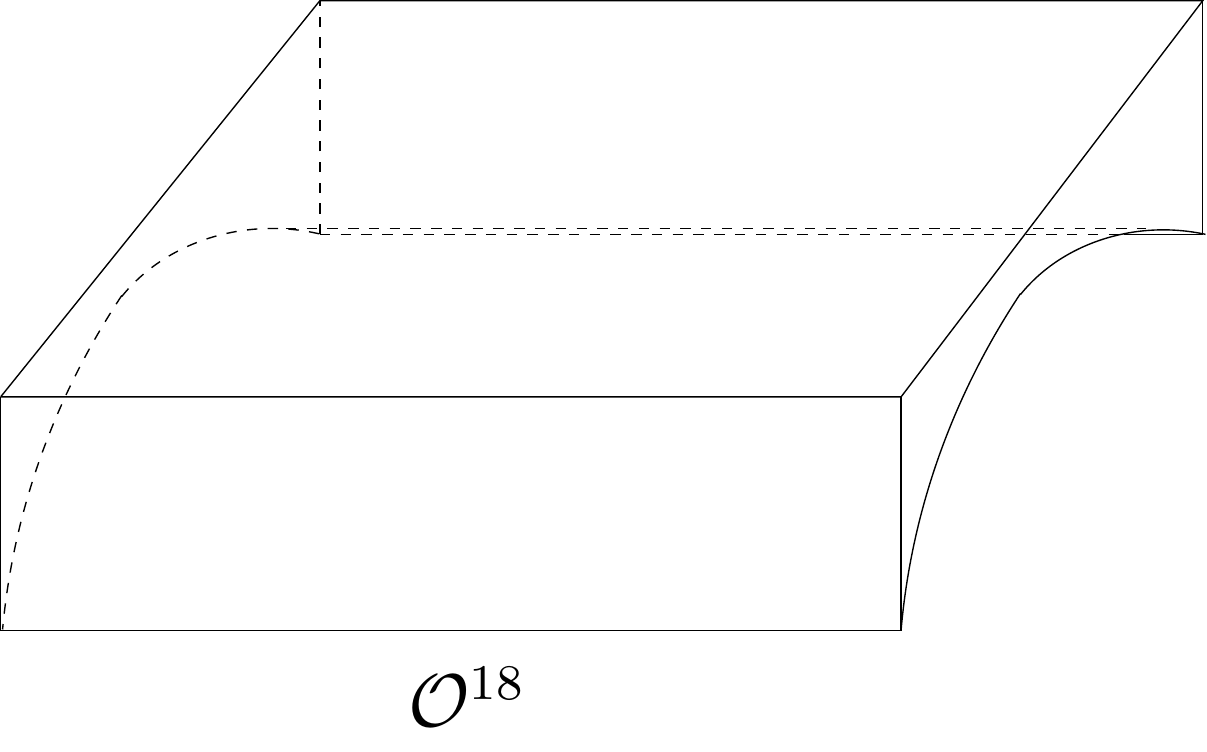}
	\caption{Auxiliary domains for a boundary element.
	}
	\label{fig:boundaryLayer}
\end{figure}

\paragraph{Boundary cell excess cost.}
Since the full boundary cell volume is occupied by material,
the volume and perimeter contribution of the boundary cell can be estimated by
\begin{equation*}
\vol_\cell=lw^2=\frac{\sqrt 3}2w^3,
\qquad
\per_\cell\leq w^2
\end{equation*}
(note that only the bottom face counts to the perimeter; the top face lies in $\partial\Omega$ and is thus not counted,
while the sides are adjacent to the neighbouring cells and thus do not form an interface with the void).
Using that the stress nowhere exceeds Frobenius norm $\sqrt3$, the compliance can be estimated from above as
\begin{equation*}
\comp_\cell\leq3\vol_\cell=\frac{3\sqrt 3}2w^3.
\end{equation*}
Thus, the excess cost contribution of a boundary cell becomes
\begin{equation*}
\Delta\J_\cell
=\comp_\cell+\vol_\cell+\varepsilon\per_\cell-2Fw^2l
\leq2\sqrt3w^3+\varepsilon w^2.
\end{equation*}

\paragraph{Full construction.}
The layers of elementary cells are stacked as previously described, where the last, $n$th layer is composed of boundary cells.
We let index $N$ refer to the last layer whose elementary cells satisfy $w<l$ (thus $w_{j}=l_j$ for $j=N+1,\ldots,n-1$).
For the time being let us assume $N\geq1$.
In principle one could stop layering at index $N$ and directly introduce the boundary layer,
however, it turns out that in that case the optimal energy scaling in the small force regime is only reached for $F\gtrsim\varepsilon^{2/7}$, which is why we will continue adding layers.

Let us now identify the width $w_1$ of the coarsest elementary cells.
Since all layers have to sum up to total height $1$, we have
\begin{multline*}
1
=2\sum_{k=1}^{n}l_k
=2\left[\sum_{k=1}^NF^{\frac14}w_k^{\frac32}\varepsilon^{-\frac12}+\sum_{k=N+1}^{n-1}w_k+\frac{\sqrt 3}2w_n\right]\\
=4\left[F^{\frac14}\varepsilon^{-\frac12}w_1^{\frac32}\sqrt2\sum_{k=1}^N2^{-\frac32k}+w_1\sum_{k=N+1}^{n-1}2^{-k}+\frac{\sqrt 3}22^{-n}w_1\right]
\sim F^{\frac14}\varepsilon^{-\frac12}w_1^{\frac32}+2^{-N}w_1
\end{multline*}
as long as $1<N<n$, where we exploited the geometric series.
Due to $F^{\frac14}\varepsilon^{-\frac12}w_1^{\frac32}=l_1\geq l_{N+1}=2^{-N}w_1$ we arrive at $1\sim F^{\frac14}\varepsilon^{-\frac12}w_1^{\frac32}$.
Hence we pick
\begin{equation*}
w_1\sim F^{-\frac16}\varepsilon^{\frac13}.
\end{equation*}
In more detail, $w_1$ shall be the largest width smaller than $F^{-\frac16}\varepsilon^{\frac13}/8$
such that $\ell$ is an integer multiple of $w_1$.
For this to satisfy $w_1\sim F^{-\frac16}\varepsilon^{\frac13}$ we require $\ell\geq F^{-\frac16}\varepsilon^{\frac13}$.
Note that with this choice of $w_1$ the heights of all layers actually add up to less than one, which is remedied by increasing $l_1$ until the total height exactly equals $1$.
This little change of $l_1$ increases the excess cost $\Delta\J_\cell$ of the coarsest elementary cells at most by a bounded factor.

We next specify the number $n-N$ of layers with $w=l$.
Since the excess cost of the boundary layer at least scales like $\varepsilon$, which is achieved by the choice $w_n\sim\varepsilon$, we set $n$ to be the first index for which
\begin{equation*}
w_n\leq\varepsilon
\end{equation*}
(thus $n=1+\lfloor\log_2\tfrac{2w_1}\varepsilon\rfloor$ for the floor function $\lfloor\cdot\rfloor$).
Furthermore, $N$ is the last index for which $w_N<l_N$ or equivalently
\begin{equation*}
w_N\geq\varepsilon/\sqrt F.
\end{equation*}
Therefore we have
\begin{equation*}\textstyle
n-N
=\log_2\tfrac{w_N}{w_n}
\leq|\log_2\sqrt F|.
\end{equation*}

We are finally in the position to estimate the total excess cost.
Note that in the $i$th layer there are $(\ell/w_i)^2$ many elementary or boundary cells.
Denoting the excess cost of cells in layer $i$ by $\Delta\J_\cell^i$,
the total excess cost thus is
\begin{multline*}
\Delta\J
=2\!\sum_{k=1}^{n}\!\big(\!\tfrac{\ell}{w_k}\!\big)^{\!2\!}\Delta\J_{\cell}^k
\leq2\!\sum_{k=1}^{N}\!\big(\!\tfrac{\ell}{w_k}\!\big)^{\!2\!}64F^{\frac34}w_k^{\frac52}\varepsilon^{\frac12}
+2\hspace{-1.5ex}\sum_{k=N+1}^{n-1}\hspace{-1.5ex}\big(\!\tfrac{\ell}{w_k}\!\big)^{\!2\!}\sqrt Fw_k^{2}\varepsilon
+2\big(\!\tfrac{\ell}{w_n}\!\big)^{\!2\!}(2\sqrt3w_n^3+\varepsilon w_n^{2})\\
\lesssim\ell^2\left[F^{\frac34}\varepsilon^{\frac12}w_1^{\frac12}\sum_{k=1}^{\infty}2^{\frac{1-k}2}
+(n-N)\sqrt F\varepsilon
+\varepsilon\right]
\lesssim\ell^2F^{\frac{2}{3}}\varepsilon^{\frac{2}{3}}+\ell^2\varepsilon.
\end{multline*}
In the regime of small forces the first summand dominates, in the regime of extremely small forces the second one.

So far we had assumed $N\geq1$ or equivalently $w_1<l_1$.
For our above choice of $w_1$ this only holds for $\varepsilon\lesssim\sqrt F$.
If this is violated, we instead have $N=0$, and the calculation above yields $\Delta\J\lesssim\ell^2[(n-N)\sqrt F\varepsilon+\varepsilon]$.
Since again $n-N\leq|\log_2\sqrt F|$, the scaling is not impaired.


\subsection{Intermediate force}\label{sec:intermediateForce}
In the regime of intermediate forces the optimal energy scaling is also obtained by a two-dimensional construction, which is constantly extended along the third dimension.
Indeed, let $\bar\Omega=(0,\ell)\times(0,1)$ and let $\bar\O\subset\bar\Omega$ and $\bar\sigma:\bar\Omega\to\R_\sym^{2\times2}$ be such
that $\bar\sigma=0$ on $\bar\Omega\setminus\bar\O$, $\div\bar\sigma=0$ in $\bar\Omega$ and $\bar\sigma n=F(e_2\cdot n)e_2$ on the boundary $\partial\bar\Omega$. Then
\begin{equation*}
\O=(0,\ell)\times\bar\O\subset\Omega
\qquad\text{and}\qquad
\sigma(x_1,x_2,x_3)=\left(\begin{smallmatrix}
0&0\\0&\bar\sigma(x_2,x_3)
\end{smallmatrix}\right)
\end{equation*}
satisfy $\sigma\in\Sigma_\ad^{F,\ell,1}(\O)$, and
\begin{equation*}
\Delta\J
=\ell\left(\int_{\bar\Omega}|\bar\sigma|^2\,\d x+\hd^2(\bar\O)+\varepsilon\per_{\bar\Omega}(\bar\O)-2F\ell\right).
\end{equation*}
Hence we have reduced the three-dimensional construction to the problem of finding $\bar\O$ and $\bar\sigma$ with
\begin{equation*}
\Delta\bar\J
\vcentcolon=\int_{\bar\Omega}|\bar\sigma|^2\,\d x+\hd^2(\bar\O)+\varepsilon\per_{\bar\Omega}(\bar\O)-2F\ell
\lesssim\ell\varepsilon^{\frac23}.
\end{equation*}
Such $\bar\O$ and $\bar\sigma$ were constructed in \cite{KoWi14}, where $\Delta\bar\J\lesssim\ell F^{1/3}\varepsilon^{2/3}$ is proved for any $F\leq1-\frac1{16}$.
The three-dimensional extension of this construction is illustrated in \cref{fig:intermediateForces} right.
In this section we want to complement this construction by yet an alternative, valid for $F>\frac1{16}$ and satisfying $\Delta\bar\J\lesssim\ell(1-F)^{1/3}\varepsilon^{2/3}$.
Its three-dimensional extension is illustrated in \cref{fig:intermediateForces} left.
As before, we first specify the two-dimensional elementary cells and estimate their excess cost,
then we specify the boundary cells and estimate their excess cost,
and finally we describe the full construction.

\paragraph{Elementary cell construction.}
The elementary cell is visualized in \cref{fig:intermediateForcesCell}.
Placing the coordinate system origin at its bottom centre, it occupies the volume $(-\frac w2,\frac w2)\times(0,l)$.
We only describe the construction in the right half, the left half being mirror-symmetric.
The material region is partitioned into the simple shapes
\begin{equation*}
\arraycolsep=1.4pt
\begin{array}{rl}
\O^1&=\co\{P_1,P_2,P_3,P_{11}\},\\
\O^2&=\co\{P_4,P_5,P_7,P_{10}\},\\
\O^3&=\co\{P_4,P_8,P_9,P_{11}\},
\end{array}
\qquad
\begin{array}{rl}
\O^4&=\co\{P_5,P_6,P_7\},\\
\O^5&=\co\{P_4,P_{10},P_{11}\}.\\
\
\end{array}
\end{equation*}
with point coordinates
\begin{equation*}
\arraycolsep=1.4pt
\begin{array}{rl}
P_1&=(\frac w2,0),\\
P_2&=(\frac w2,l),\\
P_3&=(\frac w2-s,l),
\end{array}
\qquad
\begin{array}{rl}
P_4&=(\frac w2-s,a),\\
P_5&=(s,l),\\
P_6&=(0,l),
\end{array}
\qquad
\begin{array}{rl}
P_7&=(0,l-a),\\
P_8&=(0,a),\\
P_9&=(0,0),
\end{array}
\qquad
\begin{array}{rl}
P_{10}&=((1-F)\frac w2,0),\\
P_{11}&=(\frac w2-s,0),\\
\
\end{array}
\end{equation*}
where we abbreviated
\begin{equation*}
s=Fw/4
\qquad\text{and}\qquad
a=s\tan\alpha.
\end{equation*}
The angle $\alpha$ is again implicitly and uniquely determined as a function of $w$, $l$, and $s$.
The union of these simple shapes forms the material region.
The stress is defined as $\bar\sigma=\sum_{i=1}^5\chi_{\O^i}\sigma^i$ for
\begin{equation*}
\sigma^1=e_2\otimes e_2,\quad
\sigma^2=\tau\otimes\tau,\quad
\sigma^3=-e_1\otimes e_1,\quad
\sigma^4=\sigma^5=\Id
\end{equation*}
with $\tau=(-\sin\alpha,\cos\alpha)^T$.
It is straightforward to check that $\bar\sigma$ is divergence-free and has unit normal tensile stress
on $[(-\frac w2,-(1-F)\frac w2)\cup((1-F)\frac w2,\frac w2)]\times\{0\}$ and $[(-\frac w2,-(1-\frac F2)\frac w2)\cup(-s,s)\cup((1-\frac F2)\frac w2,\frac w2)]\times\{l\}$,
balanced exactly by the elementary cells on the next and previous layer.

\begin{figure}
  \centering
	\includegraphics[scale=1]{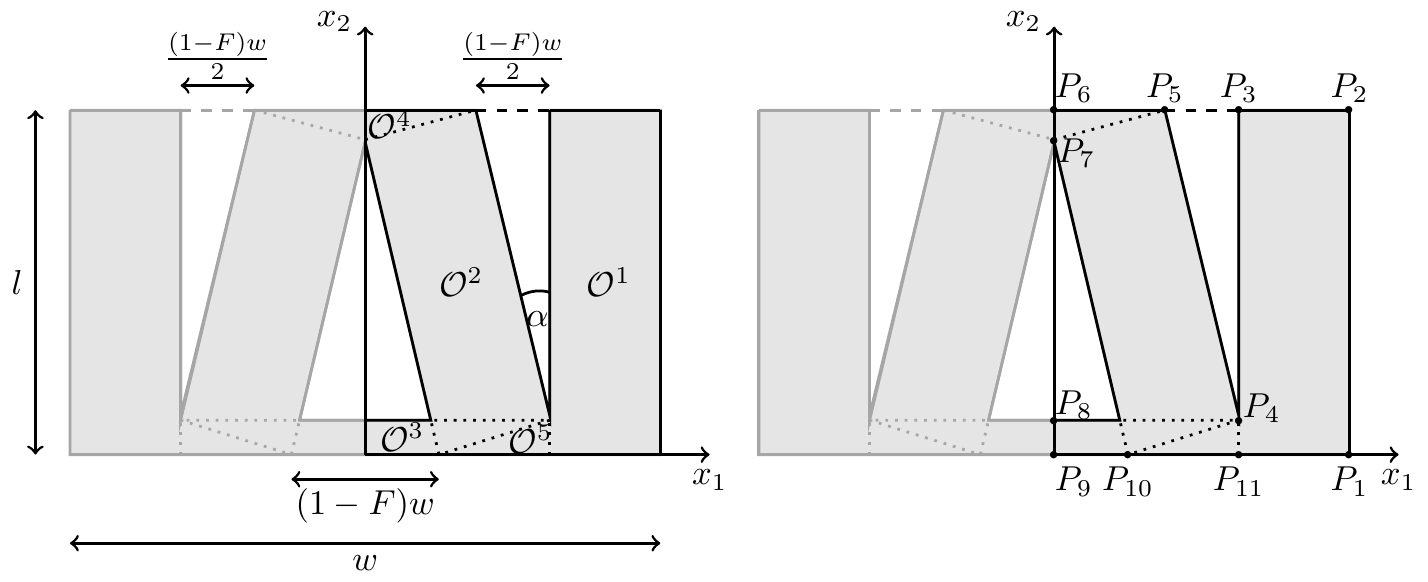} 
	\caption{Illustration of the two-dimensional elementary cell construction.}
	\label{fig:intermediateForcesCell}
\end{figure}

\paragraph{Elementary cell excess cost.}
The volumes of the simple geometries can readily be calculated as
\begin{equation*}
\begin{matrix}
\O^1&\O^2&\O^3&\O^4&\O^5\\
\hline
sl&
\frac s{\cos^2\alpha}(l-a)&
a(\frac w2-s)&
s\tfrac a{2}&
s\tfrac a{2}
\end{matrix}
\end{equation*}
so that the material volume of the elementary cell can be estimated from above as
\begin{equation*}
\vol_\cell
\leq2\sum_{j=1}^{5}\vol(\O^j)
=2\left[sl+\tfrac s{\cos^2\alpha}(l-a)+a(\tfrac w2-s)+sa\right]
\leq2sl(1+\tfrac1{\cos^2\alpha})+wa.
\end{equation*}
Throughout the construction we will ensure $(1-F)w\leq l$, from which we obtain
\begin{equation*}
\tan\alpha\leq\tfrac{(1-F)w}l,
\qquad
a\leq\tfrac{(1-F)sw}l,
\qquad
\tfrac1{\cos^2\alpha}
\leq1+\tfrac{(1-F)^2w^2}{l^2},
\qquad
\tfrac1{\cos\alpha}
\leq1+\tfrac{(1-F)^2w^2}{2l^2}.
\end{equation*}
Consequently the volume estimate becomes
\begin{equation*}
\vol_\cell
\leq4sl+\tfrac{2(1-F)^2sw^2}l+\tfrac{(1-F)sw^2}l
\leq4sl+\tfrac{3(1-F)sw^2}l.
\end{equation*}
Likewise, the surface area shared between void and each simple geometry can readily be calculated as
\begin{equation*}
\setcounter{MaxMatrixCols}{11}
\begin{matrix}
\O^1&\O^2&\O^3&\O^4&\O^5\\
\hline
<l&
<2l/\cos\alpha&
<2(1-F)w&
0&
0
\end{matrix}
\end{equation*}
so that the perimeter contribution from the elementary cell can be estimated from above as
\begin{equation*}
\per_\cell
\leq2l+4\tfrac l{\cos\alpha}+4(1-F)w
\leq12l.
\end{equation*}
Finally, noting that the stress nowhere exceeds Frobenius norm $\sqrt2$, the squared $L^2$-norm of the stress on the elementary cell can be estimated via
\begin{multline*}
\comp_\cell
\leq2\left[|\sigma^1|^2\vol(\O^1)+|\sigma^2|^2\vol(\O^2)+2\vol(\O^3\cup\O^4)\right]\\
\leq2sl+2\tfrac s{\cos^2\alpha}(l-a)+2a(w-s)
\leq4sl+4\tfrac{(1-F)sw^2}l.
\end{multline*}
Summarizing, the excess cost contribution of an elementary cell can be estimated via
\begin{equation*}
\Delta\bar\J_\cell
=\vol_\cell+\comp_\cell+\varepsilon\per_\cell-2Fwl
\leq12\left(\tfrac{(1-F)w^3}l+\varepsilon l\right).
\end{equation*}
We now specify the elementary cell height as the minimizer
\begin{equation*}
l=(1-F)^{\frac12}w^{\frac32}\varepsilon^{-\frac12}
\end{equation*}
so that the excess cost contribution of the elementary cell becomes
\begin{equation*}
\Delta\bar\J_\cell
\leq24(1-F)^{\frac12}w^{\frac32}\varepsilon^{\frac12}.
\end{equation*}

\paragraph{Boundary cell construction.}
Again the last layer is composed of special boundary cells of width $w$ and height $l$ (as usual we drop the index $n$).
Placing a coordinate system at the bottom centre of the boundary cell it occupies the volume $(-\tfrac w2,\tfrac w2)\times(0,l)$, whose right half we denote by
\begin{equation*}
\omega\vcentcolon=(0,\tfrac w2)\times(0,l).
\end{equation*}
We choose
\begin{equation*}
l=\sqrt2w/2-2s
\end{equation*}
and fill the boundary cell completely with material.
We describe the stress field only in the right half of the boundary cell, the left half being mirror symmetric.
Let us first abbreviate
\begin{equation*}
s=Fw/4,
\qquad
Z=(w/2,-2s)
\end{equation*}
and $B_Z(r)$ to be the ball of radius $r$ centred at $Z$.
The stress field in the boundary cell is then specified as
\begin{equation*}
\bar\sigma=\sum_{j=6}^{9}\chi_{\O^j}\sigma^j
\end{equation*}
with
\begin{equation*}
\arraycolsep=1.4pt
\begin{array}{rl}
\O^{6}&=B_Z(\sqrt8s)\cap\omega,\\
\O^{7}&=\{Z+tv\,|\,|v|=1,\,t\in(\sqrt8s,\sqrt 2\frac w2),\,Z+\sqrt 2\frac w2v\in \omega\},\\
\O^{8}&=\omega\setminus B_Z(\sqrt 2\frac w2),\\
\O^{9}&=(0,\frac w2)\times(0,\sqrt8s-2s),
\end{array}
\quad
\begin{array}{rl}
\sigma^{6}&=\Id,\\
\sigma^{7}(x)&=\frac{\sqrt8s}{|x-Z|}\frac{(x-Z)\otimes(x-Z)}{|x-Z|^2},\\
\sigma^{8}&=F\Id,\\
\sigma^{9}&=-e_1\otimes e_1
\end{array}
\end{equation*}
(see \cref{fig:2DBoundary}).
Again it is straightforward to check that $\bar\sigma$ is divergence-free
and has boundary stresses compatible with the below layer of elementary cells and the boundary load applied on $\partial\bar\Omega$.

\begin{figure}
	\includegraphics[scale=1,trim=0 10 0 0]{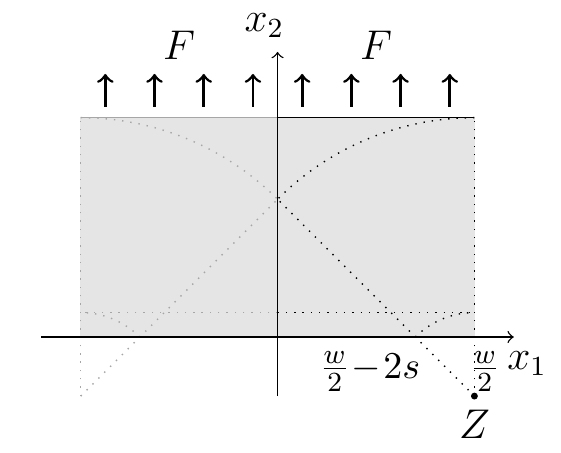}\hfill
	\includegraphics[scale=1]{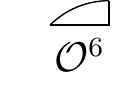}
	\includegraphics[scale=1]{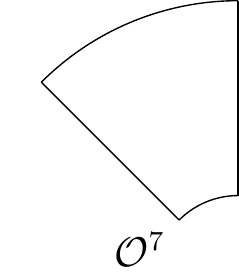}
	\includegraphics[scale=1]{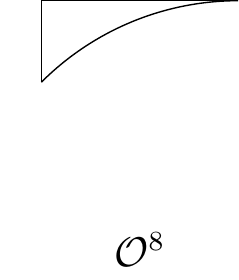}
	\includegraphics[scale=1]{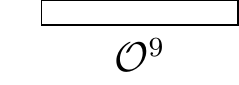}
	\caption{Illustration of the boundary cell construction for intermediate forces.}
	\label{fig:2DBoundary}
\end{figure}

\paragraph{Boundary cell excess cost.}
Since the boundary cell is filled completely with material and since the stress nowhere exceeds Frobenius norm $\sqrt2$,
the volume, free perimeter and compliance of the boundary cell can be estimated as
\begin{equation*}
\vol_\cell=wl<\tfrac{w^2}{\sqrt2},
\qquad
\per_\cell=(1-F)w,
\qquad
\comp_\cell<2\vol_\cell<2\tfrac{w^2}{\sqrt2}.
\end{equation*}
Its contribution to the excess cost thus becomes
\begin{equation*}
\Delta\bar\J_\cell
=\comp_\cell+\vol_\cell+\varepsilon\per_\cell-2Fwl
\leq3\tfrac{w^2}{\sqrt2}+\varepsilon(1-F)w.
\end{equation*}

\paragraph{Full construction.}
Let us first identify the width $w_1$ of the coarsest elementary cells.
Since all layers have to sum up to height $1$ we have
\begin{equation*}
1=2\sum_{k=1}^nl_k
=2^{\frac52}(1-F)^{\frac12}\varepsilon^{-\frac12}w_1^{\frac32}\sum_{k=1}^{n-1}2^{-\frac32k}+\sqrt2w_n
\sim(1-F)^{\frac12}\varepsilon^{-\frac12}w_1^{\frac32}
\end{equation*}
so that we choose $w_1$ to be the largest width smaller than $(1-F)^{-\frac13}\varepsilon^{\frac13}$ such that $\ell$ is an integer multiple of $w_1$.
If $\ell\geq(1-F)^{-\frac13}\varepsilon^{\frac13}$ this implies
\begin{equation*}
w_1\sim(1-F)^{-\frac13}\varepsilon^{\frac13}.
\end{equation*}
As before, the layer heights now sum up to something less than $1$, which is remedied by sufficiently increasing $l_1$.

We will stop layering as soon as $(1-F)w_n\geq l_n$, thus $w_n\sim(1-F)\varepsilon$.
Denoting again the excess cost contribution of the cells in layer $i$ by $\Delta\bar\J_\cell^i$,
the total excess cost of all cells can then be bounded via
\begin{equation*}
\Delta\bar\J
\leq2\sum_{i=1}^n\big(\tfrac\ell{w_i}\big)\Delta\bar\J_\cell^i
\leq2\ell\sum_{i=1}^{n-1}24(1-F)^{\frac12}w_1^{\frac12}2^{\frac{1-i}2}\varepsilon^{\frac12}+\tfrac6{\sqrt2}\ell w_n+2\ell(1-F)\varepsilon
\lesssim\ell(1-F)^{\frac13}\varepsilon^{\frac23}.
\end{equation*}
Note that for the superconductor problem the scaling actually would be $\ell(1-F)^{\frac23}\varepsilon^{\frac23}$.
Our different power of $(1-F)$ stems from the excess compliance in regions $\O^4$ and $\O^2\cap\O^3$.
A more careful construction might be able to recover the better power, however, we do not attempt this here
since $F$ is bounded away from $1$ anyway in the regime of intermediate forces.


\subsection{Large force}
Here we describe the construction from \cref{fig:constructionOverview} right in which the elementary cells consist of full material blocks with small roughly conical holes drilled inside.
The construction of an admissible stress field with optimal scaling requires substantially more effort than in the superconductor setting.
Due to the estimate of the previous section it suffices to provide a construction of the correct energy scaling for $1-F<\frac1{64}$, which we will assume in the following.

\paragraph{Elementary cell material distribution.}
As before, we use coordinates such that the elementary cell occupies the volume
\begin{equation*}
\omega\vcentcolon=(-\tfrac w2,\tfrac w2)^2\times(0,l).
\end{equation*}
The elementary cell is a full block of material with five cone-like holes arranged as in \cref{fig:elementaryCellLargeForce} left:
a big central one pointing upwards and four identical downward pointing ones of half the diameter, centred within each quarter of the elementary cell.
Their bases are such that four elementary cells of half the width can be attached on top of $\omega$ such that the cones seamlessly fit together.
In more detail, abbreviate the solid of revolution around the $x_3$-axis with radius function $f$ by
\begin{equation*}
K[f]=\{x\in\R^2\times(0,l)\,|\,x_1^2+x_2^2\leq f(x_3)^2\}.
\end{equation*}
If we denote the cone tips by
\begin{equation*}
P_0=(0\ 0\ l)^T
\qquad\text{ and }\qquad
P_1,...,P_4=(\pm\tfrac{w}{4} \ \pm\tfrac{w}{4} \ 0)^T
\end{equation*}
and the $x_3$-dependent radius of the four smaller cones by $R(x_3)$,
then the material distribution $\O$ within the elementary cell is given by
\begin{equation*}
\omega\setminus(K_0\cup\ldots\cup K_4)
\end{equation*}
for the cones
\begin{equation*}
K_0=P_0-K[2R]
\qquad\text{ and }\qquad
K_i=P_i+K[R]
\quad\text{for }i=1,\ldots,4.
\end{equation*}
Note that for symmetry reasons we chose the central cone to be identical to the four smaller ones, only flipped upside down and dilated by the factor $2$ along the first two dimensions.
We next specify the radius function $R(x_3)$.
For $\varepsilon=0$ it is known that the optimal microstructures everywhere exhibit a material volume fraction of $F$.
We aim to match this exactly on each horizontal slice of the elementary cell and thus require
\begin{equation*}
4\pi R(x_3)^2+\pi(2R(l-x_3))^2
=w^2(1-F)
=\vcentcolon 4\pi a^2.
\end{equation*}
Furthermore, the elementary cells will be imposed with a normal stress at their top and bottom boundary,
and since this normal stress cannot immediately be redirected to a nonvertical direction without producing infinite compliance at the cone tips,
we require
\begin{equation*}
R'(0)=0.
\end{equation*}
This is one of the central differences to the corresponding construction in the superconductor setting.
Together with the previous condition this also implies $R'(l)=0$.
One possible solution is to take the cross-sectional area of the cones to be the quintic polynomial fulfilling the above conditions,
\begin{equation*}
R(x_3)=a\bar R(\tfrac{x_3}{l})
\qquad\text{with }
\bar R(z)=\sqrt{10z^3-15z^4+6z^5}.
\end{equation*}
Note that the material volume of the elementary cell is $Fw^2l$ by construction.

\begin{figure}
	\includegraphics[scale=0.65, angle=90, trim = 15 0 0 0]{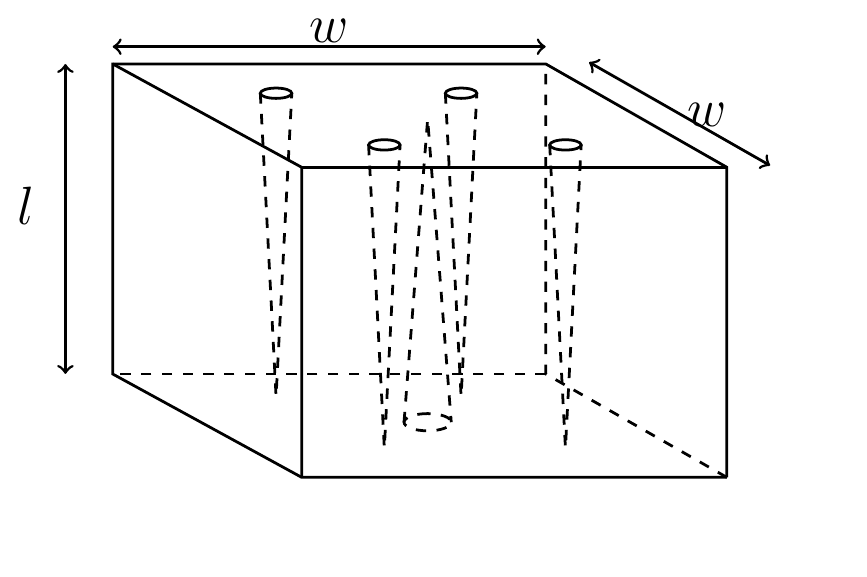} \hfill
	\includegraphics[scale=1]{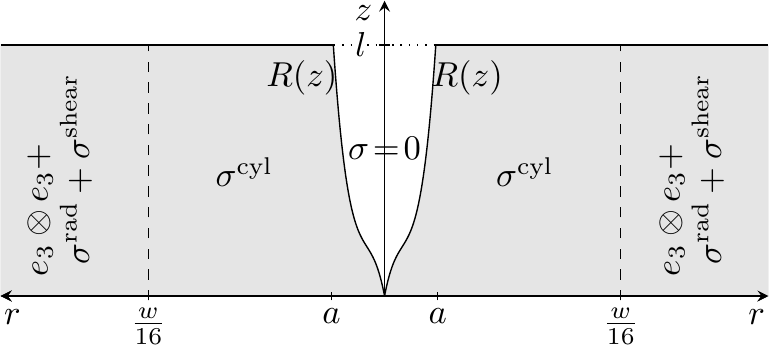}
	\caption{Sketch of the elementary cell (left) and vertical cut through one of the small cones (right).}
	\label{fig:elementaryCellLargeForce}
\end{figure}

\paragraph{Elementary cell stress field construction.}
The construction of the stress field within the elementary cell is substantially more complicated than in the superconductor setting.
We are looking for a divergence-free $\sigma:\omega\to\R^{3\times3}_\sym$ with $\sigma=0$ on $C_0\cup\ldots\cup C_4$
and unit tensile stress $\sigma n=n$ on the top and bottom material boundary $[(-\frac w2,\frac w2)^2\times\{0,l\}]\setminus(\overline{C_0}\cup\ldots\cup\overline{C_4})$.
We will first construct a radially symmetric stress field within cylinders around the cones
and in a second step construct the stress field in the complement of these cylinders.
In more detail, let $C[r]$ denote the infinite cylinder around the $x_3$-axis of radius $r$, then we define
\begin{equation*}
C_0=\omega\cap C[\tfrac w8]
\qquad\text{and}\qquad
C_i=\omega\cap(P_i+C[\tfrac w{16}])
\quad\text{for }i=1,\ldots,4
\end{equation*}
to be disjoint cylinders around the cones $K_0,\ldots,K_4$.
Since we assumed $1-F<\frac1{64}$, these cylinders contain the cones completely.
We then compose the stress field via
\begin{equation*}
\sigma(x)=\begin{cases}
\sigmain(x)&\text{if }x\in C_0\cup\ldots\cup C_4,\\
e_3\otimes e_3+\sigmarad(x)+\sigmashear(x)&\text{else,}
\end{cases}
\end{equation*}
where $\sigmain$, $\sigmarad$, and $\sigmashear$ will be specified below.

The stress field in $C_0,\ldots,C_4$ is reduced to a stress field $\bar\sigma:\omega\cap C[\frac w{16}]\to\R^{3\times3}_\sym$ with $\bar\sigma=0$ on $K[R]$ and unit tensile stress at the top and bottom boundary via
\begin{equation*}
\sigmain(x)=\begin{cases}
\hfil\bar\sigma(x-P_i)&\text{if }x\in C_i\text{ for }i\in\{1,\ldots,4\},\\
\left(\begin{smallmatrix} 4\bar\sigma_{11} & 4\bar\sigma_{12} & -2\bar\sigma_{13} \\ 4\bar\sigma_{12} & 4\bar\sigma_{22} & -2\bar\sigma_{23} \\ -2\bar\sigma_{13} & -2\bar\sigma_{23} & \bar\sigma_{33} \end{smallmatrix}\right)
(\tfrac{x_1}2,\tfrac{x_2}2,l-x_3)&\text{if }x\in C_0.
\end{cases}
\end{equation*}
By construction, $\sigmain$ is divergence-free if $\bar\sigma$ is.
We take $\bar{\sigma}$ radially symmetric, where the simple intuition behind our ansatz is
that the material $(C[\frac w{16}]\cap\omega)\setminus K[R]$ is decomposed into infinitely many infinitesimally thin chalice-like shells
that are each in equilibrium on their own and do not interact.
We parameterize the family of chalices by their radius $s$ at $x_3=0$ and denote their radius at height $x_3$ by $R_s(x_3)$ (thus $R_0=R$).
The vertical stress component in each of these chalices shall be $1$ (which is known to be preferred from the case $\varepsilon=0$).
The balance of forces acting in vertical direction on the union of chalices for $s\in[0,S]$ with $S\in[0,\frac{w}{16}]$ arbitrary then requires
\begin{equation*}
\pi\left(R_S(x_3)^2-R(x_3)^2\right)=\pi\left(R_S(0)^2-R(0)^2\right)=\pi S^2
\qquad\text{for all }x_3\in(0,l)
\end{equation*}
so that
\begin{equation*}
R_s(x_3)=\sqrt{s^2+R(x_3)^2}.
\end{equation*}
Using cylindrical coordinates $(r,\varphi,z)$ and denoting the radial, circumferential and vertical unit vectors by $e_r$, $e_{\varphi}$ and $e_z$, respectively,
the tangent plane to a chalice is spanned by $t(r,\varphi,z)=R_{s(r,z)}'(z) \, e_r+e_z$ and $e_{\varphi}$,
where $s(r,z)=\sqrt{r^2-R(z)^2}$ is the parameter of the chalice containing the circle of radius $r$ at height $z$.
Since the chalices do not interact, they can only have in-plane stress, and due to radial symmetry all shear stresses are zero so that $\hat{\sigma}$ must be of the form
\begin{equation*}
\bar{\sigma}(r,\varphi,z)=\alpha(r,z) \, t(r,\varphi,z)\otimes t(r,\varphi,z)+\beta(r,z) \, e_{\varphi}\otimes e_{\varphi}.
\end{equation*}
So as to have vertical stress component 1 we must have $\alpha=1$, thus
\begin{equation*}
\bar\sigma(r,\varphi,z)
=\left(\begin{matrix}\bar\sigma_{rr} & \bar\sigma_{r\varphi} & \bar\sigma_{rz} \\
\bar\sigma_{r\varphi} & \bar\sigma_{\varphi\varphi} & \bar\sigma_{\varphi z} \\
\bar\sigma_{rz} & \bar\sigma_{\varphi z} & \bar\sigma_{zz} \end{matrix}\right)(r,\varphi,z)
=\left(\begin{smallmatrix}(R_{s(r,z)}')^2 & 0 & R_{s(r,z)}' \\ 0 & \beta(r,z) &  0 \\ R_{s(r,z)}' & 0 & 1 \end{smallmatrix}\right)
\end{equation*}
in cylindrical coordinates and the basis $e_r,e_\varphi,e_z$.
Note that
\begin{equation*}
R_{s(r,z)}'(z)=\frac{R(z) \, R'(z)}{\sqrt{s(r,z)^2+R(z)^2}}=\frac{R(z) \, R'(z)}{r}.
\end{equation*}
The condition for $\bar\sigma$ to be divergence-free in cylindrical coordinates reads
\begin{align*}
\frac{1}{r}\frac{\partial}{\partial r}(r \, \bar\sigma_{rr})+\frac{1}{r}\frac{\partial\bar\sigma_{r\varphi}}{\partial\varphi}+\frac{\partial\bar\sigma_{rz}}{\partial z}-\frac{\bar\sigma_{\varphi\varphi}}{r}&=0,\\
\frac{1}{r}\frac{\partial}{\partial r}(r \, \bar\sigma_{r\varphi})+\frac{1}{r}\frac{\partial\bar\sigma_{\varphi\varphi}}{\partial\varphi}+\frac{\partial\bar\sigma_{\varphi z}}{\partial z}+\frac{\bar\sigma_{r\varphi}}{r}&=0, \\
\frac{1}{r}\frac{\partial}{\partial r}(r \, \bar\sigma_{rz})+\frac{1}{r} \, \frac{\partial\bar\sigma_{\varphi z}}{\partial\varphi}+\frac{\partial\bar\sigma_{zz}}{\partial z}&=0,
\end{align*}
from which we derive
\begin{equation*}
\beta(r,z)
=\frac{\partial}{\partial r}(r \, \bar\sigma_{rr})+r\frac{\partial\bar\sigma_{rz}}{\partial z}
=-\frac{(R(z) \, R'(z))^2}{r^2}+(R(z) \, R'(z))'
=\frac{(R(z)^2)''}2-\frac{(R(z)\,R'(z))^2}{r^2}.
\end{equation*}
By construction $\bar\sigma$ is divergence-free, zero inside the cones, and it satisfies the correct boundary condition at the top and bottom.

\begin{remark}[Improved stress field]
By letting the chalices interact via a normal stress component one could reduce the excess compliance further,
but it turns out that this only leads to a reduction by a constant factor, leaving the scaling the same.
%
%
\end{remark}

In the complement of the five cylinders we employ a vertical unit tensile stress plus two stress fields $\sigmarad,\sigmashear:\omega\setminus(C_0\cup\ldots\cup C_4)\to\R^{3\times3}_\sym$
that compensate the boundary stress induced by $\sigmain$ on $\partial C_0\cup\ldots\cup\partial C_4$.
In more detail, $\sigmarad$ will be a pure inplane stress-field in each horizontal cross-section, accommodating the normal stress on the cylinder walls imposed by $\sigmain$,
while $\sigmashear$ will be a stress field accommodating the shear stress on the cylinder walls.
The stress $\sigmarad$ will be radially symmetric around each cylinder and will itself be contained in a slightly larger cylinder
of radius $T=\frac{3w}{32}$ at the four outer cones and radius $2T$ at the inner one (so that those larger cylinders do not yet overlap).
It is reduced to a stress field $\tilde\sigma:\omega\cap(C[T]\setminus C[\frac w{16}])\to\R^{3\times3}_\sym$ via
\begin{equation*}
\sigmarad(x)=\begin{cases}
\hfil\tilde\sigma(x-P_i)&\text{if }x\in\omega\cap(P_i+C[T])\setminus C_i\text{ for }i\in\{1,\ldots,4\},\\
\left(\begin{smallmatrix} 4\tilde\sigma_{11} & 4\tilde\sigma_{12} & -2\tilde\sigma_{13} \\ 4\tilde\sigma_{12} & 4\tilde\sigma_{22} & -2\tilde\sigma_{23} \\ -2\tilde\sigma_{13} & -2\tilde\sigma_{23} & \tilde\sigma_{33} \end{smallmatrix}\right)
(\tfrac{x_1}2,\tfrac{x_2}2,l-x_3)&\text{if }x\in\omega\cap C[2T]\setminus C_0.
\end{cases}
\end{equation*}
As already suggested, $\tilde\sigma$ shall be of the form
\begin{equation*}
\tilde\sigma(r,\varphi,z)=\left(\begin{matrix}\tilde\sigma_{rr} & 0 & 0 \\ 0 & \tilde\sigma_{\varphi\varphi} & 0 \\ 0 & 0 & 0 \end{matrix}\right) (r,\varphi,z)
\end{equation*}
in cylindrical coordinates and the basis $e_r,e_\varphi,e_z$.
It has to satisfy $\div \, \tilde\sigma=0$ as well as the boundary conditions
\begin{equation*}
\tilde\sigma_{rr}(T,\varphi,z)=0
\qquad\text{ and }\qquad
\tilde\sigma_{rr}(\tfrac{w}{16},\varphi,z)=\bar\sigma_{rr}(\tfrac{w}{16},\varphi,z)=\left(\frac{R(z) \, R'(z)}{w/16}\right)^2=\vcentcolon  g_1(z)
\end{equation*}
at the outer and inner cylinder boundary (recall that it compensates the normal stress of $\sigmain$ at the inner boundary, which is given by $\bar\sigma_{rr}$).
A possible choice of $\tilde\sigma$ is
\begin{equation*}
\tilde\sigma_{rr}(r,\varphi,z)
=\kappa(z)(\tfrac1{T^2}-\tfrac1{r^2})
\qquad\text{and}\qquad
\tilde\sigma_{\varphi\varphi}(r,\varphi,z)
=\kappa(z)(\tfrac1{T^2}+\tfrac1{r^2})
\end{equation*}
for
\begin{equation*}
\kappa(z)
=\frac{g_1(z)}{\frac{1}{T^2}-\frac{1}{(w/16)^2}}
=-\tfrac{9}{5}(R(z)R'(z))^2
\end{equation*}
(this choice is actually the equilibrium stress under the above conditions, which is a straightforward exercise to derive).
\notinclude{
The optimal one can e.\,g.\ be computed as $\text{D}u$ for a 2D radial displacement $u(\vec{r})=f(r) \, \vec{r}$ (with $r=|\vec{r}|$) satisfying $\Delta u=0$: Writing

\begin{equation*}
\text{D}u(\vec{r})=f(r) \, \mathds{1}+f'(r) \, \frac{\vec{r}\otimes\vec{r}}{r}=\frac{\text{D}u+\text{D}u^T}{2}
\end{equation*}
\\
the solution has to fulfil

\begin{equation*}
\div\left(\frac{\text{D}u+\text{D}u^T}{2}\right)=\Delta u=\left(f''(r)+\frac{3f'(r)}{r}\right) \, \vec{r}=0
\end{equation*}
\\
from which we get

\begin{equation*}
f(r)=\kappa_1+\frac{\kappa_2}{r^2}  
\end{equation*}
\\
and consequently find

\begin{equation*}
\hat{\sigma}_2=\frac{\text{D}u+\text{D}u^T}{2}=\text{D}u=\kappa_1 \, \mathds{1}+\frac{\kappa_2}{r^2} \, \left(\mathds{1}-2 \, \frac{\vec{r}}{r}\otimes\frac{\vec{r}}{r}\right) \ \ \ .
\end{equation*}
\\
Taking care of the Neumann boundary conditions $\hat{\sigma}_2(R,\varphi,z)\cdot \vec{r}=0$ and $\hat{\sigma}_2(\frac{w}{16},\varphi,z)\cdot\vec{r}=g_1(z)$ we further get the coefficients

\begin{equation*}
\kappa_1=\frac{\kappa_2}{R^2} \ \ , \ \ \  \kappa_2=\frac{g_1(z)}{\frac{1}{R^2}-\frac{1}{(w/16)^2}}=-\frac{9}{1280}w^2g_1(z)
\end{equation*}
\\
which result in

\begin{equation*}
\hat{\sigma}_{2,rr}=\frac{9}{1280}w^2g_1(z)\left(\frac{1}{r^2}-\frac{1}{R^2}\right) \ , \ \ \hat{\sigma}_{2,\varphi\varphi}=-\frac{9}{1280}w^2g_1(z)\left(\frac{1}{r^2}+\frac{1}{R^2}\right)
\end{equation*}
\\
for the non-vanishing entries of $\hat{\sigma}_2$. \\
}

The stress $\sigmashear$ finally is chosen as follows:
For fixed height $z$ it has to accomodate the (vertical) shear stress on the cylinder walls $\partial C_i$ induced by $\sigmain$,
which equals $\bar\sigma_{rz}(\frac w{16},\varphi,z)=\frac{16}wR(z)R'(z)$ on $\partial C_1,\ldots,\partial C_4$,
while it is given by $-2\bar\sigma_{rz}(\frac w{16},\varphi,l-z)=-2\frac{16}wR(l-z)R'(l-z)=-\frac{32}wR(z)R'(z)$ on $\partial C_0$.
Thus, for $g_2(z)=\frac{R(z)R'(z)}{w/16}$ we require
\begin{alignat*}{2}
\sigmashear n
&=2g_2(x_3) \, e_3 &\text{ on }&\partial C_0,\\
\sigmashear n
&=-g_2(x_3) \, e_3 &\text{ on }&\partial C_i, \, i=1,...,4
\end{alignat*}
with $n$ the unit outward normal to $\omega\setminus(C_0\cup\ldots\cup C_4)$.

\begin{remark}[Inplane equilibrium of shear stresses]
Since $\partial C_0$ has exactly twice the circumference of $\partial C_1\cup\ldots\cup\partial C_4$,
the shear forces at cross-section $z$ are in equilibrium
(which essentially is a consequence of our choice of $R$).
Therefore, $\sigmashear_{33}$ will be zero.
This not only simplifies the construction of $\sigmashear$ but is also important for the energy scaling:
a nonzero $\sigmashear_{33}$ would yield a large compliance in combination with the overall vertical unit tensile stress $e_z\otimes e_z$.
Such complications only arise due to the symmetry condition on the stress field and are absent in the superconductor setting.
\end{remark}

\begin{figure}
	\includegraphics[scale=0.29]{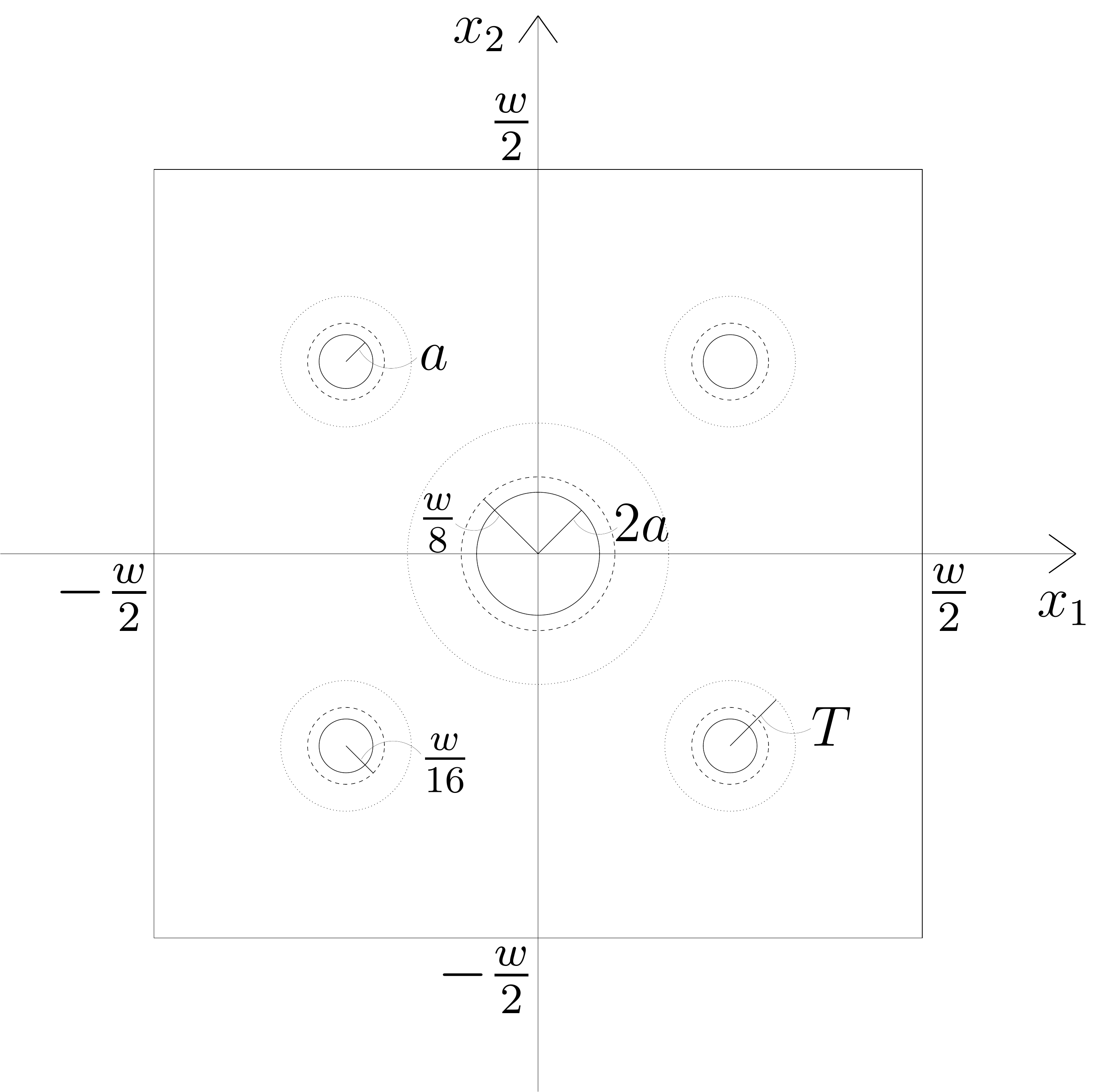}\hfill
	\includegraphics[scale=0.29]{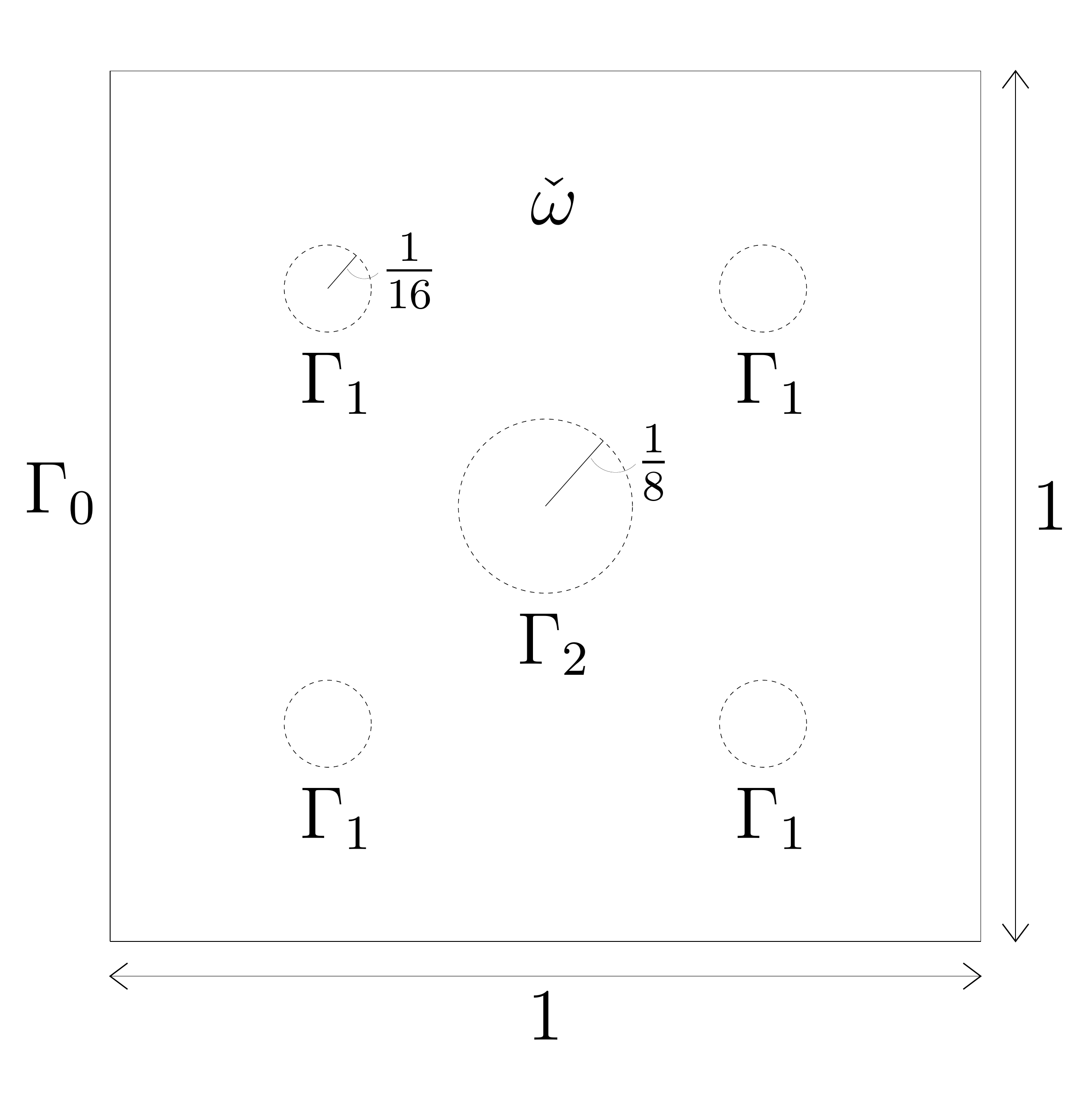}
	\caption{Top view of the elementary cell, original (left) and rescaled version (right).}
	\label{fig:crossSections}
\end{figure} 

Consider now a cross-section of $\omega\setminus(C_0\cup\ldots\cup C_4)$ and rescale it by the factor $\frac1w$ to have sidelength $1$.
Let us denote that rescaled cross-section by $\check\omega$ and its boundary components by $\Gamma_0,\Gamma_1,\Gamma_2$ as indicated in \cref{fig:crossSections}.
Now let $u:\check\omega\to\R$ satisfy
\begin{alignat*}{2}
\Delta u
&=0&\text{ in }&\check\omega, \\
\partial_nu
&=0&\text{ on }&\Gamma_0,\\
\partial_nu
&=-1&\text{ on }&\Gamma_1, \\
\partial_nu
&=2&\text{ on }&\Gamma_2,
\end{alignat*}
where $\partial_nu$ denotes the normal derivative of $u$ (a solution exists due to $\int_{\partial\check\omega}\partial_nu \, \d x=0$; this is related to the previous remark).
Furthermore, let the two-dimensional stress field $\check\sigma: \check\omega\rightarrow\R^{2\times 2}_\sym$ satisfy
\begin{alignat*}{2}
\div \, \check\sigma
&=\nabla u&\text{ in }&\check\omega, \\ 
\check\sigma\cdot n
&=0&\text{ on }&\partial\check\omega
\end{alignat*}
(such a $\check\sigma$ exists due to $\int_{\check\omega}\div \, \check\sigma \, \d x=\int_{\check\omega}\nabla u \, \d x=0=\int_{\partial\check\omega}\check\sigma\cdot n \, \d x$,
where the second equality holds by symmetry of $u$)
and define $\sigmasheartilde:\check\omega\times(0,1)\to\R^{3\times 3}_\sym$ as
\begin{equation*}
\sigmasheartilde(x)=\left(\begin{matrix}-g_3'(x_3) \, \check\sigma(x_1,x_2) & g_3(x_3) \, \nabla u(x_1,x_2) \\ g_3(x_3) \, \nabla u(x_1,x_2)^T & 0 \end{matrix}\right)
\end{equation*} 
for $g_3(z)=z^2(1-z)^2$.
By construction we have $\div \, \sigmasheartilde=0$ and
\begin{equation*}
\sigmasheartilde(x)n=\begin{cases}
\hfil0&\text{on }[\Gamma_0\times(0,1)]\cup[\check\omega\times\{0,1\}],\\
-g_3(x_3)e_3 & \text{on }\Gamma_1\times(0,1),\\
\hfil 2g_3(x_3)e_3 & \text{on }\Gamma_2\times(0,1)
\end{cases}
\end{equation*}
for $n$ the unit outward normal.
Thus, noting that $g_2(z)=240\frac{a^2}{lw} \, g_3\left(\frac{z}{l}\right)$, the divergence-free stress field
\begin{equation*}
\sigmashear(x)=240\left(\frac{a}{w}\right)^2\left(\begin{matrix}(\frac{w}{l})^2\sigmasheartilde_{11}& (\frac{w}{l})^2\sigmasheartilde_{12} & (\frac{w}{l})\sigmasheartilde_{13} \\ (\frac{w}{l})^2\sigmasheartilde_{12} & (\frac{w}{l})^2\sigmasheartilde_{22} & (\frac{w}{l})\sigmasheartilde_{23} \\ (\frac{w}{l})\sigmasheartilde_{13} & (\frac{w}{l})\sigmasheartilde_{23} & \sigmasheartilde_{33} \end{matrix}\right)\left(\frac{x_1}{w},\frac{x_2}{w},\frac{x_3}{l}\right)
\end{equation*}
satisfies all required boundary conditions.

\paragraph{Elementary cell excess cost.}
As usual we denote the elementary cell contribution to compliance, volume, and perimeter by $\comp_\cell$, $\vol_\cell$, and $\per_\cell$, respectively.
Furthermore we again evenly distribute $\J^{*,F,\ell}_0=2F\ell^2$ over all of $\Omega$ so that its proportion within an elementary cell is $2Fw^2l$.
Recalling now $\vol_\cell=Fw^2l$, the excess cost contribution of the elementary cell is
\begin{multline*}
\Delta\J_\cell
=\comp_\cell+\vol_\cell+\varepsilon\per_\cell-2Fw^2l
=\comp_\cell-\vol_\cell+\varepsilon\per_\cell\\
=\int_{\omega\setminus\bigcup_{i=0}^4K_i}|\sigma|^2-1\,\d x+\varepsilon\per_\cell
=\int_{\omega\setminus\bigcup_{i=0}^4C_i}|\sigma|^2-1\,\d x
+\sum_{i=0}^4\int_{C_i\setminus K_i}|\sigma|^2-1\,\d x
+\varepsilon\per_\cell.
\end{multline*}
For $i\in\{1,\ldots,4\}$ we now estimate (for better readability we drop the argument $z$ of $R$ and $\bar R$)
\begin{align*}
\int_{C_i\setminus K_i}|\sigma|^2-1\,\d x
&=\int_{\omega\cap C[\frac w{16}]\setminus K[R]}|\bar\sigma|^2-1\,\d x\\
&=\int_{0}^{l}\int_{R}^{w/16}r\left(2\left(\frac{R\, R'}{r}\right)^2+\left(\frac{R \, R'}{r}\right)^4+\beta(r,z)^2\right) \, \d r \, \d z \\
&\lesssim\int_{0}^{l}\int_{R}^{w/16}r\left(2\left(\frac{R \, R'}{r}\right)^2+\left(\frac{R \, R'}{r}\right)^4+\left(\left(R^2\right)''\right)^2\right) \, \d r \, \d z \\
&=\int_{0}^{1}\int_{\bar R}^{w/(16a)}\frac{2a^4}{lr}\left(\bar R\,\bar R'\right)^2+\frac{a^6}{l^3r^3}\left(\bar R\,\bar R'\right)^4+\frac{ra^6}{l^3}\left(\left(\bar R^2\right)''\right)^2 \, \d r \, \d z \\
&=\int_{0}^{1}\frac{2a^4}{l}\log\left(\frac{w}{16a\bar R}\right)\left(\bar R\,\bar R'\right)^2+\frac{a^6}{2l^3}\left(\frac{1}{\bar R^2}-\frac{256a^2}{w^2}\right)\left(\bar R\,\bar R'\right)^4\\
&\hspace{30ex}+\frac{a^6}{2l^3}\left(\frac{w^2}{256a^2}-\bar R^2\right)\left(\left(\bar R^2\right)''\right)^2 \, \d z \\
&\lesssim\frac{a^4}{l}\log\frac{w}{a}\int_{0}^{1}\left(\bar R\,\bar R'\right)^2\,\d z+\frac{a^4}{l}\int_{0}^{1}|\log(16\bar R)|\left(\bar R\,\bar R'\right)^2\,\d z\\
&\hspace{20ex}+\frac{a^6}{l^3}\int_0^1\frac{\left(\bar R\,\bar R'\right)^4}{\bar R^2}\,\d z+\frac{a^4w^2}{l^3}\int_0^1\left(\left(\bar R^2\right)''\right)^2 \, \d z\\
&\lesssim\frac{a^4}{l}\log\frac{w}{a}+\frac{a^4w^2}{l^3},
\end{align*}
where in the last step we used the boundedness of all integrals as well as $a<\frac w2$.
The analogous computation for $i=0$ yields the same scaling.
Similarly,
\begin{align*}
\int_{\omega\setminus\bigcup_{i=0}^4C_i}|\sigma|^2-1\,\d x
&=\int_{\omega\setminus\bigcup_{i=0}^4C_i}\big| e_z\otimes e_z+\sigmarad+\sigmashear\big|^2-1 \, \d x\\
&=\int_{\omega\setminus\bigcup_{i=0}^{4}C_i}\big|\sigmarad+\sigmashear\big|^2 \, \d x \\
&\lesssim\int_{\omega\setminus\bigcup_{i=0}^{4}C_i}\big|\sigmarad\big|^2 \, \d x+\int_{\omega\setminus\bigcup_{i=0}^{4}C_i}\big|\sigmashear\big|^2 \, \d x.
\end{align*}
We estimate both integrals separately. We have
\begin{align*}
\int_{\omega \setminus \bigcup_{i=0}^{4}C_i}\!\!\big|\sigmarad\big|^2 \, \d x
&=4\int_0^l\!\!\int_{\frac w{16}}^T\!\!r\!\left(\tilde\sigma_{rr}^2\!+\!\tilde\sigma_{\varphi\varphi}^2\right)\d r\,\d z
\!+\!\!\!\int_0^l\!\!\int_{\frac w8}^{2T}\!\!r\!\left((4\tilde\sigma_{rr}(\tfrac r2,\varphi,z))^2\!+\!(4\tilde\sigma_{\varphi\varphi}(\tfrac r2,\varphi,z))^2\right)\d r\,\d z\\
&\lesssim\int_0^l\int_{\frac w{16}}^Tr\left(\tilde\sigma_{rr}^2+\tilde\sigma_{\varphi\varphi}^2\right)\,\d r\,\d z\\
&=\int_0^l\int_{\frac w{16}}^Tr\left(\frac95\right)^2(R\,R')^4\left[\left(\frac1{T^2}-\frac1{r^2}\right)^2+\left(\frac1{T^2}+\frac1{r^2}\right)^2\right]\,\d r\,\d z\\
&=\frac{81a^8}{25w^2l^3}\int_0^1\int_{\frac 1{16}}^{\frac 3{32}}r(\bar R\,\bar R')^4\left[\left(\frac{32^2}{9}-\frac1{r^2}\right)^2+\left(\frac{32^2}{9}+\frac1{r^2}\right)^2\right]\,\d r\,\d z\\
&\lesssim\frac{a^6}{l^3},
\end{align*}
where in the last step we used the boundedness of the integral and that $a<\frac w2$.
Also,
\begin{align*}
\int_{\omega\setminus\bigcup_{i=0}^{4}C_i}\big|\sigmashear\big|^2 \, \d x
&=lw^2\int_0^1\!\!\int_{\check\omega}240^2\left(\frac{a}{w}\right)^4\left(\left(\frac{w}{l}\right)^4g_3'(x_3)^2\big|\check\sigma\big|^2+2\left(\frac{w}{l}\right)^2g_3(x_3)^2 \, \big|\nabla u\big|^2\right) \, \d x \\
&\lesssim\frac{a^4w^2}{l^3}\int_0^1\!\!\int_{\check\omega}g_3'(x_3)^2\big|\check\sigma(x_1,x_2)\big|^2\,\d x+\frac{a^4}{l}\int_0^1\!\!\int_{\check\omega}g_3(x_3)^2 \, \big|\nabla u(x_1,x_2)\big|^2 \, \d x \\
&\lesssim \frac{a^4w^2}{l^3}+\frac{a^4}{l},
\end{align*}
where again we used boundedness of the integrals.
Finally, the surface area of all material-void interfaces within an elementary cell can be estimated as
\begin{equation*}
\per_\cell=\hd^2(K_0)+\ldots+\hd^2(K_4)\lesssim al
\end{equation*}
so that the total excess cost contribution can be summarized as
\begin{equation*}
\Delta\J_\cell
\lesssim\frac{a^4}{l}\log\frac{w}{a}+\frac{a^4w^2}{l^3}+\varepsilon al
\lesssim\underbrace{(1-F)^2 \, \frac{w^4}{l} \, \big|\log(1-F)\big|}_{=\vcentcolon\Delta \J_1(l)}+\underbrace{(1-F)^2 \, \frac{w^6}{l^3}}_{=\vcentcolon\Delta \J_2(l)}+\underbrace{\varepsilon wl\sqrt{1-F}}_{=\vcentcolon\Delta \J_3(l)}.
\end{equation*}
Now $\Delta\J_1(l)+\Delta\J_3(l)$ is minimized by
\begin{equation*}
l_1=(1-F)^{\frac{3}{4}}\big|\log(1-F)\big|^{\frac12}w^{\frac32}\varepsilon^{-\frac12},
\end{equation*}
while $\Delta\J_2(l)+\Delta\J_3(l)$ is minimized (up to a constant factor) by
\begin{equation*}
l_2=(1-F)^{\frac{3}{8}}w^{\frac54}\varepsilon^{-\frac{1}{4}}.
\end{equation*}
Consequently, since both $\Delta\J_1$ and $\Delta\J_2$ are decreasing in $l$, the optimal elementary cell height for fixed width $w$ is given by
\begin{equation*}
l_{\opt}(w)=\max\big\{l_1,l_2\big\}.
\end{equation*}
Thus we obtain
\begin{align*}
\Delta\J_{\cell}
&\lesssim\Delta\J_1(l_{\opt})+\Delta\J_2(l_{\opt})+\Delta\J_3(l_{\opt}) \\
&\lesssim\Delta\J_1(l_1)+\Delta\J_2(l_2)+\max\big\{\Delta\J_3(l_1),\Delta\J_3(l_2)\big\} \\
&\lesssim\max \left\{(1-F)^{\frac{5}{4}}|\log(1-F)|^{\frac12}\varepsilon^{\frac{1}{2}}w^{\frac{5}{2}}, \, (1-F)^{\frac{7}{8}}\varepsilon^{\frac{3}{4}}w^{\frac{9}{4}}\right\}.
\end{align*}

\paragraph{Full construction.}
To satisfy the boundary condition $\sigma n=Fe_3$ at the top and bottom boundary of $\Omega$,
we stack infinitely many layers of elementary cells on top of each other, that is, we choose $n=\infty$.
Let us now identify the with $w_1$ of the coarsest elementary cells.
The total height of the construction has to equal $1$, hence
\begin{equation*}
1
=2\sum_{k=0}^{\infty}l_k
=2\sum_{k=0}^{\infty}l_\opt(w_k)
\sim l_{\opt}(w_1)
\end{equation*}
using the fact that $l_\opt(w_k)\in l_\opt(w_1)[2^{\frac{(1-k)3}2},2^{\frac{(1-k)5}4}]$ and the geometric series.
We deduce
\begin{equation*}
w_1\sim\min\left\{(1-F)^{-\frac{1}{2}}|\log(1-F)|^{-\frac{1}{3}}\varepsilon^{\frac{1}{3}}, \, (1-F)^{-\frac{3}{10}}\varepsilon^{\frac{1}{5}}\right\}.
\end{equation*}
Due to our assumption $\varepsilon^{\frac{2}{3}}\lesssim(1-F)|\log(1-F)|^{-\frac{1}{3}}$ in the regime of large forces, the minimum is always the first of both terms.
As in the previous sections, we in fact choose $w_1$ to be the width closest to $(1-F)^{-\frac{1}{2}}|\log(1-F)|^{-\frac{1}{3}}\varepsilon^{\frac{1}{3}}$
such that the resulting total construction height does not exceed $1$ and $\ell$ is an integer multiple of $w_1$,
for which we require $\ell\geq(1-F)^{-\frac{1}{2}}|\log(1-F)|^{-\frac{1}{3}}\varepsilon^{\frac{1}{3}}$.
As before, the total height $1$ is then restored by slightly increasing the height $l_1$ of the coarsest elementary cells.

Finally, abbreviating the excess cost contribution of the cells in layer $i$ (of which there are $(\ell/w_i)^2$) by $\Delta\J_\cell^i$,
the total excess cost is estimated (exploiting the geometric series) as
\begin{multline*}
\Delta\J
=2\sum_{i=0}^{\infty}\left(\tfrac{\ell}{w_i}\right)^2\Delta\J_{\cell}^i
\sim\left(\tfrac{\ell}{w_{1}}\right)^2\Delta\J_{\cell}^1
\lesssim\ell^2(1-F)^{\frac{5}{4}}|\log(1-F)|^{\frac12} \varepsilon^{\frac{1}{2}}w_{1}^{\frac{1}{2}}\\
\sim\ell^2(1-F)|\log(1-F)|^{\frac{1}{3}}\varepsilon^{\frac{2}{3}}.
\end{multline*}

\subsection{Extremely large force.}
The construction for this regime is trivial, we take $\O=\Omega$ to be a full block of material without any holes. The equilibrium stress in this case can readily be checked to be $\sigma=Fe_3\otimes e_3$, so
\begin{equation*}
\Delta\J
=\comp^{F,\ell}(\Omega)+\vol(\Omega)+\varepsilon\per_\Omega(\Omega)-2F\ell^2
=F^2\ell^2+\ell^2+0-2F\ell^2
=(1-F)^2\ell^2.
\end{equation*}

\section{Acknowledgements}
This work was supported by the Deutsche Forschungsgemeinschaft (DFG, German Research Foundation) under the priority program SPP 2256, grant WI 4654/2-1, and under Germany's Excellence Strategy EXC 2044 -- 390685587, Mathematics M\"unster: Dynamics--Geometry--Structure. It was further supported by the Alfried Krupp Prize for Young University Teachers awarded by the Alfried Krupp von Bohlen und Halbach-Stiftung.

\bibliographystyle{abbrv}
\bibliography{references}

\end{document}